\RequirePackage{fix-cm}
\documentclass[sn-mathphys]{sn-jnl}

\usepackage[cal=boondox]{mathalfa}
\usepackage{mathtools}
\usepackage{bbm}
\usepackage{enumitem}
\numberwithin{equation}{section}
\jyear{2021}%

\theoremstyle{thmstyleone}%
\newtheorem{theorem}{Theorem}[section]
\theoremstyle{thmstyletwo}%

\theoremstyle{thmstylethree}%
\newtheorem{definition}{Definition}[section]%
\newtheorem{assumption}{Assumption}[section]%

\newtheorem{lemma}{Lemma}[section]
\newtheorem{example}{Example}[section]%
\newtheorem{remark}{Remark}[section]%

\begin{document}

\title[Mixed-dimensional poromechanical models]{Mixed-dimensional poromechanical models of fractured porous media}


\author[1]{\fnm{W. M.} \sur{Boon}}\email{wietse@kth.se}

\author[2]{\fnm{J. M.} \sur{Nordbotten}}\email{jan.nordbotten@uib.no}

\affil[1]{\orgdiv{Department of Mathematics}, \orgname{KTH Royal Institute of Technology}, \orgaddress{\street{Lindstedtsvägen 25}, \postcode{11428}, \state{Stockholm}, \country{Sweden}}}

\affil[2]{\orgdiv{Department of Mathematics}, \orgname{University of Bergen}, \orgaddress{\state{Bergen}, \country{Norway}}}


\abstract{We combine classical continuum mechanics with the recently developed calculus for mixed-dimensional problems to obtain governing equations for flow in, and deformation of, fractured materials. We present models both in the context of finite and infinitesimal strain, and discuss non-linear (and non-differentiable) constitutive laws such as friction models and contact mechanics in the fracture. Using the theory of well-posedness for evolutionary equations with maximal monotone operators, we show well-posedness of the model in the case of infinitesimal strain and under certain assumptions on the model parameters.}

\maketitle

\section{Introduction}
\label{sec:introduction}

The general topic of fractured porous media is of great importance in applications from biomedicine, to industrial materials, to subsurface geophysics. Its successful mathematical treatment requires a combination of classical elasticity \cite{hughes1983mathematical} with  contact mechanics  \cite{kikuchi1988contact}, and poromechanics  \cite{coussy2005poromechanics} all of which must be extended to allow for complex geometric descriptions. Much progress has been made recently on the understanding of fluid flow in fractured porous media utilizing the conceptual framework of mixed-dimensional geometries  \cite{boon2021functional, martin2005modeling}, which allows for lower-dimensional representations of fractures and their intersections. 

 Despite the importance and recent attention, the mathematical modeling and analysis of flow and deformation in fractured porous media is still far behind the needs of numerical analysts and practitioners. As a response to this, the current paper has two main aims: First, to provide the first consistent and frame-invariant mathematical model for fractured porous media on mixed-dimensional geometries. Second, to provide a well-posedness theory covering a broad class of problems of relevance to applications. 

\subsection{ Introduction to modeling and analysis of fractured porous media}

A realistic model of flow in fractured porous media necessarily requires a mathematical description of both the fluid flow as well as the mechanical response. This combination includes important nonlinearities stemming from the finite strain theory itself, combined with the contact-mechanical problem in the fracture, and finally the non-linear dependence of fluid flow on the fracture aperture.  These nonlinearities appear in the context of a problem that essentially has a saddle-point structure due to the coupling of flow and deformation. To date, and to the best of the knowledge of the authors, this problem has not been successfully analyzed in full. In this contribution, we will exploit recent developments in the form of mixed-dimensional calculus, together with abstract results from the theory of nonlinear monotone evolutionary equations, to provide a frame-invariant and self-consistent theory that, together with well-posedness results for poromechanics of fractured media, extends well beyond existing analysis. 

The present work needs to be seen in context of three independent developments over the last two decades. Firstly, in terms of modeling, it has been recognized since the work of Martin et al.  \cite{martin2005modeling}, that fluid flow in fractured materials can be successfully modeled using a co-dimension one representation of the fracture. We will refer to such models, where the underlying geometry is composed of domains with different topological dimension, as \textit{mixed-dimensional}.  By now, mixed-dimensional models for fluid flow in fractured media are well established, both from the perspective of well-posedness  \cite{boon2018robust}, as well as their approximation properties relative to the underlying equidimensional problem  \cite{bukavc2017dimensional, angot2009asymptotic}. Secondly, the present authors have developed a general framework for considering mixed-dimensional models of this type, where using the language of exterior calculus and differential forms, basic concepts of mixed-dimensional functions and operators are established  \cite{boon2021functional}. This leads to a mixed-dimensional functional analysis, which has been shown to inherit many of the tools associated with standard functional analysis. Thirdly, Picard and collaborators have developed the existence theory for evolutionary equations in the setting of maximally monotone operators  \cite{picard2011partial, picard2015well}. In particular, they have shown how poroealisticity can be analyzed in this framework  \cite{mcghee2010note} and that the setting is well suited to handle nonlinearities such as arise for contact problems  \cite{trostorff2012alternative}. The combination of these three developments is the foundation that allows us to consider the poromechanical contact problem which lies at the heart of poromechanics for fractured media. However, a key missing ingredient in the above is the representation of poromechanics as a mixed-dimensional model. 

 Earlier works have considered this problem using more standard approaches. Girault et al. have considered coupled poromechanics for fracture with a mixed-dimensional formulation for flow in the sense of a lower-dimensional flow representation within the fracture  \cite{girault2019mixed}, however in their analysis they have disregarded the nonlinearities associated with changes in fracture aperture, both as it pertains to the flow problem, but also the contact mechanics. Furthermore, geometric complexity is ignored as only a single fracture is considered. A different perspective was taken by Yotov et al., who considered the problem in an equidimensional sense using Stokes’ equation for the flow in the fracture, but again considering only infinitesimal aperture changes such that contact was disregarded  \cite{ambartsumyan2019nonlinear}. Bonaldi et al show well-posedness for the case where non-linearities arise due to  multiphase flow, but consider only linearized mechanics \cite{BONALDI202140}. This expands on similar results for the single-phase case in \cite{girault2015lubrication}. Finally, we mention also the work of Cusini et al, which address numerical method for this coupled problem \cite{cusini2021simulation}. While they consider geometric complexity, they limit their discussion to quasi-static, small-strain kinematics. This limitation, in particular, implies that only small slip-lengths are allowed in the resulting problem. None of the works discussed above considered finite strain modeling. Numerical and other modeling contributions have been summarized in two recent reviews \cite{berre2019flow,ambartsumyan2019nonlinear}, where important contributions relevant for this paper include the work of Jha and Juanes \cite{jha2014coupled}, Garipov et al \cite{garipov2016discrete}, Norbeck et al \cite{norbeck2016embedded}, Berge et al \cite{berge2020finite} and Stefansson et al \cite{stefansson2021fully}. 

 With the above background in mind, we here summarize the main contributions of this paper: 
\begin{enumerate}
    \item A frame-invariant formulation of finite strain suitable for fractured media within the context of mixed-dimensional calculus, allowing for a large class of complex fracture networks, and its correspondence to classical finite strain theory. 

    \item Governing equations for finite strain poromechanics of fractured media expressed in terms of mixed-dimensional variables and operators for infinitesimal strain, while allowing for contact mechanics, frictional sliding, and lubrication theory for flow in fracture. 

    \item Well-posedness theory for a linearized strain model, in the presence of contact mechanics and friction, under certain constraints on the constitutive laws.

\end{enumerate}
\subsection{Outline}

The remainder of this paper is structured as follows. Section \ref{sec:preliminaries} introduces the fundamental definitions used in the formulation and analysis of mixed-dimensional models. We discuss the \textit{mixed-dimensional continuum assumption} that is central in handling the different length scales inherent to these models. The admissible geometry is then introduced and we keep track of the connectivity between subdomains using directed acyclic graphs (DAGs). These DAGs allow us to create function spaces containing scalar and vector-valued functions that are relevant to modeling poromechanics in fractured media. All functions are defined on smooth reference domains and we use concepts from exterior calculus to appropriately map these to physical space. 

Section \ref{sec:mixed-dimentional-strain} derives invariant strain measures for the mixed-dimensional setting. The definitions follow a ``top-down" approach in which we a strain measure is formed as the linearization of a rotationally invariant finite strain. Additional attention is given to the volumetric strain as it forms the key term that couples the flow and mechanics equations in Biot poroelasticity models.

The mixed-dimensional poromechanics model is presented in Section \ref{sec:mixed-dimentional-poromechanics}. The model consists of the physical conservation principles of mass and momentum supplemented by appropriate constitutive laws. Two models are derived, based on the finite and linearized strain measures of Section \ref{sec:mixed-dimentional-strain}, respectively. This section concludes with a discussion relating our model to classic models of (poro)elasticity. 

Section \ref{sec:well-posedness} focuses on the well-posedness analysis of our model. We introduce a set of simplifying assumptions on the consititutive laws that ensures that the system can be analyzed as an evolutionary equation. Using the assumed monotonicity of our relations, we obtain well-posedness of our model in temporally weighted spaces. To close the section, example models are presented that contain conventional choices for the constitutive laws and satisfy our assumptions.

We conclude the paper in Section \ref{sec:summary}, highlighting the necessary aspects of our model that ensure physical relevance and well-posedness. The paper is supplemented by Appendix \ref{sec:appendix}, which gives the background on evolutionary equations necessary for the well-posedness analysis. 

\section{Preliminaries: Mixed-dimensional modeling and analysis}\label{sec:preliminaries}

In this section, we make precise the problem setting, its geometry and the operators adapted to the mixed-dimensional problems. The first subsection is more general and introductory in nature providing a continuum mechanical perspective on mixed-dimensional modeling, while the remaining sections lay the mathematical foundation for the exposition that follows. 

\subsection{Problem setting and motivation}
\label{sec:probelm_setting}

Classical continuum mechanics (which we will refer to as fixed-dimensional whenever needed to avoid confusion) is the modeling tool which allows for the derivation and statement of the classical field equations  \cite{hughes1983mathematical, truesdell2004non}. In particular, it leads to the development of conservation laws and constitutive laws satisfying suitable notions of frame invariance. A key building block for continuum mechanics is the assumption that the notion of a continuum is a reasonable modeling choice. We choose to formulate this as follows (the precise statement of the continuum assumption is not essential, see e.g. the thorough discussion in  \cite{truesdell2004non}):
\begin{definition}[Fixed-dimensional continuum assumption]\label{def:2.1}
For a domain \( Y\subset\mathbb{R}^{n}\), there exists a scale of consideration \(\mathcal{l}_{0}\), such that for any quantity of interest \(\mathcal{m}\), and a \( n\)-dimensional ball \( B_{x,\mathcal{l}_{0}}^{n}\) centered on \(x\) and with radius \(\mathcal{l}_{0}\), the integral below is well-defined, and the approximation is sufficiently accurate for the applications of interest: 
\begin{equation}\label{eq:2.1}
\overline{\mathcal{m}}(x)\approx\frac{1}{\left\vert B_{x,\mathcal{l}_{0}}^{n}\right\vert }\int_{B_{x,\mathcal{l}_{0}}^{n}}^{}\mathcal{m} \mathrm{d}V.
\end{equation}
\end{definition}
In other words, our formulation of the fixed-dimensional continuum assumption states that a point evaluation of a quantity (say, porosity of a porous material), can be approximated by a (say, volume) integral of characteristic size \(\mathcal{l}_{0}\), and that this approximation is accurate enough that the precise size (and indeed shape) of the integral is immaterial. As a classical example, one notes that for porosity, it is typically taken as a modeling assumption that a scale of consideration exists, the so-called ``Representative Elementary Volume" (on the order of 10 to 100 times the mean grain size), wherein the porosity is well defined  \cite{bear1979groundwater,nordbotten2011geological}. At lower scales, the integral in \eqref{eq:2.1} will be strongly affected by the precise number of grains in the integration volume. In the continuation, we will only be interested in continuum scales, and omit the overbar on the continuum quantity. 

The classical continuum assumption is suitable for a wide range of applications, and underlies the vast majority of real-world industrial computations in applied engineering. 

\begin{figure}[!htbp]
\centering
\includegraphics[width=\textwidth]{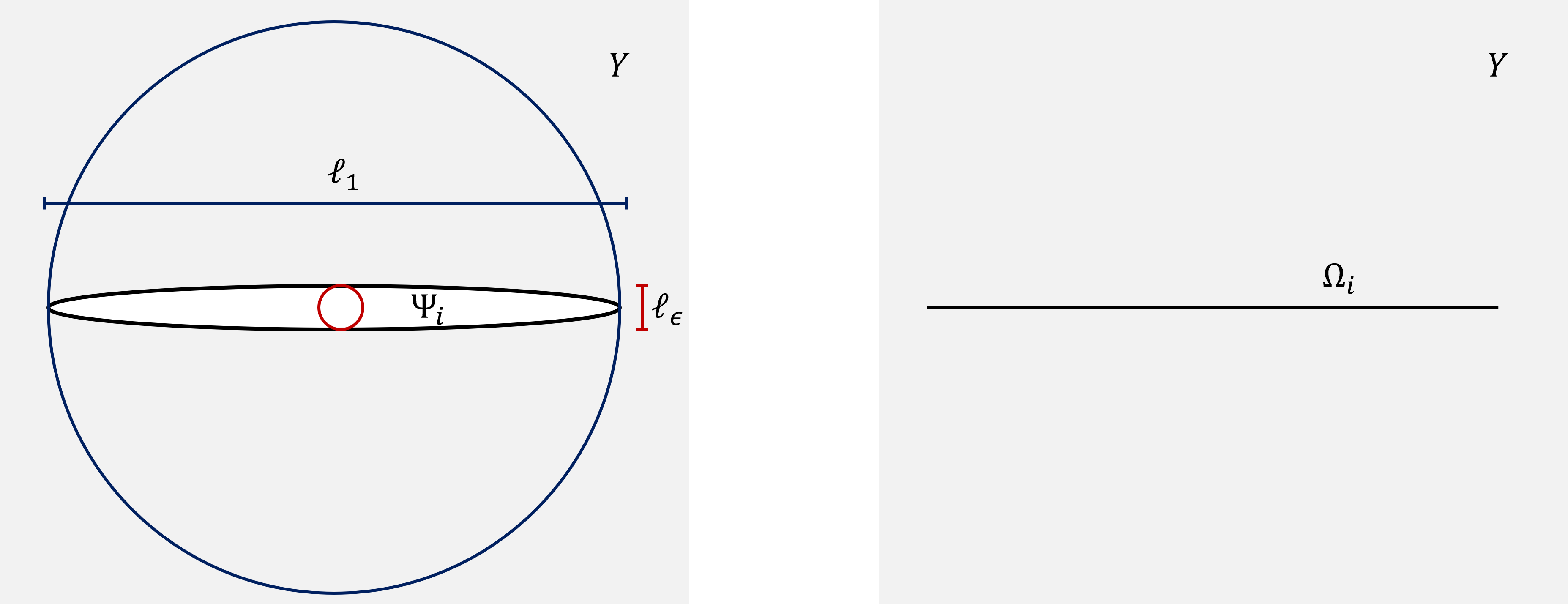}
\caption{Left: Illustration of a domain $Y$, with a high aspect-ratio inclusion $\Psi_i$, with maximal inscribed and minimum covering balls shown in red and blue, respectively. Right: Illustration of the same domain $Y$, where the high aspect-ratio inclusion is now modeled by a lower-dimensional manifold $\Omega_i$.}
\label{fig:illustration-of-domain-y}
\end{figure}

Our interest herein is in problems for which the geometry in consideration contains high-aspect-ratio inclusions \( \Psi_{i}\), indexed by $i$, that in some sense interferes with the continuum assumption. To be concrete (although the argument is more general), we consider thin fractures and their intersections. We characterize these high-aspect ratio inclusions by two length-scales, \(\mathcal{l}_{\epsilon }\) and \(\mathcal{l}_{1}\), corresponding to the diameter of the maximal inscribed and minimum covering ball, respectively (as illustrated for a single manifold in the left part of Figure \ref{fig:illustration-of-domain-y}). We are furthermore interested in the case where the small length scale violates the continuum assumption, i.e. where the following ordering holds: 
\begin{equation}
\label{eq:2.2}
\mathcal{l}_{\epsilon } \ll \mathcal{l}_{0} \ll \mathcal{l}_{1}.
\end{equation}
For such problems the standard fixed-dimensional continuum assumption cannot be applied, since the high-aspect feature (and variables within it) cannot be appropriately defined. Depending on the needs of the application at hand, we may nevertheless be inclined to consider \(\mathcal{l}_{0}\) as the appropriate modeling scale, in which case we have no other choice than to represent the thin inclusions as manifolds of lower dimension. We note that the intersection of such manifolds will have even lower dimension yet.  

The above provides the motivation for considering a \textit{mixed-dimensional continuum assumption}, wherein we are still interested in a domain \( Y\subset\mathbb{R}^{n}\), but where we allow this domain to contain a set of manifolds of topological dimension \( d<n\) (as illustrated for a single manifold in the right part of Figure \ref{fig:illustration-of-domain-y}). We formulate this extension of Definition \ref{def:2.1} as follows:  
\begin{definition}[Mixed-dimensional continuum assumption]\label{def:2.2}
Any inclusion \( \Psi_{i}\subset Y\) which satisfies \eqref{eq:2.2} can be well-represented by a \( d_{i}\)-dimensional manifold \( \Omega_{i}\). Moreover on the scale of consideration \(\mathcal{l}_{0}\),  then for any quantity of interest \(\mathcal{m}\), and a \( d_{i}\)-dimensional ball \( B_{x,\mathcal{l}_{0}}^{d_{i}}\subset \Omega_{i}\) centered on \(\mathbf{x}\in \Omega_{i}\), the integrals below are well-defined, and the approximation is sufficiently accurate for the applications of interest: 
\begin{equation}\label{eq:2.3}
\overline{\mathcal{m}}(x)\approx\frac{1}{\left\vert B_{x,\mathcal{l}_{0}}^{d_i}\right\vert }\int_{B_{x,\mathcal{l}_{0}}^{d_i}}^{}\int_{\Psi_{i}^{\perp }}^{}\mathcal{m }\mathrm{d}V_{\perp } \mathrm{d}V_{\parallel}.
\end{equation}
\end{definition}
Here we use the notation \( \Psi_{i}^{\perp }\) to indicate the cross-section of \( \Psi_{i}\) orthogonal to \( \Omega_{i}\), and we denote the measure of integration perpendicular and parallel to \( \Omega_{i}\) by \( \mathrm{d}V_{\perp }\) and \( \mathrm{d}V_{\parallel}\), respectively.

We remark that our definitions of continuum assumptions suffer from the usual weaknesses \cite{truesdell2004non}, in that they are poorly adapted to quantities near boundaries and for variables which have macroscopic discontinuities. These issues can be resolved by appealing to more technical definitions. However, as we will not deal with the issue of upscaling in the following, but merely use the result that continuum variables can be assumed to be sufficiently accurate for applications of interest, we will not elaborate these details further.

In the following, we assume that we are always within the setting of Definition \ref{def:2.2}, and proceed to make the notion of mixed-dimensional continuum variables precise, and apply the framework of continuum mechanics to derive the governing equations for the poroelastic response in fractured porous media. 

\subsection{Geometry}\label{sec:geometry}

To initialize our description of a mixed-dimensional problem, we first consider an admissible mixed-dimensional partitioning. The partitioning of Figure \ref{fig:illustration-of-domain-y} is, in a sense, too simple since it does not keep track of the boundaries between domains of different dimension. To achieve this, we herein introduce structured partitions of the domain along with the corresponding graph representations. These graphs first provide a canonical way to describe the connectivity between subdomains. Additionally, these definitions give the structure that allows us to define spaces of mixed-dimensional functions and the associated semi-discrete operators in Sections \ref{sec:mixed_dimensional} and \ref{sec:differential_operators}. For a detailed exposition of these concepts in the scalar setting, we refer to  \cite{boon2021functional}.

We will only consider problems embedded in a \( n\)-dimensional Cartesian ambient domain, and we are primarily concerned with the case \( n = 3\). Thus, let \( Y\subset\mathbb{R}^{n}\) be given, and let it be decomposed into non-overlapping, oriented  manifolds \( \Omega_{i}\) of topological dimension \( d_{i}\) such that \( Y = \cup_{i\in I}^{}\Omega_{i}\) with \( I\) the index set. In order to distinguish the domain and the partition, we will refer to the partition as \( \Omega\) without a subscript. 

We will not allow for arbitrary partitions, and therefore introduce a concept of admissible partitions. This requires some preliminaries. We first give each manifold \( \Omega_{i}\) some additional hierarchical structure  \cite{boon2021functional}. Each \( \Omega_{i}\) is \( C^{1}\)-diffeomorphic to a smooth \textit{reference domain} denoted by \( X_{i}\) and we denote the corresponding mapping by \( \phi_{0,i}:X_{i}\rightarrow \Omega_{i}\). We then endow each manifold with a directed acyclic graph, defined as follows.
\begin{definition}\label{def:2.3}
 A rooted Directed Acyclic Graph (DAG) \(\mathfrak{S}_{i}\) with \( i \in I\), is conforming to \(\Omega_{i}\)  if for all nodes \( j \in \mathfrak{S}_{i}\) :
\begin{itemize} \small
    \item There exists a root \( s_{j}\in I\) such that \( \phi_{0,j}(X_{j}) = \Omega_{s_{j}}\). Moreover, we assume \( s_{i} = i\) for each root \( i\) for convenience.

    \item For each descendant \( l \in I_j\), where \( I_{j}\) is the set containing the descendants of a node \( j\in\mathfrak{F}\), a differentiable map \( \phi_{j,l}:X_{l}\rightarrow\overline{X}_{j}\) exists with bounded derivative. We denote its range by \( \partial_{l}X_{j}\). Compound maps telescope in the sense that
\begin{equation*}
\phi_{k,l} = \phi_{k,j}\circ \phi_{j,l}
\end{equation*}
for each ancestor \( k\in\mathfrak{S}_{i}\) of \( j\).

    \item The descendants uniquely cover the parent node in the sense that
\begin{equation*}
\bigcup_{j\in \mathfrak{S}_{i}}^{}\phi_{i,j}(X_{j}) =\overline{X}_{i} \setminus \phi_{0,i}^{-1}(\partial Y\cap \partial \Omega_{i}).
\end{equation*}

In other words, each point \( x_{i}\) in reference domain \( X_{i}\) and on its boundary is uniquely associated to a node \( j\in\mathfrak{S}_{i}\) and a point \( x_{j}\in X_{j}\) such that \( x_{i} = \phi_{i,j}(x_{j})\). For \( x_{i}\) on the boundary \( \partial X_{i}\), we have \( j\) a descendant of \( i\) whereas for \( x_{i}\) in the interior of \( X_{i}\), we have \( j = i\). All points that are mapped to the \small physical boundary \( \partial Y\) by \( \phi_{0, i}\) are exempt from this rule. 
\end{itemize}
\end{definition}
From this we see that each \( \Omega_{i}\) is indeed a manifold, and furthermore we have access to a partition of its boundary through the DAG \(\mathfrak{S}_{i}\). This partitioning is illustrated in Fig. \ref{fig:illustration-of-reference-domain}.
\begin{figure}[!htbp]
\includegraphics[width=\textwidth]{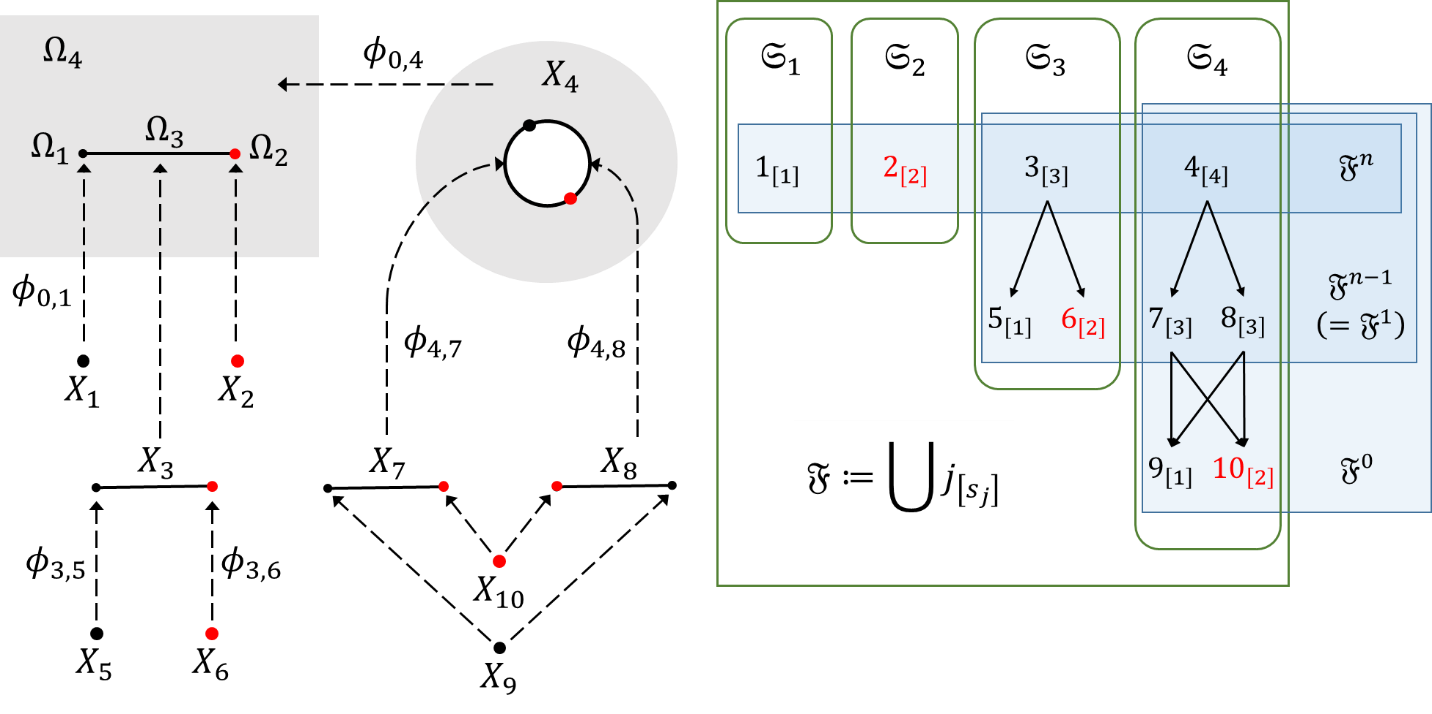}
\caption{Left: Illustration of reference domains and mappings to a physical domain with a single slit. The domains that map to \( \Omega_{2}\) and their respective images under the mappings \( \phi_{i,j}\) are highlighted in red. To comply with Definition \ref{def:2.4}, each \( X_{j}\) with equal \( s_{j}\) is considered equal. Right: The structure of the forest \(\mathfrak{F}\) and its component DAGs. For each node \( j\), the value of \( s_{j}\) denotes index of the domain with which \( \phi_{0,j}(X_{j})\) coincides in the physical domain. The \( k\)-forests \(\mathfrak{F}^{k}\), depicted in blue, are introduced in Section \ref{sec:differential_operators}.}
\label{fig:illustration-of-reference-domain}
\end{figure}

Based on the structure given in Definition \ref{def:2.3}, we can now provide a global structure to partitions \( \Omega\) of \( Y\) as follows  \cite{boon2021functional}.
\begin{definition}\label{def:2.4}
A forest \(\mathfrak{F } \coloneqq \bigcup_{i\in I}^{}\mathfrak{S}_{i}\) is \textit{conforming} to \( \Omega\) if the DAGs \(\mathfrak{S}_{i}\) are conforming to \( \Omega_{i}\) for all \( i\in I\) in the sense of Definition \ref{def:2.3}, and if for any \( j_{1},j_{2}\) such that \( s_{j_{1}} = s_{j_{1}}\), it holds that \( X_{j_{1}} = X_{j_{2}}\) and \( \phi_{0,j_{1}} = \phi_{0,j_{2}}\).
\end{definition}
Our main concern is partitions with conforming forests, and for clarity, we encode this in the following definition: 
\begin{definition}\label{def:2.5}
A partition \( \Omega\) of \( Y\) is \textit{admissible} if a conforming forest \(\mathfrak{F}\) exists. For any admissible partition, we denote the product space of reference domains as \(\mathfrak{X}_{i}\coloneqq\prod_{j\in\mathfrak{S}_{i}}^{}X_{j}\) and \(\mathfrak{X } \coloneqq \prod_{i\in I}^{}\mathfrak{X}_{i}\). 
\end{definition}
Definition \ref{def:2.5} allows for a rather large generality of domains, including curved and self-intersecting domains, and multiple examples are provided in the cited reference  \cite{boon2021functional}. In the present context, a relevant illustration for the case of a single fracture is provided below. 
\begin{example}
In the case of a single fracture, as was discussed in Figure \ref{fig:illustration-of-domain-y}, it creates a geometry as illustrated in Figure \ref{fig:illustration-of-reference-domain}. Note that in this example, only the domain \( \Omega_{4}\) has a non-contractible reference domain \( X_{4}\) associated with it, which comes in part from the fact that it has two boundaries (from ``the top" and from ``the bottom") neighboring the fracture \( \Omega_{3}\).
\end{example}
\begin{remark}
This structure naturally allows for lower-dimensional domains to terminate at, or even exist entirely on, the global boundary \( \partial Y\). In turn, boundary conditions can naturally be inherited on the fractures through the mappings \( \phi_{0,j}\). 
\end{remark}
As a convention for indexing, we will use as above \( i\in I\) for the roots of the DAGs, \( j\in\mathfrak{S}_{i}\) for the components of DAG \( i\), and finally we will write \( j\in\mathfrak{F}\) for the components of the full forest. We moreover define the following index set
\begin{equation}
I^{d} \coloneqq \left\{ i\in I \mid   d_{i} = d\right\},
\end{equation}
Additionally, let \( I_{j}^{d} \coloneqq \left\{ l\in I_{j} \mid   d_{j} = d\right\} .\) We will moreover use \( I^{d<n}\) to denote the set \(\left\{ i\in I\mid  d_{i}<n\right\}\).

For any domain \( X_{i}\) with \( \Omega_{i} = \phi_{0,i}(X_{i})\), we denote by \(\mathbf{F}_{i} \coloneqq D\phi_{0,i}\) the Frechet derivative of the \( C^{1}\) mapping \( \phi_{0,i}\), defined as the linear operator such that for any vector \( v\in\mathbb{R}^{d_{i}}\) and any point \( x\in X_{i}\) then
\begin{equation}\label{eq:2.5}
\mathbf{F}_{i}(x)v = \lim_{\epsilon \rightarrow 0}\frac{\phi_{0,i}(x+\epsilon v)-\phi_{0,i}(x)}{\epsilon }.
\end{equation}
Note that in terms of vector-matrix notation, which we conform to herein, we represent \(\mathbf{F}\) by a matrix whose rows correspond to gradients of the components of \( \phi_{0, i}\).

We will need appropriate extensions of the mappings \( \phi_{0, i}\), so that we can transform vectors in \(\mathbb{R}^{n}\). For root nodes \(i\in I\), i.e. the manifolds of dimension \( d_i \), we define these in the following way:
\begin{definition}\label{def:2.6}
Let \( i\in I\). Let \(\hat{X}_{i}\) be an open domain such that \( \dim(\hat{X}_{i}) = n\) and  \(\overline{X}_{i}\times\left\{ 0\right\}^{n-d_i} \subset\hat{X}_{i}\). Then the \textit{extended mapping} \(\hat{\phi }_{0,i} :\hat{X}_{i}\rightarrow\hat{\Omega }_{i}\subset Y\) is defined such that 
\begin{itemize} \small
    \item \(\hat{\phi }_{0,i} = \phi_{0,i}\) in \( \overline{X}_{i}\times\left\{ 0\right\}^{n-d_i}\).

    \item Orthogonality with respect to \( X_{i}\) and \( \Omega_{i}\) is preserved, i.e. the standard basis vector(s) \(\mathbf{e}_{d}\) for \( d>d_{i}\) is/are mapped to \( D\hat{\phi }_{0,i}\mathbf{e}_{n}\perp T\Omega_{i}\) with \( T\Omega_{i}\) the tangent bundle of \( \Omega_{i}\). 

    \item \(\hat{\phi }_{0,i}\) has a fixed scaling \(\mathcal{l}_{i}\) with respect to the orthogonal complement of \( X_{i}\), i.e. it holds that \( \mathrm{vol}(D\hat{\phi }_{0,i}) =\mathcal{l}_{i}^{n-d_i}\mathrm{vol}(D\phi_{0,i})\) in \( \overline{X}_{i}\times\left\{ 0\right\}^{n-d_i}\).

\end{itemize}
\end{definition}
We recall that the \( d_{i}\) dimensional volume spanned by the derivative of a map \( D\phi_{0,i}\) is given by 
\begin{equation} \label{eq: definition vol}
\mathrm{vol}(D\phi_{0,i}) \coloneqq \sqrt{\det((D\phi_{0,i})^{T}D\phi_{0,i})}.
\end{equation}
The definition of the extended mappings implies that for \(i\in I^n\), then simply \(\hat{\phi }_{0,i} = \phi_{0,i}\).

In Definition \ref{def:2.6}, the extension is given a length-scale designated by a parameter \(\mathcal{l}_{i}\sim\mathcal{l}_{\epsilon }\). This is an important point with respect to modeling, because this implies that we choose a conceptually arbitrary length-scale transversely to the fractures. Indeed, this is a necessary consequence of the mixed-dimensional continuum assumption: Since the transverse opening of a fracture is negligible (relative to the metric of the problem), if we are to nevertheless measure this opening, it must be measured in a different metric. This situation is analogous to multi-scale expansions encountered in homogenization, where the model depends in a non-negligible way on a (fine-scale) coordinate, which has negligible extent relative to the main (coarse-scale) coordinate of the problem (see e.g.  \cite{hornung1996homogenization}). 

For the leaves of the DAGs, corresponding to e.g. boundaries of the solid matrix and the tips of the fractures, we have more freedom to choose an extended mapping, resulting in an arbitrary extended coordinate system around the boundaries of domain. We specify the definition of the extended mappings on the boundaries of domains as follows:
\begin{definition}\label{def:2.7}
Let \( i\in I\) and let \( j\in I_{i}\) be a descendant. Let \(\hat{X}_{j}\subset\mathbb{R}^{n}\) with \(\overline{X}_{j}\times\left\{ 0\right\}^{n-d_{j}}\subset\hat{X}_{j}\). Then an \textit{extended mapping }\(\hat{\phi }_{i,j} :\hat{X}_{j}\rightarrow\hat{X }_i \) satisfies
\begin{itemize} \small
    \item \(\hat{\phi }_{i,j} = \phi_{i,j}\) in \(\overline{X}_{j}\times\left\{ 0\right\}^{n-d_{j}}\).

    \item Orthogonality with respect to \( X_{j}\) and \( X_i\) is preserved, i.e. the standard basis vector(s) \(\mathbf{e}_{d}\) for \( d>d_{i}\) is/are mapped to \( D\hat{\phi }_{i,j}\mathbf{e}_{d}\perp TX_i\).

    \item \(\hat{\phi }_{i,j}\) has a fixed scaling with respect to orthogonal complement of \( X_{j}\), i.e. it holds that \( \mathrm{vol}(D\hat{\phi }_{i,j}) =  \mathrm{vol}(D\phi_{i,j})(x)\) in \( \overline{X}_{j}\times\left\{ 0\right\}^{n-d_{j}}\).
\end{itemize}
The extended mapping to the physical domain is given by composition, \(\hat{\phi }_{0,j} = \hat{\phi }_{0,i} \circ \hat{\phi }_{i,j}\).
\end{definition}
The above definition does not uniquely specify an extended mapping whenever \(d_i-d_j \geq 2\), however the precise choice of extensions has no impact on the following derivations, and we therefore omit a further specification.

For mechanics, we will be interested in deformation. Thus, we will allow for the domain \( Y\), the partition \( \Omega_{i}\), and the mappings \( \phi_{0,i}\) to be time-dependent. However, we will not allow for structure of the partition to change, i.e. the DAGs \(\mathfrak{S}_{i}\), the topological dimensions \( d_{i}\) and the identification of root nodes \( s_{i}\) are not variable and neither are the mappings \( \phi_{i,j}\) for \( j\in I_{i}\). Nor do we allow for any macroscopic opening of the fractures, that is to say, we a priori assume that the dynamics stay within the range of validity of the mixed-dimensional continuum assumption. Nevertheless, we allow for sliding of the fractures, as well as fracture opening on the scale of \(\mathcal{l}_{\epsilon }\). In order to capture this, we extend the definition of coordinate mappings:
\begin{definition}\label{def:2.8}
For \( j, k\in\mathfrak{F}\) with \( s_{j} = s_{k}\), we define the coordinate mapping
\begin{equation*}
\phi_{j, k} \coloneqq \phi_{0,j}^{-1}\circ \pi_{j,k} \circ \phi_{0,k},
\end{equation*}
with \( \pi_{j, k}(x_{k})\) identifying the point on \( \Omega_{j}\) closest to \( x_{k}\), i.e. 
\begin{equation*}
\pi_{j, k}(x_{k}) \coloneqq \text{argmin}_{x_{j}\in \Omega_{j}}\left\vert x_{k}-x_{j}\right\vert.
\end{equation*}
\end{definition}

Note that if \( j, k\) are members of the same conforming DAG, then \( \pi_{j,k}\) is the identity operator. In turn, Definition \ref{def:2.8} generalizes the definition of \( \phi_{j,k}\) to nodes of different DAGs. Moreover, the mixed-dimensional continuum assumption assures that the expression $\left\vert x_{k}-x_{j}\right\vert = \mathcal{O}( \ell_\epsilon)$, and as such we infer that $\pi_{j, k}(x_{k})$ should be uniquely defined in our context. We will interpret opening of fractures that are sufficiently large to lead to non-uniqueness of $\pi_{j, k}(x_{k})$ as outside the range of validity of the modeling scales.

\begin{remark}
\label{remark:2.2}
We recall that since the mappings \( \phi_{0,i}\) are time-dependent, the coordinate map \( \pi_{j, k}\) will also be. Thus mappings \( \phi_{j,k}\), when \( j\) and \( k\) belong to different DAGs \(\mathfrak{S}_{i_{1}}\) and \(\mathfrak{S}_{i_{2}}\), will also in general be time-dependent. 
\end{remark}
Concluding this section, we provide an overview of the most important definitions related to mixed-dimensional geometries.
\begin{table}[!htbp]
    \centering
    \caption{Summary of geometrical concepts}
    \label{tab:summary_of_geometrical_concepts}
    \begin{tabular}{|c|l|}
    \hline
         \(I\) & Index set of roots, each root corresponds to a subdomain \(\Omega_i \subseteq Y\).  \\
        \(I^{d}\) & Subset of roots \(i \in I\) with dimension \(d_i = d\).\\
        \(\mathfrak{S}_i\) & For each \(i \in I\), a Directed Acyclic Graphs (DAG) keeps track of its boundaries.\\
        \(\mathfrak{F}\) & The forest is the collection of all DAGs.\\
        \(X_j\) & A smooth reference domain corresponding to node j.\\
        \(\phi_{i,j}\) & Coordinate map from \(X_{j}\) to \(X_{i}\). \(i = 0\) indicates the mapping to physical space.\\
        \hline
    \end{tabular}
\end{table}

\subsection{Mixed-dimensional function spaces}\label{sec:mixed_dimensional}
As a basis for mixed-dimensional modeling of the poromechanical system, we start by defining the appropriate variables on the geometry from Section \ref{sec:geometry}. A critical aspect of our model is that the variables will be defined on reference domains whereas the equations describing the physical model relate to the physical domain. It is therefore important to correctly transform functions between these domains. 

We approach these transformations systematically by using concepts from the field of exterior calculus, in particular the equivalent representation of functions as differential forms. The \textit{pullback operator} then provides the appropriate transformation mappings and, additionally, we obtain a canonical definition for trace operators. However, in order to make this presentation accessible to a broader audience, we only briefly exploit the calculus of differential forms, and then translate the definitions in terms of the representation by ``standard" functions. We refer the interested reader to  \cite{boon2021functional} for more details on the mixed-dimensional exterior calculus framework. Readers not familiar with exterior calculus are encouraged to skip ahead to Examples \ref{eg:2.3} and \ref{eg:2.3b} and use these as a guide to the exposition.

We start with the following key building block, which is illustrated in the right part of Figure \ref{fig:illustration-of-reference-domain}:
\begin{definition}\label{def:2.9}
The \( k\)\textit{-forest} \(\mathfrak{F}^{k}\mathfrak{\subseteq F}\) is defined for \( 0\leq k\leq n\) as the subgraph induced by the nodes \\
\begin{equation*}
\bigcup_{
\substack{
i\in I \\ 
 d_{i}\geq n-k \\ }
}\left\{  j\in\mathfrak{S}_{i} \mid  d_{i}-d_{j}\leq n-k\right\} .
\end{equation*}
\end{definition}
As is apparent here, we keep track of the codimension between \( X_{j}\) and \( X_{i}\) and the difference between \( n\) and \( k.\) Later in this subsection, these determine on which boundary segments of given codimension we define function traces. For ease of reference, we summarize the integer values that play important roles in this section in the following table.
\begin{table}[!htbp]
    \centering
    \caption{Summary of integers describing mixed-dimensional functions}
    \label{tab:summary_of_integers_describing_multi_dimentional_functions}
    \begin{tabular}{|c|p{8cm}|}
        \hline
         \(n\)& Dimension of the physical domain \(Y\). \\
         \(d_j\)& Dimension of subdomain \(X_j\).\\
         \(k\) & Order of the mixed-dimensional differential form.\\
         \(k_j\) & Local order of the differential form on \(X_{j}\).\newline This depends on \(k\) and the codimension between \(X_j\) and its root.\\
         \(p\) & Integer to distinguish between scalar and vector-valued forms. Here in, we are interested in \(p = 1\) for the flow equations and \(p = n\) for elasticity.\\
    \hline
    \end{tabular}
\end{table}

We continue with our brief exposition of mixed-dimensional differential forms (a full account is given in  \cite{boon2021functional}). The following five definitions suffice for the purposes of this work (confer e.g.  \cite{spivak1965} for a concise introduction to fixed-dimensional differential forms).
\begin{definition}\label{sec:2.10}
For any \( 0\leq k\leq n\), let the mixed-dimensional reference domain \(\mathfrak{X}^{k}\) be denoted \(\mathfrak{X}^{k} \coloneqq \coprod_{j\in\mathfrak{F}^{k}}^{}X_{j}\) with \( \coprod\) denoting the disjoint union. Each \(\mathfrak{X}^{k}\) is thus a collection of subdomains that corresponds to a \( k\)-forest \(\mathfrak{F}^{k}\), exemplified in Fig. \ref{fig:illustration-of-reference-domain}.
\end{definition}
We continue by defining the linear forms that are continuous on each \( X_{j}\subseteq\mathfrak{X}^{k}\).
\begin{definition}\label{sec:2.11}
For any \( 0\leq k\leq n\), let the mixed-dimensional \textit{locally continuous} \( k\)-forms with values in \(\mathbb{R}^{p}\), for \(p\in \{1,n\}\), be denoted \(\tilde{C}\mathfrak{L}^{k}(\mathfrak{X}^{k},\mathbb{R}^{p})\), and defined as a product space of alternating differential \( k_{j} \)-linear forms  \\
\begin{equation*}
\tilde{C}\mathfrak{L}^{k}(\mathfrak{X}^{k},\mathbb{R}^{p}) \coloneqq \prod_{j\in\mathfrak{F}^{k}}^{}C^{1}\Lambda^{k_{j}}(X_{j},\mathbb{R}^{p}),
\end{equation*}
where the local order is given by \( k_{j} = d_{i}-(n-k)\) for \( j\in\mathfrak{S}_{i}\), and where \( C^{1}\Lambda^{k_{j}}(X_{j},\mathbb{R}^{p})\) are bounded \( C^{1}\)-continuous forms on \( X_{j}\). 
\end{definition}
We refer to \(\mathfrak{a}\in \tilde{C}\mathfrak{L}^{k}(\mathfrak{X}^{k},\mathbb{R}^{p})\) as a \textit{mixed-dimensional differential form }since it contains elements defined on manifolds of different dimensionalities. An element of such a function space will be denoted using the Gothic font. To extract a local form on, say, \( X_{j}\) from \(\mathfrak{a}\), we will use the notation \( \iota_{j}\mathfrak{a}\in C^{1}\Lambda^{k_{j}}(X_{j},\mathbb{R}^{p})\). This allows us to generalize the normal algebraic operators by insisting that they commute with \(\iota\), thus if also \(\mathfrak{b}\in \tilde{C}\mathfrak{L}^{k}(\mathfrak{X}^{k},\mathbb{R}^{p})\), then \(\iota_j(\mathfrak{a}+\mathfrak{b})= \iota_j(\mathfrak{a})+\iota_j(\mathfrak{b})\), and similarly for subtraction, multiplication and division.

\begin{example}
\label{eg:2.2}
We recall that a differential \( k\)-linear form \( a\in C^{1}\Lambda^{k}(X_{j}, \mathbb{R})\), takes as argument \( k\) vectors from the tangent space \( TX_{j} =\mathbb{R}^{d_{j}}\). Thus \( a(x) :(\mathbb{R}^{d_{j}})^{k_{i}}\mathbb{\rightarrow R}\). Moreover, the skew-symmetric (alternating) properties of \( C^{1}\Lambda^{k}(X_{j}, \mathbb{R})\) ensure that permutations of vectors alternate signs, e.g. for \( v_{1},v_{2}\in TX_{j}\), and \( k = 2\) then \( a(x)(v_{1},v_{2}) = -a(x)(v_{2},v_{1})\).  
\end{example}
Coordinate transformations are naturally handled in the context of exterior calculus through the pullback operator, defined next.
\begin{definition}\label{def:2.10c}
For any differentiable mapping \( \phi  : X\rightarrow \phi(X)\), the pullback \( \phi^{\ast }\) of the differential form \( a\in C^{1}\Lambda^{k}(\phi(X),\mathbb{R}^{p})\) is the unique operator such that \( \phi^{\ast }a\) satisfies for \( v_{1},\ldots ,v_{k}\in TX\) \\ \begin{equation*}
(\phi^{\ast }a)(v_{1},\ldots ,v_{k}) = a((D\phi) v_{1},\ldots ,(D\phi) v_{k})
\end{equation*}
\end{definition}
Additionally, we have the trace operator that maps continuous differential forms to forms defined on the boundary.
\begin{definition}\label{def:2.10d}
For the boundary \( \partial X\), the trace of the differential form \( a\in C^{1}\Lambda^{k}(X, \mathbb{R}^{p})\) is denoted \( Tr_{\partial X} a\in C^{1}\Lambda^{k}(\partial X,\mathbb{R}^{p})\) and is defined as the restriction of \( a\) to the manifold \( \partial X\). 
\end{definition}
We combine the locally continuous differential forms in order to establish a notion of globally continuous forms by exploiting the pullback and trace operators.
\begin{definition}\label{def:2.11}
For any \( 0\leq k\leq n\), let the mixed-dimensional \textit{continuous} \( k\)-forms be denoted \( C\mathfrak{L}^{k} (\mathfrak{X}^{k},\mathbb{R}^{p})\), and defined as 
\begin{align*}
C\mathfrak{L}^{k}(\mathfrak{X}^{k},\mathbb{R}^{p}) \coloneqq &\left\{ \mathfrak{a}\in \tilde{C}\mathfrak{L}^{k}(\mathfrak{X}^{k},\mathbb{R}^{p}) \mid    \phi_{i,j}^{\ast }\mathrm{Tr}_{\partial_{j}X_{i}}\iota_{i}\mathfrak{a} = \varepsilon_{i,j}\iota_{j}\mathfrak{a},\right. \\
&\left. \forall i,j\in\mathfrak{F}^{k} \text{ with } i\in I \text{ and } j\in I_{i}\right\}.
\end{align*}

Here, \( \varepsilon\) is the orientation indicator that takes the value \( \varepsilon_{i,j} = 1\) if \( \partial_{j}X_{i}\) and \( \phi_{i,j}(X_{j})\) have the same orientation, and \( \varepsilon_{i,j} = -1\) otherwise.
\end{definition}
\begin{remark}
\label{remark:2.3}
The orientation \( \varepsilon_{i,j}\) can directly be calculated by verifying whether the composition \(\hat{\phi }_{j}^{-1}\hat{\phi }_{i}\) of extended coordinate maps preserves orientation in \(\mathbb{R}^{n}\).
\end{remark}
An important detail is that the space \( C\mathfrak{L}^{k}(\mathfrak{X}^{k},\mathbb{R}^{p})\) is not globally continuous, as the continuity is only imposed within each DAG \(\mathfrak{S}_{i}\). Pre-empting later developments, we note that e.g. deformations in \( C\mathfrak{L}^{0}(\mathfrak{X}^{0},\mathbb{R}^{n})\) are therefore allowed to be discontinuous across fractures in physical space. 

The above definitions of mixed-dimensional differential forms, as well as their pullback and trace operators, allow us now to consider representations as mixed-dimensional functions as used in the remainder of this paper. We first start by identifying the standard representation of differential forms in terms of classical functions from multivariate calculus. This discussion will exclusively consider the case of \( n = 3\).  Similar representations are used for lower dimensions (see e.g.  \cite{arnold2018finite}).
\begin{definition}\label{def:2.9h}
At any given point \( x\in X_{j}\), the space of differential forms \( C^{1}\Lambda^{k}(X_{j},\mathbb{R}^{p})\) has \( p \binom{d_{j}}{k}\) degrees of freedom. The \textit{standard representation} \(\grave{a}=\mathbbm{r}\acute{a}\) of a differential form \(\acute{a}\in C^{1}\Lambda^{k}(X_{j},\mathbb{R}^{p})\) is given for \( p = 1\) with respect to the standard basis for \(\mathbb{R}^{n}\) as follows (when we need to distinguish between the form and its representation, we denote the form by an \textit{accent aigu}, and the representation by an \textit{accent grave}): 
\begin{enumerate}  \small
    \item \label{def: 2.9.1}
    For \( k = 0\), the differential forms coincide with functions \(\grave{a}\in C^{1}(X_{j})\), thus \(\grave{a} =\acute{a}\). 

    \item \label{def: 2.9.2}
    For \( k = 1\), the differential one-forms \( \Lambda^{k}(X_{j})\) are represented by vector functions \(\grave{a}\in C^{1}(X_{j},TX_{j})\). This representation is the Riesz representation, which satisfies for vector fields \( v_{1}\in TX_{j}\) that \(\acute{a}(v_{1}) = v_{1} \cdot \grave{a}\). 

    \item \label{def: 2.9.3}
    For \( k = n-1\), the differential forms \( \Lambda^{k}(X_{j})\) are represented by ``flux" functions \(\grave{a}(x)\in C^{1}(X_{j},TX_{j})\). This representation satisfies for vector fields \( v_{1},v_{2}\in TX_{j}\) that \(\acute{a}(v_{1},v_{2}) = \mathrm{vol}(\grave{a},v_{1},v_{2})\), where \( \mathrm{vol}\) is the volume of the parallelopiped spanned by its arguments. 
    \item \label{def: 2.9.4}
      For \(k=n\), the differential forms \(a \in \Lambda^{k} (X_j)\) coincide with “density” functions \(\grave{a}(x) \in C^1 (X_j )\). This representation satisfies for vector fields \(v_1,v_2,v_3 \in TX_j\) \text{that} \(\acute{a}(v_1, v_2, v_3) = \grave{a} \mathrm{vol} (v_1, v_2, v_3)\)
\end{enumerate}

Since both pullback and trace act differently depending on the order of the form, we use the vernacular ``flux" to distinguish representations of \( n-1\) forms from representations of \( 1\)-forms, and similarly ``density" to distinguish representations of \( n\)-forms from representations of \( 0\)-forms. 
\end{definition}
\begin{remark}
\label{remark:2.4}
From Definition \ref{def:2.9h}, it is apparent that for \( n\leq 2\) (and importantly for our context, domains \( X_{j}\) for \( d_{j}\leq 2\)), the choice of representation is not unique, since e.g. the possibility \( 1 = k = d_{j}-1\) exists. This is a classical observation, and is resolved in the current context by the following conventions: 1) For \( p = n\), representations as (vector) functions, i.e. \ref{def: 2.9.1}. and \ref{def: 2.9.2}. in Definition \ref{def:2.9h}, are preferred over \ref{def: 2.9.3}. and \ref{def: 2.9.4}. When needed, we emphasize this representation by the subscript \(\mathbbm{r}_{-}\). 2) For \( p = 1\) representations as fluxes and densities, i.e. \ref{def: 2.9.3}. and \ref{def: 2.9.4}. in Definition \ref{def:2.9h}, are preferred over \ref{def: 2.9.1}. and \ref{def: 2.9.2}. When needed, we emphasize this representation by the subscript \(\mathbbm{r}_{+}\).
\end{remark}

Once a choice of representations has been established, we now have a one-to-one correspondence between mixed-dimensional differential forms and their function counterparts, and the inverse representation \(\mathbbm{r}^{-1}\) is thus well-defined. The definitions of mixed-dimensional function spaces are now implied by the previous developments. 
\begin{definition}\label{def:2.16}
For any mixed-dimensional form \(\acute{\mathfrak{a}}\in C\mathfrak{L}^{k}(\mathfrak{X}^{k},\mathbb{R}^{p})\) we denote its standard mixed-dimensional representation as \(\grave{\mathfrak{a}} =\mathbbm{r}_{\pm }\acute{\mathfrak{a}}\in C(\mathfrak{X}^{k},\mathbb{R}^{p})\), iff \(\mathbbm{r}_{\pm }\iota_{i}\acute{\mathfrak{a}} = \iota_{i}\grave{\mathfrak{a}}\) for all \( i\in I\). For \( p = n\) the spaces of \textit{continuous mixed-dimensional functions} on \(\mathfrak{X}^{k}\) that are relevant for this paper are given for \( k = 0, 1\) by the choice:
\begin{equation*}
C(\mathfrak{X}^{k},\mathbb{R}^{n}) \coloneqq \mathbbm{r}_{-}C\mathfrak{L}^{k}(\mathfrak{X}^{k},\mathbb{R}^{n}).
\end{equation*}
These are referred to as mixed-dimensional vector functions (\( k = 0\)) and matrix functions (\( k = 1\)).

For \( p = 1\) the spaces of continuous\textit{ }mixed-dimensional functions on \(\mathfrak{X}^{k}\) that are relevant for this paper are given for \( k = n-1,n\) by the choice:
\begin{equation*}
C(\mathfrak{X}^{k}, \mathbb{R}) \coloneqq \mathbbm{r}_{+}C\mathfrak{L}^{k}(\mathfrak{X}^{k}, \mathbb{R}).
\end{equation*}
We refer to the latter spaces as mixed dimensional fluxes (\( k = n-1\)) and densities (\( k = n\)). 
\end{definition}
\begin{definition}
\label{def:2.17}
The space of \textit{continuous mixed-dimensional functions with vanishing trace} is given by
\begin{equation*}
\mathring{C}(\mathfrak{X}^{k},\mathbb{R}^{p}) \coloneqq \{\grave{\mathfrak{a}}\in C(\mathfrak{X}^{k},\mathbb{R}^{p})\mid  \mathrm{Tr}_{\partial_{Y}X_{j}}\iota_{j}\acute{\mathfrak{a}} = 0, \forall j\in\mathfrak{F}^{k}\},
\end{equation*}
with \( \partial_{Y} X_{j} \coloneqq \phi_{0,j}^{-1}(\partial Y\cap \partial \Omega_{j}).\)

\end{definition}
On \( C(\mathfrak{X}^{k},\mathbb{R}^{p})\), we introduce a component-wise inner product as follows (we use angled brackets to denote inner products, and reserve parenthesis for tuples):
\begin{align}
\left\langle\mathfrak{a, b}\right\rangle_{\mathfrak{X}^{k}} &\coloneqq \sum_{j\in\mathfrak{F}^{k}}^{}\left\langle \iota_{j}\mathfrak{a,}\iota_{j}\mathfrak{b}\right\rangle_{X_{j}}, &
\forall \mathfrak{a, b} &\in C(\mathfrak{X}^{k},\mathbb{R}^{p}).
\end{align}
This naturally induces an \( L^{2}\)-norm 
\begin{align}
\Vert\mathfrak{a}\Vert_{\mathfrak{X}^{k}} &\coloneqq \sqrt{\left\langle\mathfrak{a, b}\right\rangle_{\mathfrak{X}^{k}}}, &
\forall\mathfrak{a} &\in C(\mathfrak{X}^{k},\mathbb{R}^{p}).
\end{align}
The space of \( L^{2}\) integrable functions can now be defined as the closure of the continuous functions with respect to this norm  \cite{boon2021functional}:
\begin{definition}
\label{def:2.18}
For \( 0\leq k\leq n\) and \( p \in \{1, n\}\), let the space of mixed-dimensional\textit{ square integrable functions} on \(\mathfrak{X}^{k}\) be defined as
\begin{equation*}
L^{2}(\mathfrak{X}^{k},\mathbb{R}^{p}) \coloneqq \overline{C(\mathfrak{X}^{k},\mathbb{R}^{p})}.
\end{equation*}

\end{definition}

We emphasize that this definition is a direct result of the representation of differential in terms of conventional function spaces. As is clear from the definition (and motivated by Definition \ref{def:2.9h}), even when the space is generated with \( p = 1\), some function components may be vector-valued, as we will see in the examples below.
\begin{example}
\label{eg:2.3}
The following two cases exemplify the spaces for \( p = 1\) that are relevant to our model.
\begin{itemize}  \small
    \item Let \( k = n\), then \(\mathfrak{F}^{n}\) is given by the roots \( I\). The number of degrees of freedom \( p \binom{d_{j}}{k} = \binom{n}{n} = 1\) in this case, so we have \( L^{2}(\mathfrak{X}^{n}, \mathbb{R}) = \prod_{i\in I}^{}L^{2}(X_{i}, \mathbb{R})\). This is the space in which we will define scalar density functions such as fluid pressures. In Figure \ref{fig:illustration-of-reference-domain}, these roots have indices \( i\) with \( 1\leq i\leq 4\).

    \item Let \( k = n-1\) and \(\mathfrak{a}\in L^{2}(\mathfrak{X}^{n-1}, \mathbb{R})\). Then for each root \( i\in I\) with \( d_{i}\geq 1\), \( \iota_{i}\mathfrak{a}\) is given by a vector function in \( L^{2}(X_{i},\mathbb{R}^{d_{i}})\). Moreover, for each \( j\in I_{i}^{d_{i}-1}\), we have that \( \iota_{j}\mathfrak{a}\in L^{2}(X_{j}, \mathbb{R})\), i.e. a scalar distribution on each boundary that models an interface between manifolds of codimension one. This will form our space in which we define the fluid flux, both internal to each subdomain and across interfaces. In Figure \ref{fig:illustration-of-reference-domain}, these functions are then vectors on the root nodes \( i = \{ 3,4\}\) and scalars on the nodes \( j\) with \( 5\leq j\leq 8\). Note that for \( \Omega_{3}\), there are three nodes \( j\in\left\{ 3,7,8\right\}\) with \( s_{j} = 3\). This allows us to separate the tangential flux inside the fracture (\( j = 3\)) and the normal flux entering from the two sides (\( j \in\left\{ 7,8\right\}\)).
\end{itemize}
\end{example}
\begin{example}
\label{eg:2.3b}
Similarly, we give examples of the two relevant spaces for \( p = n\). 
\begin{itemize}  \small
    \item Let \( k = 0\). We have that \(\mathfrak{F}^{0}\) consists of the roots \( i\in I^{n}\) and all their descendants. Thus, for \(\mathfrak{a}\in L^{2}(\mathfrak{X}^{0},\mathbb{R}^{n})\) and \( i\in I^{n}\), we have that \( \iota_{j}\mathfrak{a}\in L^{2}(X_{j},\mathbb{R}^{n})\) for all \( j\in\mathfrak{S}_{i}\). We will use this space to model the displacement of the solid. In Figure \ref{fig:illustration-of-reference-domain}, these functions are defined on the root with index \( 4\) and the descendant nodes \( j\) with \( 7\leq j\leq 10\).

    \item Let \( k = 1\) and \( n = 3\). The 1-forest \(\mathfrak{F}^{1}\) then contains all roots \( i\in I^{2}\cup I^{3}\) and their descendants \( j\) with \( d_{j}\geq d_{i}-2\). For the roots \( i\in I^{3}\), we have that \( \iota_{i}\mathfrak{a}\in L^{2}(X_{i},\mathbb{R}^{3\times d_{i}})\) and for \( j\in I_{i}^{1}\cup I_{i}^{2}\), it follows that \( \iota_{j}\mathfrak{a}\in L^{2}(X_{j},\mathbb{R}^{3\times d_{j}})\). On the other hand, for \( i\in I^{2}\), we obtain \( \iota_{i}\mathfrak{a}\in L^{2}(X_{i},\mathbb{R}^{3})\) and \( \iota_{j}\mathfrak{a}\in L^{2}(X_{j},\mathbb{R}^{3})\) for \( j\in I_{i}^{0}\cup I_{i}^{1}\). This space will be used to model displacement gradients, allowing us to define stresses and strains in the bulk matrix and its boundaries, as well as across  fractures. In Figure \ref{fig:illustration-of-reference-domain}, these functions are defined on the same domains as the flux functions discussed in Example \ref{eg:2.3}. However, for \( n = 3\), the flux and stresses will have different domains of definition. 

\end{itemize}
\end{example}
The final spaces required are those that are defined on all reference domains, which becomes particularly useful when we consider volumetric strains in Section \ref{sec:mixed_dimensional_linearized} and \ref{sec:mixedd_dimensional_volume}. These are defined for \(p\in \{1,n\}\) as e.g.:
\begin{equation}\label{eq:2.9}
L^{2}(\mathfrak{X}, \mathbb{R}^p) \coloneqq \prod_{j\in\mathfrak{F}}^{}L^{2}(X_{j}, \mathbb{R}^p).
\end{equation}
Analogous definitions extend to \(L^{\infty}(\mathfrak{X}, \mathbb{R}^p)\) and \(C^m(\mathfrak{X}, \mathbb{R}^p)\). 

In summary, the mixed-dimensional function spaces are defined using their equivalent representations as differential forms of order \( k\). This gives us access to pullback and trace operators, as is illustrated in Figure \ref{fig:pullback-of-function}. In turn, function spaces in the physical domain are defined in the next subsection such that their pullback onto reference domains \(\mathfrak{X}^{k}\) have certain regularity properties.
\begin{figure}[!htbp]\centering
\includegraphics[width=0.7\textwidth]{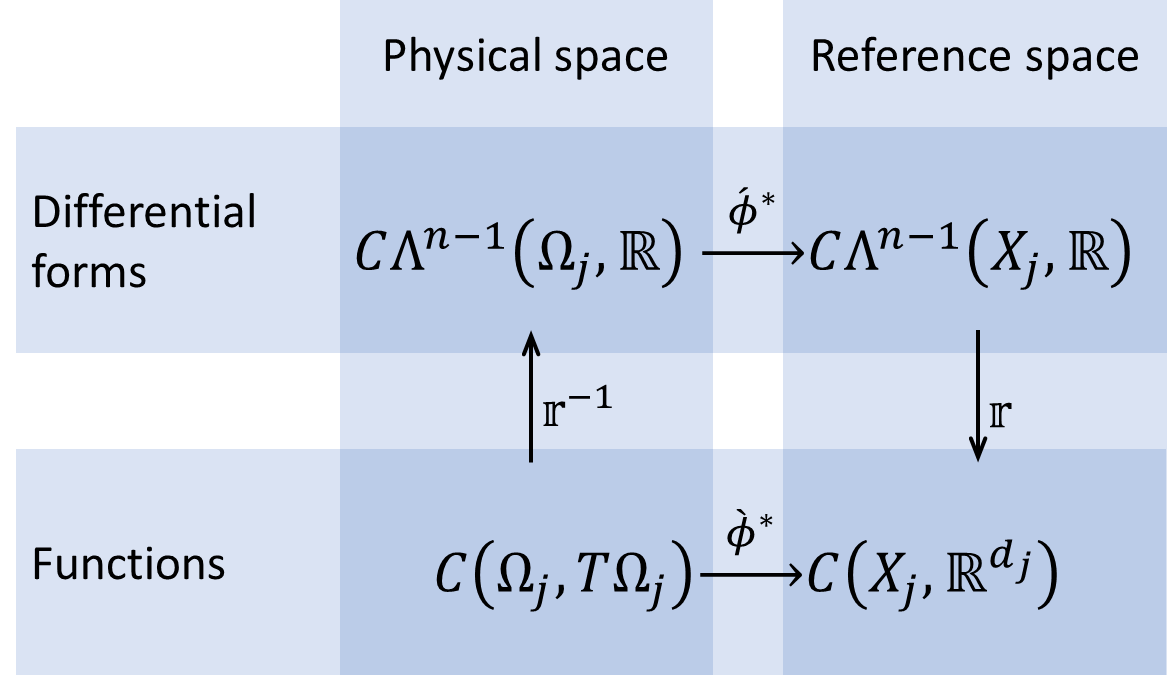}
\caption{The pullback \(\grave{\phi }^{\ast } \coloneqq\mathbbm{r}\acute{\phi}^\ast \mathbbm{r}^{-1} \) of a function is defined using its representation as a differential form. Illustrated here is the case of \( p = 1\) and \( k = n-1\), i.e. the flux functions, for which the operator \(\grave{\phi }^{\ast }\) is known as the Piola transform. The trace operator on functions is defined analogously as \(\grave{\mathrm{Tr}}\coloneqq\mathbbm{r} \acute{\mathrm{Tr}} \mathbbm{r}^{-1}\). Due to the commutativity of this diagram, we omit the accents when denoting these operators.}
\label{fig:pullback-of-function}
\end{figure}

\subsection{Differential operators}\label{sec:differential_operators}

The standard differential operators on manifolds can be extended to the setting of mixed-dimensional geometries. However, in order to achieve the proper coupling between the domains, as required from physical relevance (i.e. the use of the differential operators in conservation laws), manifolds of adjacent dimensionality must be coupled via so-called \textit{jump operators}. 

This section presents the mixed-dimensional gradient and divergence operators, assuming continuous functions of sufficient regularity. For a rigorous exposition of all mixed-dimensional differential operators (including the curl), we refer again to  \cite{boon2021functional}, which follows a classical construction of Čech and de Rham cohomology.

Let us start by defining the jump operator by \(\mathbbm{d}:C(\mathfrak{X}^{k},\mathbb{R}^{p})\rightarrow C(\mathfrak{X}^{k+1},\mathbb{R}^{p})\) that maps between subdomains of codimension one. We define this mapping by introducing a key set of indices. 
\begin{definition}\label{def:2.19}
For any root \( i\in I\), let the index set \( J_{i}\) be given by
\begin{equation*}
 J_{i} \coloneqq \left\{ j\in\mathfrak{F} \mid s_{j} = i \text{ and } j\in I_{\hat{\jmath}} \text{ for some } \hat{\jmath}\in I^{d_{i}+1}\right\} .
\end{equation*}

In other words, this is the set of nodes that coincide in the physical domain with \( i\) and have a root of dimension \( d_{j} = d_{i}+1\). Then the \textit{jump} \(\mathbbm{d}_{\Phi }\mathfrak{a}\in C(\mathfrak{X}^{k+1},\mathbb{R}^{p})\) of a continuous mixed-dimensional function \(\mathfrak{a}\in C(\mathfrak{X}^{k},\mathbb{R}^{p})\) is defined on a subdomain \( X_{i}\) with \( i\in I\) by the signed sum
\begin{equation*}
\iota_{i}(\mathbbm{d}_{\Phi }\mathfrak{a}) =(-1)^{n-k}\left(\sum_{l\in J_{i}}^{}\varepsilon_{i,l}\phi_{l,i}^{\ast }\iota_{l}\mathfrak{a}\right),  \forall i\in\mathfrak{F}^{k+1}\cap I
\end{equation*}

Where the pullback is used to map the function to the appropriate subdomain \( X_{i}\). It is defined through the representation of functions as differential forms, cf. Figure \ref{fig:pullback-of-function}.

The jump \(\mathbbm{d}_{\Phi }\) is extended to descendants \( j\in I_{i}\) by imposing that \(\mathbbm{d}_{\Phi }\mathfrak{a}\) is in \( C(\mathfrak{X}^{k+1},\mathbb{R}^{p})\). The remaining values \( \iota_{j}\mathbbm{d}_{\Phi }\mathfrak{a}\) with \( j\in\mathfrak{F}^{k+1}\setminus I\) are thus determined by trace values, cf. Definition \ref{def:2.11}.
\end{definition}

For an illustration of the domain and range of the jump operator, we refer to Figure \ref{fig:illustration-of-reference-domain}, recalling that \(\mathfrak{X}^{k}\) is the set of subdomains corresponding to the \( k\)-forest \(\mathfrak{F}^{k}\).

We note, as emphasized in Remark \ref{remark:2.2}, that when \( \Phi  = \Phi(t)\) is time-dependent, then so are the operators \( \pi_{l,i}\) and the mappings \( \phi_{l,i}\), and hence also the definition of the operator \(\mathbbm{d}_{\Phi(t)}\). As a notational shorthand, we denote this time-dependent jump in reference space as \(\mathbbm{d}_{t} =\mathbbm{d}_{\Phi(t)}\) when emphasis is needed, and otherwise simply write \(\mathbbm{d} = \mathbbm{d}_{\Phi }\) also for the jump operator on reference space to declutter the presentation.  This dependence of \(\mathbbm{d}_{\Phi }\) on the mapping \( \Phi\) has the (intended) consequence that the jump term \(\mathbbm{d}_{t}\) remains local in physical space for two points in contact, even when the domains they belong to are sliding relative to each other. 
\begin{remark}
\label{remark:2.5}
By the mixed-dimensional continuum assumption, \(\left\vert \pi_{l,i}(x)-x\right\vert\mathcal{ = O}(\mathcal{l}_{\epsilon })\). This definition allows for geometries that are slightly more general than a conforming forest. Indeed, for geometries with a conforming forest \( \pi_{l,i}(x) = x\). Furthermore, the closest point projection is Lipschitz continuous in the limit of infinitesimal smooth deformations. As a consequence, for a configuration \( \Phi\) and a smooth perturbation \( \Psi\) tangential to all boundaries, the derivatives considered from the ``left" and ``right" limits coincide, such that for all \( i\in I^{2}\):
\begin{equation*}
\lim_{\epsilon \rightarrow 0} \epsilon^{-1}\iota_{i}(\mathbbm{d}_{\Phi +\epsilon \Psi }\Phi) = \lim_{\epsilon \rightarrow 0} \epsilon^{-1}\iota_{i}(\mathbbm{d}_{(\Phi +\epsilon \Psi)-\epsilon \Psi }(\Phi -\epsilon \Psi)) = \lim_{\epsilon \rightarrow 0} \epsilon^{-1}\iota_{i}(\mathbbm{d}_{\Phi }(\Phi -\epsilon \Psi)).
\end{equation*}
This limit will be useful when considering the linearized theories later. 
\end{remark}
Next, we will construct the mixed-dimensional differential operators. Formally, these can be defined based on the differential forms and the exterior derivative  \cite{boon2021functional}, and then defining the differential operators on representations by requiring that commutation holds. However, as this introduces more formalisms than what is needed in the current exposition, we will here present the differential operators directly on the mixed-dimensional functions, with an understanding that mixed-dimensional exterior derivatives can similarly be defined on the mixed-dimensional forms. 

 We consider first the gradient, for which we set \( k = 0\), and define the local gradient operator \( \nabla  :C(\mathfrak{X}^{0},\mathbb{R}^{n})\rightarrow L^{2}(\mathfrak{X}^{1},\mathbb{R}^{n})\) as the standard gradient on each reference space: i.e. for \(\mathfrak{a}\in C(\mathfrak{X}^{0},\mathbb{R}^{n})\) let \( \nabla\mathfrak{a}\in L^{2}(\mathfrak{X}^{1},\mathbb{R}^{n})\) be such that
\begin{equation*}
\iota_{j}(\nabla\mathfrak{a}) =\begin{cases}
\nabla(\iota_{j}\mathfrak{a}),   & \forall j\in\mathfrak{F}^{1}\cap\mathfrak{F}^{0}, \\ 
0,   & \forall j\in\mathfrak{F}^{1}\setminus\mathfrak{F}^{0}. \\ 
\end{cases}
\end{equation*}
We emphasize that this operator takes the gradient not only of the components defined on the roots \( j\in I^{n}\), but also on its boundaries \( j\in I_{i}\).
\begin{definition}\label{def:2.20}
The \textit{mixed-dimensional gradient} on vector functions \( \mathfrak{D}_{\Phi }:C(\mathfrak{X}^{0},\mathbb{R}^{n})\rightarrow L^{2}(\mathfrak{X}^{1},\mathbb{R}^{n}), \) is defined as 
\begin{align}
\mathfrak{D}_{\Phi }\mathfrak{a } &\coloneqq \nabla\mathfrak{a} + \mathbbm{d}_{\Phi }\mathfrak{a}, &
\forall \mathfrak{a} &\in C(\mathfrak{X}^{0},\mathbb{R}).
\end{align}

The mappings \( \Phi\) are usually implied from the context, we will then omit the subscript and write \(\mathfrak{D}\).
\end{definition}
Similarly, for the mixed-dimensional divergence, we first define \((\nabla  \cdot ) :C(\mathfrak{X}^{n-1}, \mathbb{R})\rightarrow L^{2}(\mathfrak{X}^{n}, \mathbb{R})\) such that
\begin{equation*}
\iota_{j}(\nabla  \cdot \mathfrak{b}) =\begin{cases}
\nabla  \cdot (\iota_{j}\mathfrak{b}),   & \forall j\in\mathfrak{F}^{n}\cap\mathfrak{F}^{n-1}, \\ 
0,   & \forall j\in\mathfrak{F}^{n}\setminus\mathfrak{F}^{n-1}. \\ 
\end{cases}
\end{equation*}
\begin{definition}\label{def:2.21}
The \textit{mixed-dimensional divergence} \((\mathfrak{D}_{\Phi } \cdot ) :C(\mathfrak{X}^{0}, \mathbb{R})\rightarrow L^{2}(\mathfrak{X}^{1}, \mathbb{R})\) is defined as
\begin{align}
\mathfrak{D}_{\Phi }\cdot \mathfrak{ a } &\coloneqq \nabla  \cdot \mathfrak{a}+\mathbbm{d}_{\Phi } \mathfrak{a}, &
\forall \mathfrak{a} &\in C(\mathfrak{X}^{0}, \mathbb{R}).
\end{align}
The mappings \( \Phi\) are usually implied from the context, we will then omit the subscript and write \((\mathfrak{D} \cdot )\).
\end{definition}
It is important to note that these operators do not possess the same adjointness properties as the conventional gradient and divergence operators since they are defined on different \( k\)-forests. However, the adjoints of mixed-dimensional operators do play a vital role in our model and, to properly define these, we consider the differential operators as instances of densely defined unbounded linear operators (see e.g.  \cite{pedersen1989unbounded}) on \( L^{2}(\mathfrak{X}^{0},\mathbb{R}^{p})\). Taking the adjoint (see Def. \ref{def:A6}) of the mixed-dimensional gradient and divergence then leads us to the co-gradient \((\mathbb{D}_{\Phi } \cdot )\) and co-divergence \((\mathbb{D}_{\Phi })\), respectively.
\begin{definition}\label{def:2.22}
Let the \textit{mixed-dimensional co-gradient} be denoted \((\mathbb{D}_{\Phi } \cdot ): \mathrm{dom}(\mathbb{D}_{\Phi } \cdot )\subseteq L^{2}(\mathfrak{X}^{1},\mathbb{R}^{n})\rightarrow L^{2}(\mathfrak{X}^{0},\mathbb{R}^{n})\) and the \textit{mixed-dimensional co-divergence} be denoted \(\mathbb{D}_{\Phi } :\mathrm{dom}(\mathbb{D}_{\Phi })\subseteq L^{2}(\mathfrak{X}^{n}, \mathbb{R})\rightarrow L^{2}(\mathfrak{X}^{n-1}, \mathbb{R})\), defined such that for \(\mathfrak{b} \in L^2 (\mathfrak{X}^{1},\mathbb{R}^{n}) \) and \(\mathfrak{c} \in L^2 (\mathfrak{X}^{n},\mathbb{R}) \),
\begin{subequations} \label{eqs: co-differentials}
\begin{align}
\left\langle\mathbb{D}_{\Phi }\cdot \mathfrak{ b,a}\right\rangle_{\mathfrak{X}^{0}} &  = -\left\langle\mathfrak{D}_{\Phi }\mathfrak{a,b}\right\rangle_{\mathfrak{X}^{1}}, 
&\forall\mathfrak{a} &\in \mathring{C} (\mathfrak{X}^{0},\mathbb{R}^{n}), \\ \
\left\langle\mathbb{D}_{\Phi }\mathfrak{c,a}\right\rangle_{\mathfrak{X}^{n-1}} &  = -\left\langle\mathfrak{D}_{\Phi }\cdot \mathfrak{ a,c}\right\rangle_{\mathfrak{X}^{n}}, 
& \forall \mathfrak{a} &\in \mathring{C}(\mathfrak{X}^{n-1}, \mathbb{R}).
\end{align}
\end{subequations}

As with the gradient and divergence, we will in later sections mostly omit the subscript \( \Phi\).
\end{definition}
The differential operators \(\mathbb{D}_{\Phi } \cdot \) and \(\mathbb{D}_{\Phi }\) coincide with the conventional divergence and gradient on the roots, complemented by so-called \textit{half-jump} operators on the boundaries \( \partial_{j}X_{i}\) that relate \( \iota_{i}\mathfrak{a}\) and \( \iota_{s_{j}}\mathfrak{a}\). We refer the interested reader to \cite{boon2021functional} for explicit representations.  
On the other hand, Def. \ref{def:2.18} defines the mixed-dimensional gradient and divergence on the more regular spaces \( C(\mathfrak{X}^{k},\mathbb{R}^{p})\). We do not wish to require such regularity in the weak formulation of our model and we therefore expand the definition. 
\begin{definition}\label{def:2.23}
Let the mixed-dimensional gradient and divergence with boundary conditions be given by
\begin{subequations}
\begin{align}
\mathfrak{\mathring{D}}_{\Phi }&  :\mathrm{dom}(\mathfrak{\mathring{D}}_{\Phi })\subseteq L^{2}(\mathfrak{X}^{0},\mathbb{R}^{n})\rightarrow L^{2}(\mathfrak{X}^{1},\mathbb{R}^{n}),  
& \mathfrak{\mathring{D}}_{\Phi }   &\coloneqq (-\mathbb{D}_{\Phi } \cdot )' \\
(\mathfrak{\mathring{D}}_{\Phi } \cdot )&  :\mathrm{dom}(\mathfrak{\mathring{D}}_{\Phi } \cdot )\subseteq L^{2}(\mathfrak{X}^{n-1}, \mathbb{R})\rightarrow L^{2}(\mathfrak{X}^{n-1}, \mathbb{R}),  
& (\mathfrak{\mathring{D}}_{\Phi } \cdot )   &\coloneqq (-\mathbb{D}_{\Phi })' 
\end{align}
\end{subequations}
\end{definition}
\begin{remark}
The circular accent on these operators indicates that the functions in the respective domains have vanishing trace on \( \partial_{Y}\mathfrak{X}^{k}\). This is a direct consequence of using \textit{test functions} \(\mathfrak{a}\in \mathring{C}(\mathfrak{X}^{k},\mathbb{R}^{p})\) in \eqref{eqs: co-differentials}. Moreover, the role of boundary conditions could have been reversed (and indeed generalized), but we will retain the choice implied above for simplicity of exposition. 
\end{remark}
To conclude this section, we emphasize that all differential operators (and the co-differentials) introduced herein are densely defined, unbounded linear operators mapping as \( L^{2}(\mathfrak{X}^{k},\mathbb{R}^{p})\rightarrow L^{2}(\mathfrak{X}^{k+1},\mathbb{R}^{p})\), cf. Appendix \ref{sec:appendix}. In particular, the density of \( C(\mathfrak{X}^{k}, \mathbb{R})\) in \( L^{2}(\mathfrak{X}^{k}, \mathbb{R})\) was shown in Theorem 3.1 of \cite{boon2021functional}. A consequence of this statement is that the spaces \(\mathrm{dom}(\mathbb{D}_{\Phi } \cdot )\), \(\mathrm{dom}(\mathbb{D}_{\Phi })\), \(\mathrm{dom}(\mathring{\mathfrak{D}}_{\Phi })\) and \(\mathrm{dom}(\mathring{\mathfrak{D}}_{\Phi } \cdot )\) are all Hilbert spaces with respect to their respective graph norms. 

\section{Mixed-dimensional strain measures}
\label{sec:mixed-dimentional-strain}

Scalar elliptic mixed-dimensional equations are well understood \cite{boon2021functional}, and the case of \( p = 1\) and \( k\in\left\{ n-1, n\right\}\) leads to the standard equations used for mixed-dimensional models of flow in fractured porous media \cite{boon2021functional, nordbotten2017modeling, boon2018robust}. The main outstanding challenge in constitutive modeling is thus the correct treatment of the mechanical deformation in the mixed-dimensional setting. This is the topic of this section. 

Our approach in this development is to follow the ``top-down" modeling associated with classical continuum mechanics, in the tradition of e.g.  \cite{hughes1983mathematical, coussy2005poromechanics, temam2005mathematical, truesdell2004non}, adapted to the mixed-dimensional geometry and spaces presented in Section \ref{sec:preliminaries}. Thus, we obtain a mixed-dimensional finite strain theory directly for the geometric representation \(\mathfrak{F}\). The converse approach, which we will not pursue in this work, would be to take the standard theory of mechanics as applied to the domain \( Y\) with its high-aspect inclusions \( \Psi_{i}\), and derive a finite strain theory for \(\mathfrak{F}\) through an upscaling based on the limit process of \(\mathcal{l}_{\epsilon }\rightarrow 0\). We will discuss the relationship between the results obtained in this work and classical theory as posed on \( Y\) in Section \ref{sec:governing_eq_linearized}. 

\subsection{Recollection of fixed-dimensional finite strain theory}
\label{sec:recollection_fixed_dimensional}
To provide context for the mixed-dimensional strain measure introduced later, we briefly recall the standard setting of finite strain theory. We recall from \eqref{eq:2.5} in Section \ref{sec:geometry} that for domains \( X\) and \( \Omega  = \phi(X)\), we denote by \(\mathbf{F} = D\phi\) the derivative of the \( C^{1}\) mapping \( \phi\). Then
\begin{definition}\label{def:3.1}
The right Cauchy-Green deformation tensor \(\mathbf{C} :T\Omega \rightarrow T\Omega\) is defined for a configuration \( \phi\) as 
\begin{equation}
    \mathbf{C}(\phi) \coloneqq \mathbf{F}^{T}\mathbf{F}.
\end{equation}
\end{definition}
Since we are only concerned with problems embedded in \(\mathbb{R}^{n}\) with Cartesian coordinates, we will in the following use the same notation for all associated tensors. However, we will not have need for the full generality of tensor calculus as all variables are defined on subsets of \(\mathbb{R}^{d_{j}}\). We will therefore not distinguish notationally between ``raising and lowering indexes". 

In our geometric setting, the reference domain \( X\) is without physical meaning, and we will be concerned with a time-dependent physical configurations, represented by \( \phi(t)\) and where the initial state is denoted \(\underline{\phi } \coloneqq \phi(t = 0)\). These are naturally compared on the reference domain \( X\), since the deformation tensor is rotationally invariant here. Thus we have 
\begin{definition}\label{def:3.2}
Green-Lagrange strain tensor with respect to the configurations \( \phi\) and \(\underline{\phi }\) is defined by the 2-tensor 
\begin{equation}
\mathbf{E}(t) \coloneqq \frac{1}{2}(\mathbf{C}(\phi(t))\mathbf{-C}(\underline{\phi }))
\end{equation}
\end{definition}

The normalization factor $\frac{1}{2}$ is somehow arbitrary, but is typically included to ensure that the linearized strain becomes dual to the divergence operator on symmetric tensor functions. If furthermore the deformation is infinitesimal from the baseline configuration \(\underline{\phi }\) , i.e. that \( \phi(t) =\underline{\phi }+u(t)\), and \(\mathbf{F} =\underline{\mathbf{F}}+Du\), with \(\left\vert Du\right\vert \ll 1\), in the case of \( d = n\) leads to
\begin{align*}
\mathbf{E}(t) &=\frac{1}{2}((\underline{\mathbf{F}} + Du)^{T}(\underline{\mathbf{F}} + Du)-\underline{\mathbf{F}}^{T}\underline{\mathbf{F}})\\
&=\frac{1}{2}(\underline{\mathbf{F}}^{T}Du+(Du)^{T}\underline{\mathbf{F}})+(Du)^{T}Du
\end{align*}
The linearized strain tensor is obtained by retaining the first-order terms in \(\left\vert Du\right\vert\), as summarized below.
\begin{definition}
The linearized strain tensor with respect to the deformation \( u(t) = \phi(t)-\underline{\phi }\) is defined by the 2-tensor 
\begin{equation}
\mathbf{e}(t) \coloneqq\frac{1}{2}(\underline{\mathbf{F}}^{T}Du(t)+(Du(t))^{T}\underline{\mathbf{F}}).
\end{equation}
When expressed as a linear operator on \( u(t)\), we refer to this operator as the symmetric gradient, and write
\begin{equation}
D_{s}u(t)\coloneqq \mathbf{ e}(t).
\end{equation}
\end{definition}
We remark that this definition simplifies whenever the reference configuration \( X\) is equal to the baseline physical configuration \( \Omega\), since in this case \(\underline{\phi }(x) = x\), and \(\underline{\mathbf{F}} = \underline{\mathbf{F}}^{T} =\mathbf{I}\). However, due to the nature of the mixed-dimensional geometries of interest herein, this will in general not be the case in our context. For example, fractures are not restricted to be located on the \( xy\)-plane, but the corresponding reference domains are.

\subsection{Mixed-dimensional finite strain}
\label{sec:mixed_dimensional_finite}

We follow the same approach to derive a mixed-dimensional finite strain theory. To proceed, we first make precise the needed extensions of fixed-dimensional calculus to the mixed-dimensional setting. In particular, we have already defined the mixed-dimensional differential operators in Section \ref{sec:differential_operators}. While these are in principle sufficient to obtain a mixed-dimensional strain, a richer strain notion can be obtained by also considering the derivative of the mixed-dimensional extended coordinate mappings, defined in Definitions \ref{def:2.6} and \ref{def:2.7}. In the same manner as above, we therefore introduce 
\begin{definition}\label{def:3.4}
The derivative of the mixed-dimensional extended deformation \(\hat{\Phi }\) is denoted \(\boldsymbol{\mathfrak{F}} \coloneqq D\hat{\Phi }\), and satisfies \( \iota_{j}\boldsymbol{\mathfrak{F}} = D\hat{\phi }_{j}\) for all \( j\in\mathfrak{F}\).
\end{definition}

Note that the forest \(\mathfrak{F}\) and the derivative \(\boldsymbol{\mathfrak{F}}\) should not be confused. 

It is an important point that the mixed-dimensional setting now deviates from the fixed-dimensional case, in that \( \mathfrak{D}\Phi\neq D\Phi \neq \mathbb{D}\Phi\). These notions of a derivative of the configuration \( \Phi\) (all in a sense ``gradients"), have important distinctions. The derivative of the deformation \(\boldsymbol{\mathfrak{F}}\), defined on \(\mathfrak{X}\), contains information of the deformation of each \( \Omega_{i}\), but has no information about the relative placements of domains. On the other hand, the mixed-dimensional gradient \(\mathfrak{D}\Phi\), defined on \(\mathfrak{X}^{1}\), contains information on relative placements (due to the jump operator \(\mathbbm{d}\)), but only contains information regarding the deformation of the top-dimensional domains, i.e. those domains \( \Omega_{i}\) where \( d_{i} = n\). It is therefore clear that \(\mathfrak{D}\Phi\) contains the desired physical information (since the lower-dimensional domains are fractures – voids – and their precise deformation is immaterial). Conversely, the deformation \(\boldsymbol{\mathfrak{F}}\) is required for coordinate transformations, and can be thought of as a ``fabric" onto which to project vectors. 

The above discussion suggests: 
\begin{definition}\label{def:3.5}
The mixed-dimensional right Cauchy-Green deformation tensor is defined for a configuration \( \Phi\) as
\begin{equation}
\boldsymbol{\mathfrak{C }} \coloneqq (\Pi^{1}\boldsymbol{\mathfrak{F}}^{T})\mathfrak{D}(\Pi^{0}\Phi)
\end{equation}
where \( \Pi^{k}\) is the restriction from \(\mathfrak{X}\) to \(\mathfrak{X}^{k}\).
\end{definition}
\begin{remark}
The deformation \(\boldsymbol{\mathfrak{F}}\) based on the extended mappings \(\hat{\phi }\), allows for transforming vectors (and thus forces) in \(\mathbb{R}^{n}\) appropriately. Alternative suggestions for a ``symmetric" deformation tensor, such as e.g. expressions of the type \(\boldsymbol{\mathfrak{F}}^{T}\boldsymbol{\mathfrak{F}}\) or \((\mathfrak{D}(\Pi^{0}\Phi))^{T}\mathfrak{D}(\Pi^{0}\Phi)\), can be seen to be unsuitable, as the former contains no information of relative placements of domains, while the latter only retains the magnitude of displacements across a fracture, without orientation information. 
\end{remark}

Proceeding as in Section \ref{sec:recollection_fixed_dimensional}, we will use the mixed-dimensional right Cauchy-Green deformation tensor as the basis for defining a strain measure on the reference domain. We therefore consider a time-dependent mapping \( \Phi  = \Phi(t)\) from which we obtain a time-dependent deformation tensor \(\boldsymbol{\mathfrak{C}}(t)\). By again identifying time \( t = 0\) as the reference time with \(\underline{\Phi } \coloneqq \Phi(t = 0)\), then
\begin{definition}\label{def:3.7}
The mixed-dimensional Green-Lagrange strain tensor with respect to the configurations \( \Phi\) and \(\underline{\Phi }\) is defined by 
\begin{equation}
\mathfrak{E}(t) \coloneqq \varrho(\boldsymbol{\mathfrak{C}}(t)-\underline{\boldsymbol{\mathfrak{C}}}),
\end{equation}
where the mixed-dimensional gradients \(\mathfrak{D}_{t}\) are evaluated based on the configuration at time \( t\), such that in particular, \(\underline{\boldsymbol{\mathfrak{C}}} = \underline{\boldsymbol{\mathfrak{F}}}^T\mathfrak{D}_{\Phi(t)}\underline{\Phi }\). Moreover, the normalization factor \( \varrho\) is assigned the value \( \iota_{j}\varrho  =\frac{1}{2}\) for \( j\in\mathfrak{S}_{i}\) and \( i\in I^{n}\) and the value \( \iota_{j}\varrho  = 1\) otherwise. The justification for this choice will become apparent in Lemma \ref{lemma:5.1}.
\end{definition}

\begin{example}\label{eg:3.8}
We consider the interpretation of the mixed-dimensional Green-Lagrange strain tensor on domains of various dimensionality: 
\begin{enumerate} \small
    \item For top-dimensional domains, \( i\in I^{n}\), then as in the fixed-dimensional case,
    \begin{equation}
    \label{eq:3.8}
    \iota_{i}\mathfrak{E}(t) =\mathbf{E}_{i}(t).
    \end{equation}
    \item On the boundaries of the top-dimensional domains \( i\in\mathfrak{S}_{j}\), where \( j\in I^{n}\), the deformation tensor is given by \( \iota_{i}\boldsymbol{\mathfrak{C}} = (\hat{\mathbf{F}}_{i})^{T}\mathbf{F}_{i}\) with \(\hat{\mathbf{F}}_{i} \coloneqq D\hat{\phi }_{i}\). It is thus represented by a \(\mathbb{R}^{n}\times\mathbb{R}^{d_{i}}\) matrix. Then the strain takes the form
    \begin{equation}
    \iota_{i}\mathfrak{E}(t) =\frac{1}{2}(\hat{\mathbf{F}}_{i}^{T}(t)\mathbf{F}_{i}(t)-\hat{\underline{\mathbf{F}}}_{i}^{T}\underline{\mathbf{F}}_{i}) =\frac{1}{2}\binom{\mathbf{F}_{i}^{T}(t)\mathbf{F}_{i}(t)-\underline{\mathbf{F}}_{i}^{T}\underline{\mathbf{F}}_{i}}{\mathbf{0}}.
   \end{equation}

Note that due to Definition \ref{def:2.7}, the ``extended" components of \(\hat{\mathbf{F}}_{i}\) are orthogonal to \(\mathbf{F}_{i}\) (whose columns are vectors in the tangent space \( T\Omega_{i}\)), thus the last \( n-d_{i}\) rows of \( \iota_{i}\mathfrak{E}(t)\) are identically zero, justifying the claim that the precise choice of extensions in Definition \ref{def:2.7} is immaterial for the developments.

    \item For domains \( i\in I^{n-1}\) (the fractures), the mixed-dimensional gradient of the deformation is simply \( \iota_{i}\mathfrak{D}\Phi  = \iota_{i}\mathbbm{d}_{t}\Phi\), i.e. the jump in \( \phi_{j}\) between the two \( n\)-dimensional neighbors to \( \Omega_{i}\). Thus
    \begin{equation*}
        \iota_{i}\mathfrak{E}(t) =\hat{\mathbf{F}}_{i}^{T}(t)(\iota_{i}\mathbbm{d}_{t}\Phi(t))-\hat{\underline{\mathbf{F}}}_{i}^{T}(\iota_{i}\mathbbm{d}_{t}\underline{\Phi }).
    \end{equation*}
        As above, it is natural to decompose it into its parallel and normal components, denoted by subscripts \( \Vert \) and \( \perp\), respectively, which takes the form
\begin{equation}
\iota_{i}\mathfrak{E}(t) = \begin{bmatrix}
(\iota_{i}\mathfrak{E}(t))_{\parallel} \\ 
(\iota_{i}\mathfrak{E}(t))_{\perp} \\ 
\end{bmatrix} =\begin{bmatrix}
\mathbf{F}_{i}^{T}(t)(\iota_{i}\mathbbm{d}_{t}\Phi(t)) \\ 
\hat{\mathbf{F}}_{i,n}^{T}(t)(\iota_{i}\mathbbm{d}_{t}\Phi(t)) \\ 
\end{bmatrix}-\begin{bmatrix}
\underline{\mathbf{F}}_{i}^{T}(\iota_{i}\mathbbm{d}_{t}\underline{\Phi }) \\ 
\hat{\underline{\mathbf{F}}}_{i,n}^{T}(\iota_{i}\mathbbm{d}_{t}\underline{\Phi }) \\ 
\end{bmatrix}.
\end{equation}
Here we denote the \( n\)\textsuperscript{th} row of \(\hat{\mathbf{F}}_{i}^{T}(t)\) by \(\hat{\mathbf{F}}_{i,n}^{T}(t)\), which we note equals \(\hat{\mathbf{F}}_{i,n}^{T}(t) =\mathcal{l}^{-1}_i\mathbf{n}_{i}^{T}(t)\), where \(\mathbf{n}_{i}\) is the normal vector orthogonal to \( \Omega_{i}\) preserving the orientation of \(\hat{\phi }_{i}\). Therefore, the expression for the strain can be simplified to
\begin{equation} \label{eq: 3.11}
\iota_{i}\mathfrak{E}(t) =\begin{bmatrix}
\mathbf{F}_{i}^{T}(t)(\iota_{i}\mathbbm{d}_{t}\Phi(t))_{\parallel} \\ 
\mathcal{l}_{i}^{-1}(\iota_{i}\mathbbm{d}_{t}\Phi(t))_{\perp } \\ 
\end{bmatrix}-\begin{bmatrix}
\underline{\mathbf{F}}_{i}^{T}(\iota_{i}\mathbbm{d}_{t}\underline{\Phi })_{\parallel} \\ 
\mathcal{l}_{i}^{-1}(\iota_{i}\mathbbm{d}_{t}\underline{\Phi })_{\perp } \\ 
\end{bmatrix}.
\end{equation}
Here we have decomposed the displacement jump into its orthogonal and parallel components,
\begin{align*}
(\iota_{i}\mathbbm{d}_{t}\underline{\Phi })_{\perp } &\coloneqq \underline{\mathbf{n}}_{i}^{T}(\iota_{i}\mathbbm{d}_{t}\underline{\Phi })
& &\text{and} &
(\iota_{i}\mathbbm{d}\Phi_{0})_{\parallel} &\coloneqq \iota_{i}\mathbbm{d}\Phi_{0}-\underline{\mathbf{n}}_{i}(\iota_{i}\mathbbm{d}\Phi_{0})_{\perp }.
\end{align*}

Moreover, by the definition of the jump operator, the jump in the direction parallel to the fracture is identically zero, \((\iota_{i}\mathbbm{d}_{t}\Phi(t))_{\parallel} = 0\), and this term can be omitted from \eqref{eq: 3.11}. We furthermore note that by the mixed-dimensional continuum assumption the jump in the direction perpendicular to the fracture is of order \(\mathcal{l}_{\epsilon }\), thus \(\mathcal{l}_{i}^{-1}(\iota_{i}\mathbbm{d}_{t}\Phi(t))_{\perp } = \mathcal{O}(1)\)\textbf{.}  In contrast, sliding is measured as \((\iota_{i}\mathbbm{d}_{t}\underline{\Phi })_{\parallel}\), which measures the slip of the two fracture surfaces from the initial state until the current configuration. We thus arrive at the final expression for the strain in fractures,
\begin{equation}
 \iota_{i}\mathfrak{E}(t) =\begin{bmatrix}
-\underline{\mathbf{F}}_{i}^{T}(\iota_{i}\mathbbm{d}_{t}\underline{\Phi })_{\parallel} \\ 
\mathcal{l}_{i}^{-1}((\iota_{i}\mathbbm{d}_{t}\Phi(t))_{\perp }-(\iota_{i}\mathbbm{d}_{t}\underline{\Phi })_{\perp }) 
\end{bmatrix}.
\end{equation}
    \item  Since \(\mathfrak{D}\Phi\) is void on domains \( \Omega_{j}\) with \( j\in I^{d<n-1}\), so is \( \iota_{i}\mathfrak{E}(t)\) for all \( i\in\mathfrak{S}_{j}\).
\end{enumerate}
\end{example}
Again, we emphasize that the measure of opening of a fracture has arbitrary scale, depending on the choice of \(\mathcal{l}_{i}\). This implies that \(\mathfrak{E}(t)\) is a multi-scale strain measure, which we will return to in Section \ref{sec:mixed-dimentional-poromechanics} (Example \ref{eg:4.4}). We close this section by verifying that the mixed-dimensional finite strain \(\mathfrak{E}(t)\) is rotationally and translationally invariant. 
\begin{lemma}
\label{lemma:3.1}
Let \( \Phi(t)\) be a rigid body motion relative to \(\underline{\Phi }\). Then \(\mathfrak{E}(t) = 0\). 
\begin{proof}
A rigid body motion can be described by a rotation matrix \( R(t)\) and a vector \( V(t)\), both independent of space and the rotation satisfying \( R^{-1}(t) = R^{T}(t)\). Then \( \Phi(t) = R(t)\underline{\Phi }+V(t)\), i.e. for all \( i\in\mathfrak{F}\) the local mapping is given by 
\begin{equation*}
\iota_{i}\Phi(t) = R(t)\phi_{0, i}(0)+V(t).
\end{equation*}
Then since differentiation is a linear operator with constants in its null-space, we have both \(\mathbf{F}_{i}(t) = R(t)\underline{\mathbf{F}}_{i}\) and \(\hat{\mathbf{F}}_{i}(t) = R(t)\underline{\hat{\mathbf{F}}}_{i}\), while by the same argument the jump operator satisfies \( \iota_{i}\mathbbm{d}_{t}\Phi(t) = \iota_{i}\mathbbm{d}_{t}(R(t)\underline{\Phi }) = \iota_{i}R(t)\mathbbm{d}_{t}\underline{\Phi }\).

Now a direct substitution gives
\begin{equation*}
\hat{\mathbf{F}}_{i}^{T}(t)\mathbf{F}_{i}(t) =(R(t)\underline{\hat{\mathbf{F}}}_{i})^{T}R(t)\underline{\mathbf{F}}_{i} 
 =\underline{\hat{\mathbf{F}}}_{i}R^{T}(t)R(t)\underline{\mathbf{F}}_{i} =\underline{\hat{\mathbf{F}}}_{i}\underline{\mathbf{F}}_{i}
\end{equation*}
and 
\begin{equation*}
\begin{split}
\hat{\mathbf{F}}_{i}^{T}(t)(\iota_{i}\mathbbm{d}_{t}\Phi(t)) &=(R(t)\underline{\hat{\mathbf{F}}}_{i})^{T}(R(t)\iota_{i}\mathbbm{d}_{t}\underline{\Phi })\\
& =\underline{\hat{\mathbf{F}}}_{i}R^{T}(t)R(t)(\iota_{i}\mathbbm{d}_{t}\underline{\Phi }) = \underline{\hat{\mathbf{F}}}_{i}(\iota_{i}\mathbbm{d}_{t}\underline{\Phi })
\end{split}
\end{equation*}
Comparison with the local expressions for \( \iota_{i}\mathfrak{E}(t)\) provided in Example \ref{eg:3.8} verifies the lemma.
\end{proof}
\end{lemma}

\subsection{Mixed-dimensional linearized strain}
\label{sec:mixed_dimensional_linearized}

When considering a constitutive theory, our primary variable will be the displacement of the top-dimensional domains \(\mathfrak{u}\in L^{2}(\mathfrak{X}^{0},\mathbb{R}^{n})\). To be able to calculate the linearized strain in the remaining, lower-dimensional subdomains, we require an extension operator onto the domain of \( \Phi\), which we define as: 
\begin{definition}
\label{def:3.10}
A bounded linear operator \( \Xi  : L^{2}(\mathfrak{X}^{0},\mathbb{R}^{n})\rightarrow L^{2}(\mathfrak{X},\mathbb{R}^{n})\) is an \(\mathfrak{X}^{0}\)-\textit{extension} operator if it is a right-inverse of the restriction \( \Pi^{0}\), i.e. 
\begin{align}
\Pi^{0}\Xi\mathfrak{u} &\coloneqq \mathfrak{u}, & 
\forall \mathfrak{u} &\in L^{2}(\mathfrak{X}^{0},\mathbb{R}^{n}).
\end{align}
\end{definition}
\begin{remark}
We note that the most natural choice for \( \Xi\) is an averaging operator, such that for a fracture \( i\in I^{2}\), its displacement \( \iota_{i}(\Xi\mathfrak{u})\) is defined as the average displacement of the rock on the two sides. Such operators allow us to consider the representation as the mean of neighboring (rock) positions in \( I^{n}\), such that for any \( i\in I^{d<n}\) 
\begin{equation*}
\iota_{i}(\Xi\mathfrak{u}) =\frac{1}{\left\vert\hat{I}_{i}^{n}\right\vert }\sum_{j\in\hat{I}_{i}^{n}}^{}\phi_{j,i}^{\ast }(\iota_{j}\mathfrak{u}).
\end{equation*}
We study the role of extension operators in more detail in Section \ref{sec:mixedd_dimensional_volume}.
\end{remark}
We obtain a linearized strain by considering deformations \( \Phi(t)\) such that \( \Phi(t) = \Xi\mathfrak{u}(t)+\underline{\Phi }\), and where the mixed-dimensional gradients in \(\mathfrak{Du}\) are small in the sense that for all \( i\in\mathfrak{F}\), and for all \( x\in X_{i}\), it holds that 
\begin{equation}\label{eq:3.16}
\Vert(\iota_{i}\mathfrak{D}\mathfrak{u})(x)\Vert \ll \Vert\underline{\mathbf{F}}_{i}(x)\Vert.
\end{equation}
Using this we define the linearized strain as \(\mathfrak{e}(t)\), obtained by omitting ``small" terms. More precisely, we define the linearized strain as the Fréchet derivative of the finite strain in the following sense: 
\begin{definition}\label{def:3.11}
For some initial mapping \(\underline{\Phi }\), and some deformation \(\mathfrak{u}\in C(\mathfrak{X}^{0},\mathbb{R}^{n})\), let the \textit{mixed-dimensional linearized strain} be defined as
\begin{equation}
\mathfrak{e}(\mathfrak{u}) \coloneqq D\mathfrak{E}(\underline{\Phi })(\Xi\mathfrak{u}),
\end{equation}
where \( D\mathfrak{E}(\underline{\Phi })(\Psi)\) is the Fréchet derivative of \(\mathfrak{E}\) at \(\underline{\Phi }\) acting on the perturbation \( \Psi\). When expressed as a linear operator on \(\mathfrak{u}(t)\), we refer to this operator as the \emph{symmetric gradient}. Thus \(\mathfrak{D}_{s} : C^{1}(\mathfrak{X}^{0},\mathbb{R}^{n})\rightarrow L^{2}(\mathfrak{X}^{1},\mathbb{R}^{n})\) is defined as
\begin{equation}
\mathfrak{D}_{s}\mathfrak{u} \coloneqq \mathfrak{e}(\mathfrak{u}).
\end{equation}
\end{definition}
\begin{example}\label{eg:3.12}
Continuing from Example \ref{eg:3.8}, we consider the interpretation of the mixed-dimensional linearized strain tensor on domains of various dimensionality, keeping in mind that \(\underline{\Phi }\) is the unstrained state, i.e.  \(\mathfrak{E}(\underline{\Phi }) = 0\).
\begin{enumerate} \small
    \item For top-dimensional domains, \( i\in I^{n}\), then as in the fixed-dimensional case, the linearized strain tensor takes the form
    \begin{equation}
    \iota_{i}\mathfrak{e}(\mathfrak{u}) =\frac{1}{2}(\underline{\mathbf{F}}_{i}^{T}Du_{i}(t)+(Du_{i}(t))^{T}\underline{\mathbf{F}}_{i}).
    \end{equation}
    \item On the boundaries of the top-dimensional domains \( i\in\mathfrak{S}_{j}\), where \( j\in I^{n}\), the linearized strain tensor is represented by a \(\mathbb{R}^{n}\times\mathbb{R}^{d_{i}}\) matrix. It has the explicit form
    \begin{equation}
        \iota_{i}\mathfrak{e}(t) =\frac{1}{2}(\hat{\underline{\mathbf{F}}}_{i}^{T}Du_{i}(t)+(D\hat{u}_{i}(t))^{T}\underline{\mathbf{F}}_{i}) =\frac{1}{2}\begin{pmatrix}
        \underline{\mathbf{F}}_{i}^{T}Du_{i}(t)+(Du_{i}(t))^{T}\underline{\mathbf{F}}_{i} \\ 
        0 
        \end{pmatrix}.
        \end{equation}
As in the case of the finite strain, the last \( n-d_{i}\) rows of \( \iota_{i}\mathfrak{e}(t)\) are identically zero.

    \item For domains \( i\in I^{n-1}\) (the fractures), we calculate the Fréchet derivative as 
\begin{equation}
\begin{split}
\iota_{i}(D\mathfrak{E}(\underline{\Phi }))(\Psi) & = \lim_{\epsilon \rightarrow 0}\epsilon^{-1}(\iota_{i}\mathfrak{E}(\underline{\Phi }+\epsilon \Psi)-\iota_{i}\mathfrak{E}(\underline{\Phi }))\\ 
& = \lim_{\epsilon \rightarrow 0}\frac{1}{\epsilon }\begin{bmatrix}
-\underline{\mathbf{F}}_{i}^{T}(\iota_{i}\mathbbm{d}_{\underline{\Phi }+\epsilon \Psi }\underline{\Phi })_{\parallel} \\ 
\mathcal{l}_{i}^{-1}((\iota_{i}\mathbbm{d}_{\underline{\Phi }+\epsilon \Psi }(\underline{\Phi }+\epsilon \Psi))_{\perp }-(\iota_{i}\mathbbm{d}_{\underline{\Phi }+\epsilon \Psi }\underline{\Phi })_{\perp }) 
\end{bmatrix}\\ 
&  = \lim_{\epsilon \rightarrow 0}\frac{1}{\epsilon }\begin{bmatrix}
-\underline{\mathbf{F}}_{i}^{T}(\iota_{i}\mathbbm{d}_{\underline{\Phi }+\epsilon \Psi }\underline{\Phi })_{\parallel} \\ 
\epsilon\mathcal{l}_{i}^{-1}(\iota_{i}\mathbbm{d}_{\underline{\Phi }+\epsilon \Psi }\Psi)_{\perp }
\end{bmatrix} =\begin{bmatrix}
\underline{\mathbf{F}}_{i}^{T}(\iota_{i}\mathbbm{d}_{\underline{\Phi }}\Psi)_{\parallel} \\ 
\mathcal{l}_{i}^{-1}(\iota_{i}\mathbbm{d}_{\underline{\Phi }}\Psi)_{\perp }
\end{bmatrix}.
\end{split}
\end{equation}
In the final line, we have used the continuity of the jump operator, as elaborated in Remark \ref{remark:2.5} . Substituting in the extended deformation \( \Xi\mathfrak{u}\), we now obtain
\begin{equation}\label{eq:3.22}
 \iota_{i}\mathfrak{e}(\mathfrak{u}) = \iota_{i}(D\mathfrak{E}(\underline{\Phi }))(\Xi\mathfrak{u}) =\begin{bmatrix}
\underline{\mathbf{F}}_{i}^{T}(\iota_{i}\mathbbm{d}_{\underline{\Phi }}\Xi\mathfrak{u})_{\parallel} \\ 
\mathcal{l}_{i}^{-1}(\iota_{i}\mathbbm{d}_{\underline{\Phi }}\Xi\mathfrak{u})_{\perp } \\ 
\end{bmatrix} = \begin{bmatrix}
\underline{\mathbf{F}}_{i}^{T}(\iota_{i}\mathbbm{d}_{\underline{\Phi }}\mathfrak{u})_{\parallel} \\ 
\mathcal{l}_{i}^{-1}(\iota_{i}\mathbbm{d}_{\underline{\Phi }}\mathfrak{u})_{\perp } \\ 
\end{bmatrix}.
\end{equation}
Here the extension operators vanish since they are identity operators on the top-level domains (by Definition \ref{def:3.10}). 

    \item As in the finite deformation case, the linearized strain is void for \( i\in I^{d<n-1}\).
\end{enumerate}
\end{example}
\begin{remark}
Example \ref{eg:3.12} illustrates that the extension \( \Xi\) plays no role in the final expressions for the linearized strain. However, this situation changes when considering the linearization of volumetric strain in the next section. Secondly, we emphasize the trivial (but sometimes forgotten) fact that while the finite strain is rotationally invariant, its linearization is not. The importance in deriving a linearized strain from a rotationally invariant quantity is thus to ensure consistency in the limit of small deformations. 
\end{remark}
\begin{remark}\label{remark:3.4}
It is an important detail that while in the finite strain case the differential operators are time-dependent via their dependence on the jump operator \(\mathbbm{d}_{\Phi(t)}\), it is clear from Definition \ref{def:3.11} and Example \ref{eg:3.12} (see e.g. \eqref{eq:3.22}), that the differential operator in the linearized strain is evaluated at \(\underline{\Phi }\), and thus not time-dependent. 
\end{remark}
\begin{figure}[!htbp]
\includegraphics[width=\textwidth]{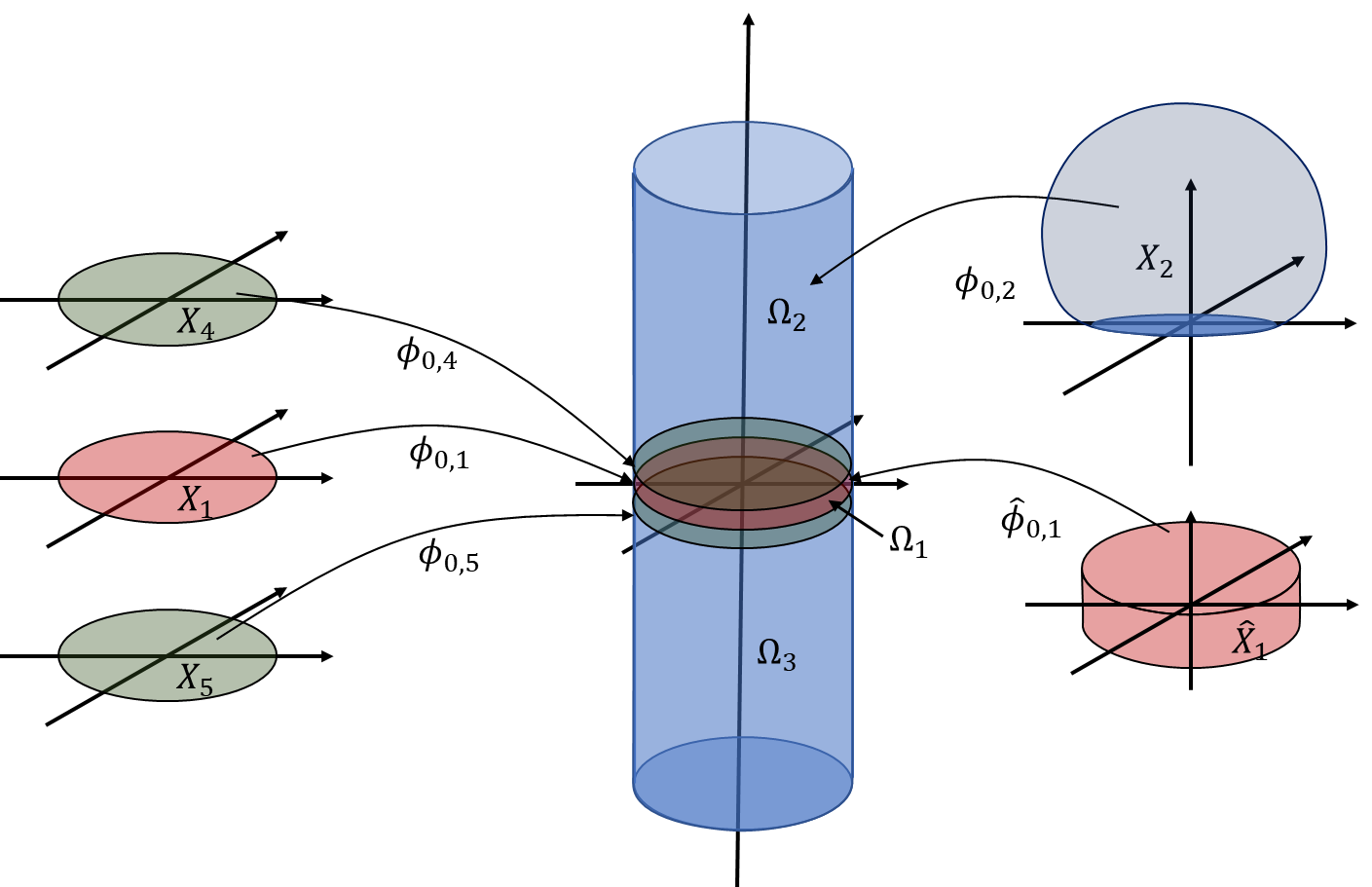}
\caption{Geometry considered in Example \ref{eg:3.14}, where all mappings can be chosen such that \(\phi_{0,j}(0)=0\) and that \((F_j)_{\Vert} = I\) near the origin.}
\label{fig:geometry}
\end{figure}

We give a second example to be more concrete. 
\begin{example}\label{eg:3.14}
Consider a circular fracture defined by the unit disc in the plane as illustrated in Figure \ref{fig:pullback-of-function}. We choose an initial mapping that is the identity mapping near the origin, i.e. \( X_{1} = B^{2}(x)\subset\mathbb{R}^{2}\), such that \(\underline{\phi }_{0,1}(x) =\left[x, 0\right]\in\mathbb{R}^{3}\) and \( \Omega_{1} =\underline{\phi }_{0,1}(X_{1})\). Let the domains ``above" and ``below" be enumerated \( 2, 3\) with mappings \( \phi_{0, 2}(x) = x\) and \( \phi_{0, 3} = x\) on their respective domains \( X_{2}\) and \( X_{3}\), and let the extended mapping of \( X_{1}\) be defined such that \(\underline{\hat{\phi }}_{0,1} =\left[x, \mathcal{l}^{-1}_1 y \right]\) for  \((x,y)\in X_{1}\times [-\epsilon ,\epsilon ]\) for some \( \epsilon >0\).  Then on the fracture, \( x\in X_{1}\), the fully linearized strain is simply 
\begin{equation}
\iota_{1}\mathfrak{e}(t;x) =\begin{pmatrix}
\iota_{2}\mathfrak{u}_{\Vert  }(t;x)-\iota_{3}\mathfrak{u}_{\Vert   }(t;x) \\ 
\mathcal{l}_{1}^{-1}(\iota_{2}\mathfrak{u}_{\perp }(t;x)-\iota_{3}\mathfrak{u}_{\perp }(t;x)) \\ 
\end{pmatrix}\in\mathbb{R}^{3},
\end{equation}
while on the lower boundary of the fracture, indexed by say \( j = 5\) such that \( s_{5} = 1\) and with \(\underline{\hat{\phi }}_{0,5} =\underline{\hat{\phi }}_{0,1}\), the strain is 
\begin{equation}
\iota_{5}\mathfrak{e}(t;x) =\frac{1}{2}((D\iota_{5}\mathfrak{u})^{T} + D\iota_{5}\mathfrak{u}) = \frac{1}{2}\begin{bmatrix}
(D_{\Vert  }\iota_{5}\mathfrak{u}_{\Vert  })^{T}+D_{\Vert  }\iota_{5}\mathfrak{u}_{\Vert  } \\ 
0  
\end{bmatrix}.
\end{equation}
Thus \( \iota_{5}\mathfrak{e}(t)\) is represented by a \( 3\times 2\) matrix similar to the horizontal components of the linearized strain. We note that when seen together, the fracture strain and (either of) the boundary strains can be combined and considered as a representation of the full strain. 

For the interior of the matrix, we recover the standard linearized strain, such as in e.g. \( X_{3}\)
\begin{equation}
\iota_{3}\mathfrak{e}(t;x) =\frac{1}{2}((D\iota_{3}\mathfrak{u})^{T} + D\iota_{3}\mathfrak{u}).
\end{equation}
\end{example}

\subsection{Mixed-dimensional volume measure and the matrix trace operator}
\label{sec:mixedd_dimensional_volume}

We will see in the continuation that flow is naturally formulated with pressures in \( L^{2}(\mathfrak{X}^{n}, \mathbb{R})\) and fluxes in \( L^{2}(\mathfrak{X}^{n-1}, \mathbb{R})\) whereas the mechanics is naturally formulated with displacements in \( L^{2}(\mathfrak{X}^{0},\mathbb{R}^{n})\) and strains in \( L^{2}(\mathfrak{X}^{1},\mathbb{R}^{n})\), cf. Examples \ref{eg:2.3} and \ref{eg:2.3b}. In order to develop appropriate coupling between flow and mechanics, we will also need operators mapping between these spaces, particularly to capture the effect of volume changes in the lower-dimensional domains (and conversely, the impact of fluid pressure on the total stress). As in the preceding sections, we present the volume measure in the setting of finite deformation first, and subsequently its linearization. 

\subsubsection{Finite deformation volumetric strain}
\label{sec:finite_deformation}

By continuity, the determinant of the derivative of the transformation \( J \coloneqq \det(D\hat{\Phi })\) contains the volume of the physical configuration relative to the volume of the reference domains. This is sufficient for the top-dimensional domains, \( i\in I^{n}\), however, for lower-dimensional domains \( i\in I^{d<n}\), it is of interest to include not only the static weight \(\mathcal{l}_{i}\), but also the change associated with the jump in displacement. We therefore define the mixed-dimensional volume \(\mathfrak{J}\) as follows. 
\begin{definition}\label{def:3.9}
Let \( \omega_{i}^{j}\in C^{1}(X_{i})\) be a set of nonnegative weights for \( i\in I^{d<n}\) and \( j\in\mathfrak{F}\) with \( s_{j} = i\). The \textit{mixed-dimensional volume density} \(\mathfrak{J}(\Phi)\) is defined such that:
\begin{enumerate} \small
    \item \label{def: 3.9.1} For \( i\in I^{n}\), then \( \iota_{i}\mathfrak{J}(\Phi) = \mathrm{vol}(\mathbf{F}_{i})\), defined in \eqref{eq: definition vol}.
    \item \label{def: 3.9.2} For \( i\in I^{d<n-1}\), then 
\begin{equation}
\iota_{i}\mathfrak{J}(\Phi) =\mathcal{l}_{i}^{n-d_{i}}\left(1+\sum_{j\in J_{i}^{n-1}}^{}\omega_{i}^{j}\mathcal{l}_{j}^{-1}\phi_{j, i}^{\ast }(\iota_{j}\mathbbm{d}_{\Phi }\Phi)_{\perp }\right)\mathrm{vol}(\mathbf{F}_{i}),\\ 
\end{equation}
where \( J_{i}^{n-1}\) is the set of indexes \( j\in \bigcup_{l\in I^{n-1}}^{}\mathfrak{S}_{l}\) such that \( s_{j} = i\).

    \item \label{def: 3.9.3} For \( i\in I\), and \( j\in I_{i}\), then \( \iota_{j}\mathfrak{J}(\Phi) = 0\).

\end{enumerate}
\end{definition}
The above definition is motivated as follows (confer also Figure \ref{fig:illustration-of-multi-scale-contact-mechanics}): As stated in the mixed-dimensional continuum assumption, Definition \ref{def:2.2}, it is natural to consider that the idealization of the fracture has some effective opening \(\mathcal{l}_{i}\), and that in general its volume per unit area changes linearly with perpendicular opening (or closing) \( \iota_{i}(\mathbbm{d}\Phi)_{\perp }\) according to a proportionality constant \( \omega_{i}^{i}\). Similarly, the cross-sectional area associated with an intersection will have some lower limit \(\mathcal{l}_{i}^{n-d_{i}}\), and change proportionally to the opening of nearby fractures meeting at that intersection, according to weights \( \omega_{i}^{j}\). We recognize that the definition stated in point \ref{def: 3.9.1} above can be written as special case of point \ref{def: 3.9.2} (by introducing the convention that \(\mathcal{l}_{i} = 1\) for \( i\in I^{n}\)), but retain separate definitions for pedagogical clarity. Finally, it is typically not relevant to consider volume changes of the surfaces (point \ref{def: 3.9.3}). 
\begin{figure}[!htbp]
\includegraphics[width=\textwidth]{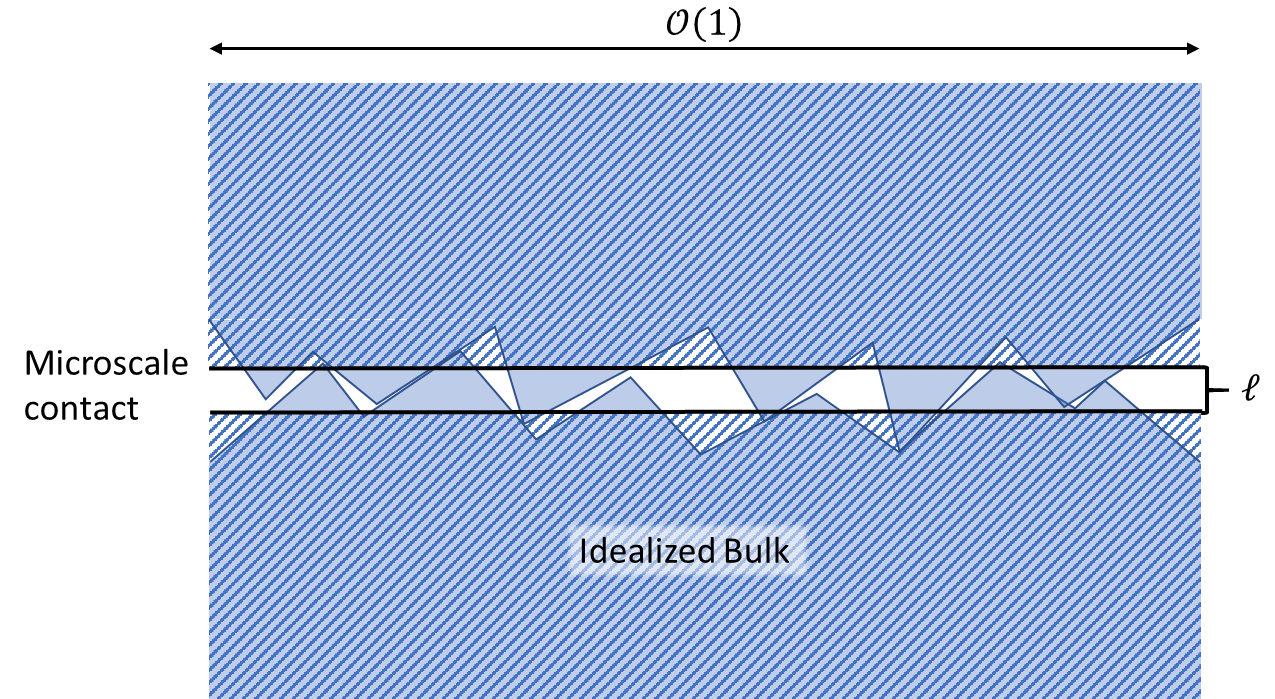}
\caption{Illustration of multi-scale contact mechanics, adapted from Oden and Martins \cite{oden1985models}}
\label{fig:illustration-of-multi-scale-contact-mechanics}
\end{figure}

\subsubsection{Linearized volume change}
\label{sec:linearized_vol}
We briefly recall under the small deformation assumption in fixed-dimensional continuum mechanics, the linearization of the volumetric change of a deformation becomes the matrix trace. To be precise, let us as in Section \ref{sec:mixed_dimensional_linearized} consider a small deformation \( \Phi(t)\) such that \( \Phi(t) = \Xi\mathfrak{u}(t)+\underline{\Phi }\), where \eqref{eq:3.16} holds. Then since the determinant commutes with the matrix product, for any \( i\in I^{n}\) we have the relationship (recall that in the case when \( d_{i} = n\), then \( \mathrm{vol} = \det\), and for \( d_{i}<n\), we have \( \mathrm{vol}(\mathbf{F}_{i})^{2} = \det(\mathbf{F}_{i}^{T}\mathbf{F}_{i})\)):
\begin{equation}
\begin{split}
\frac{J_{i}}{\underline{J}_{i}} &=\frac{\det(\mathbf{F}_{i})}{\det(\underline{\mathbf{F}}_{i})} = \det(\underline{\mathbf{F}}_{i}^{-1}\mathbf{F}_{i}) = \det(\underline{\mathbf{F}}_{i}^{-1}D(\underline{\phi }_{i}+u_{i})) = \det(I+\underline{\mathbf{F}}_{i}^{-1}Du_{i})\\
&= \det(I+\underline{\mathbf{C}}_{i}^{-1}\underline{\mathbf{F}}_{i}^{T}Du_{i})\\
& = 1+\text{ Trace }(\underline{\mathbf{C}}_{i}^{-1}\underline{\mathbf{F}}_{i}^{T}Du_{i})+\mathcal{O}(\left\vert Du\right\vert^{2})\\
&= 1+\text{ Trace }(\underline{\mathbf{C}}_{i}^{-1}\mathbf{e}(u_{i}))+\mathcal{O}(\left\vert Du\right\vert^{2}).
\end{split}
\end{equation}
where we see that the linear term is the trace of the linearized strain, scaled by the deformation tensor of the undeformed state. Thus
\begin{equation}
\frac{(DJ_{i})(u_{i})}{\det(\underline{\mathbf{F}}_{i})} = \text{ Trace }(\underline{\mathbf{C}}_{i}^{-1}\mathbf{e}(u_{i})).
\end{equation}

In the same spirit, we first consider the linearization of the mixed-dimensional volume change, and secondly identify its interpretation as a matrix trace.  
\begin{definition}
For some initial mapping \(\underline{\Phi }\), and some deformation \(\mathfrak{u}\in C(\mathfrak{X}^{0},\mathbb{R}^{n})\), the \textit{linearized mixed-dimensional volume change} \(\mathfrak{j}\in L^{2}(\mathfrak{X}, \mathbb{R})\) is defined as
\begin{equation}
\mathfrak{j}(\mathfrak{u}) \coloneqq D\mathfrak{J}(\underline{\Phi })(\Xi\mathfrak{u}),
\end{equation}
where \( D\mathfrak{J}(\underline{\Phi })(\Psi)\) is the derivative of \(\mathfrak{J}\) at \(\underline{\Phi }\) for the perturbation \( \Psi\).
\end{definition}
\begin{example}\label{eg:3.16}
Continuing from Example \ref{eg:3.12}, we consider the interpretation of the mixed-dimensional linearized volumetric strain on domains of various dimensionality. 
\begin{enumerate} \small
    \item For \( i\in I^{n}\), then we obtain as in the fixed-dimensional case
    
    \begin{equation}
     \frac{\iota_{i}\mathfrak{j}(\mathfrak{u})}{\det(\underline{\mathbf{F}}_{i})} = \text{ Trace }(\underline{\mathbf{C}}_{i}^{-1}\iota_{i}\mathfrak{e}(\mathfrak{u})).
    \end{equation}
    
    \item For \( i\in I^{n-1}\), then we first calculate the Fréchet derivative  based on its action on \( \Psi\) as (using that for a conforming forest, \((\iota_{i}\mathbbm{d}_{\underline{\Phi }}\underline{\Phi })_{\perp } = 0\) by definition, and introducing the shorthand notation \(\mathbbm{d}_{\Delta }(\underline{\Phi }+\epsilon \Psi) =\mathbbm{d}_{\underline{\Phi }+\epsilon \Psi }(\underline{\Phi }+\epsilon \Psi)-\mathbbm{d}_{\underline{\Phi }}(\underline{\Phi }+\epsilon \Psi)\)):
    
\begin{equation}
\begin{split}
\iota_{i} &D\mathfrak{J}\left(\underline{\Phi }\right)\left(\Psi\right)
= \lim_{\epsilon \rightarrow 0}\epsilon^{-1}\left(\iota_{i}\mathfrak{J}\left(\underline{\Phi }+\epsilon \Psi\right)-\iota_{i}\mathfrak{J}\left(\underline{\Phi }\right)\right)\\  
& = \lim_{\epsilon \rightarrow 0}\epsilon^{-1}\Bigg(\left(\mathcal{l}_{i}+\omega_{i}^{i}\left(\iota_{i}\mathbbm{d}_{\underline{\Phi }+\epsilon \Psi }\left(\underline{\Phi }+\epsilon \Psi\right)\right)_{\perp }\right)\mathrm{vol}\left(D\left(\underline{\phi }_{i}+\epsilon \psi_{i}\right)\right)
-\mathcal{l}_{i}\mathrm{vol}\left(\underline{\mathbf{F}}_{i}\right)\Bigg)\\
&= \lim_{\epsilon \rightarrow 0}\epsilon^{-1}\Bigg(\left(\mathcal{l}_{i}+\omega_{i}^{i}\left(\iota_{i}\mathbbm{d}_{\underline{\Phi }+\epsilon \Psi }\left(\underline{\Phi }+\epsilon \Psi\right)\right)_{\perp }\right)\\
&\left.\sqrt{\det\left(\left(D\left(\underline{\phi }_{i}+\epsilon \psi_{i}\right)\right)^{T}D\left(\underline{\phi }_{i}+\epsilon \psi_{i}\right)\right)}-\mathcal{l}_{i}\sqrt{\det\left(\underline{\mathbf{F}}_{i}^{T}\underline{\mathbf{F}}_{i}\right)}\right)\\ 
&= \lim_{\epsilon \rightarrow 0}\epsilon^{-1}\Bigg(\left(\mathcal{l}_{i}+\omega_{i}^{i}\left(\left(\iota_{i}\mathbbm{d}_{\underline{\Phi }}\left(\underline{\Phi }+\epsilon \Psi\right)\right)_{\perp }+\left(\iota_{i}\mathbbm{d}_{\Delta }\left(\underline{\Phi }+\epsilon \Psi\right)\right)_{\perp }\right)\right)\\
&\sqrt{\det\left(\underline{\mathbf{F}}_{i}^{T}\underline{\mathbf{F}}_{i}+\epsilon\left(\underline{\mathbf{F}}_{i}^{T}D\psi_{i}+\left(D\psi_{i}\right)^{T}\underline{\mathbf{F}}_{i}\right)+\epsilon^{2}\left(D\psi_{i}\right)^{T}D\psi_{i}\right)}\\
&-\mathcal{l}_{i}\sqrt{\det\left(\underline{\mathbf{F}}_{i}^{T}\underline{\mathbf{F}}_{i}\right)}\Bigg)\\
&= \lim_{\epsilon \rightarrow 0}\epsilon^{-1}\mathrm{vol}\left(\underline{\mathbf{F}}_{i}\right)\Bigg(\left(\mathcal{l}_{i}+\omega_{i}^{i}\left(\iota_{i}\left(\epsilon\mathbbm{d}_{\underline{\Phi }}\Psi +\mathbbm{d}_{\Delta }\left(\underline{\Phi }+\epsilon \Psi\right)\right)\right)_{\perp }\right)\\
&\left(1+\frac{\epsilon }{2}\text{ Trace }\left(\underline{\mathbf{C}}_{i}^{-1}\left(\underline{\mathbf{F}}_{i}^{T}D\psi_{i}+\left(D\psi_{i}\right)^{T}\underline{\mathbf{F}}_{i}\right)\right)+\mathcal{O}\left(\epsilon^{2}\right)\right)-\mathcal{l}_{i}\Bigg)\\
&= \mathrm{vol}\left(\underline{\mathbf{F}}_{i}\right)\left(\omega_{i}^{i}\left(\iota_{i}\mathbbm{d}_{\underline{\Phi }}\Psi\right)_{\perp }+\frac{\mathcal{l}_{i}}{2}\text{ Trace }\left(\underline{\mathbf{C}}_{i}^{-1}\left(\underline{\mathbf{F}}_{i}^{T}D\psi_{i}+\left(D\psi_{i}\right)^{T}\underline{\mathbf{F}}_{i}\right)\right)\right)\\
&\lim_{\epsilon \rightarrow 0} 1+\frac{\omega_{i}^{i}\mathrm{vol}\left(\underline{\mathbf{F}}_{i}\right)\left(\iota_{i}\mathbbm{d}_{\underline{\Phi }}\Psi\right)_{\perp }}{2}\text{ Trace }\left(\underline{\mathbf{C}}_{i}^{-1}\left(\underline{\mathbf{F}}_{i}^{T}D\psi_{i}+\left(D\psi_{i}\right)^{T}\underline{\mathbf{F}}_{i}\right)\right)\\
&\lim_{\epsilon \rightarrow 0} \epsilon +\omega_{i}^{i}\mathrm{vol}\left(\underline{\mathbf{F}}_{i}\right)\lim_{\epsilon \rightarrow 0}\left(\iota_{i}\left(\mathbbm{d}_{\Delta }\underline{\Phi}\right)\right)_{\perp }\left(1+\epsilon \text{ Trace }\left(\underline{\hat{\mathbf{F}}}_{i}^{-1}D\hat{\psi }_{i}\right)\right)+\mathcal{O}\left(\epsilon\right)\\
& = \mathrm{vol}\left(\underline{\mathbf{F}}_{i}\right)\left(\omega_{i}^{i}\left(\iota_{i}\mathbbm{d}_{\underline{\Phi }}\Psi\right)_{\perp }+\frac{\mathcal{l}_{i}}{2}\text{ Trace }\left(\underline{\mathbf{C}}_{i}^{-1}\left(\underline{\mathbf{F}}_{i}^{T}D\psi_{i}+\left(D\psi_{i}\right)^{T}\underline{\mathbf{F}}_{i}\right)\right)\right).
\end{split}
\end{equation}
Here we have used the continuity of the closest point projection in the definition of the jump operator to conclude that \( \lim_{\epsilon \rightarrow 0}(\iota_{i}(\mathbbm{d}_{\Delta }\Phi))_{\perp } = 0\). 
Now it follows that: 
\begin{equation}
\begin{split}
\frac{\iota_{i}\mathfrak{j}(\mathfrak{u})}{\mathcal{l}_{i}\mathrm{vol}(\underline{\mathbf{F}}_{i})} &=\frac{\iota_{i}D\mathfrak{J}(\Phi)(\Xi\mathfrak{u})}{\mathcal{l}_{i}\mathrm{vol}(\underline{\mathbf{F}}_{i})}\\
&=\frac{\omega_{i}^{i}}{\mathcal{l}_{i}}(\iota_{i}\mathbbm{d}\mathfrak{u})_{\perp }+\text{ Trace }(\underline{\mathbf{C}}_{i}^{-1}\frac{1}{2}(\underline{\mathbf{F}}_{i}^{T}D(\iota_{i}\Xi\mathfrak{u})+(D(\iota_{i}\Xi\mathfrak{u}))^{T}\underline{\mathbf{F}}_{i})).
\end{split}
\end{equation}
Here we have used that for \( i\in I^{n-1}\), it holds that \( \iota_{i}\mathbbm{d}\Xi\mathfrak{u} = \iota_{i}\mathbbm{d}\mathfrak{u}\). 
    \item For \( i\in I^{d<n-1}\), a similar calculation as above gives
    \begin{equation}
    \begin{split}
    \frac{\iota_{i}\mathfrak{j}(\mathfrak{u})}{\mathcal{l}_{i}^{n-d_{i}}\mathrm{vol}(\underline{\mathbf{F}}_{i})} &= \sum_{j\in\hat{I}_{i}^{n-1}}^{}\frac{\omega_{i}^{i}}{\mathcal{l}_{i}}(\iota_{j}\mathbbm{d}\mathfrak{u})_{\perp }\\
    &+\text{ Trace }\left(\underline{\mathbf{C}}_{i}^{-1}\frac{1}{2}(\underline{\mathbf{F}}_{i}^{T}D(\iota_{i}\Xi\mathfrak{u})+(D(\iota_{i}\Xi\mathfrak{u}))^{T}\underline{\mathbf{F}}_{i})\right).
    \end{split}
    \end{equation}
    Here we have used that the initial state is a conforming forest, thus since \( s_{i} = s_{j}\) it holds that \(\underline{\phi }_{j,i}^{\ast }\) is the identity.
    \item  For \( i\in I\), and \( j\in I_{i}\), then \( \iota_{j}\mathfrak{j}(\mathfrak{u}) = 0\).
\end{enumerate}
\end{example}
The above calculation shows that for fractures and intersections, \( i\in I^{d<n}\), the linearized volume change has two components, which have the natural interpretations of transverse opening and longitudinal stretching. 

Motivated by the fixed-dimensional case, we wish to express the linearized volume change in terms of the linearized strain. We begin by observing that with the representation of \( \Xi\) as defined above, and \( i\in I^{d<n}\), then 
\begin{equation}
\begin{split}
&\text{ Trace }(\underline{\mathbf{C}}_{i}^{-1}\frac{1}{2}(\underline{\mathbf{F}}_{i}^{T}D(\iota_{i}\Xi\mathfrak{u})+(D(\iota_{i}\Xi\mathfrak{u}))^{T}\underline{\mathbf{F}}_{i}))\\ &=\frac{1}{\left\vert\hat{I}_{i}^{n}\right\vert }\sum_{j\in\hat{I}_{i}^{n}}^{}\text{ Trace }(\underline{\mathbf{C}}_{i}^{-1}\frac{1}{2}(\underline{\mathbf{F}}_{i}^{T}D(\iota_{j}\hat{\mathfrak{u}})+(D(\iota_{j}\hat{\mathfrak{u}}))^{T}\underline{\mathbf{F}}_{i}))\\ &=\frac{1}{\left\vert\hat{I}_{i}^{n}\right\vert }\sum_{j\in\hat{I}_{i}^{n}}^{}\text{ Trace }(\underline{\mathbf{C}}_{i}^{-1}(\iota_{j}\mathfrak{e})_{\Vert  })
\end{split}
\end{equation}
This calculation again exploits that the initial state is a conforming forest, since for \( s_{i} = s_{j}\) it holds that \(\underline{\mathbf{F}}_{i} =\underline{\mathbf{F}}_{j}\). Moreover, we recall that for \( i\in I^{n-1}\), it holds that \(\frac{1}{\mathcal{l}_{i}}(\iota_{i}\mathbbm{d}\mathfrak{u})_{\perp } =(\iota_{i}\mathfrak{e})_{\perp }\). This allows us to introduce the following: 
\begin{definition}
The \textit{mixed-dimensional matrix trace} operator \(\mathfrak{T'}:L^{2}(\mathfrak{X}^{1},\mathbb{R}^{n})\rightarrow L^{2}(\mathfrak{X}^{n}, \mathbb{R})\) is defined as follows. For \(\mathfrak{u}\in L^{2}(\mathfrak{X}^{1},\mathbb{R}^{n})\), and \(\mathfrak{e} = \mathfrak{D}_{s}\mathfrak{u}\in L^{2}(\mathfrak{X}^{1},\mathbb{R}^{n})\) the trace operator satisfies 
\begin{equation}
\mathfrak{j}(\mathfrak{u})\coloneqq\mathfrak{J}(\underline{\Phi })\mathfrak{T'e}=\mathfrak{J}(\underline{\Phi })\mathfrak{T}'(\mathfrak{D}_{s}\mathfrak{u})
\end{equation}
\end{definition}
\begin{example}Continuing from Example \ref{eg:3.16}, we consider the interpretation of the mixed-dimensional matrix trace on domains of various dimensionality.
\begin{enumerate} \small
    \item For \( i\in I^{n}\), then as in the fixed-dimensional case
\begin{equation*}
\iota_{i}(\mathfrak{T}'\mathfrak{e}) = \text{ Trace}(\underline{\mathbf{C}}_{i}^{-1}\iota_{i}\mathfrak{e})
\end{equation*}

    \item For \( i\in I^{d<n}\), then
\begin{equation*}
\iota_{i}(\mathfrak{T}'\mathfrak{e}) = \sum_{j\in\hat{I}_{i}^{n-1}}^{}\omega_{i}^{j}(\iota_{j}\mathfrak{e})_{\perp }+\frac{1}{\left\vert\hat{I}_{i}^{n}\right\vert }\sum_{j\in\hat{I}_{i}^{n}}^{}\text{ Trace }(\underline{\mathbf{C}}_{i}^{-1}(\iota_{j}\mathfrak{e})_{\Vert  })
\end{equation*}

    \item For \( i\in I\), and \( j\in I_{i}\), then \( \iota_{j}(\mathfrak{T}'\mathfrak{e}) = 0\).

\end{enumerate}
\end{example}

\begin{figure}
    \centering
    \includegraphics[width=0.8\textwidth]{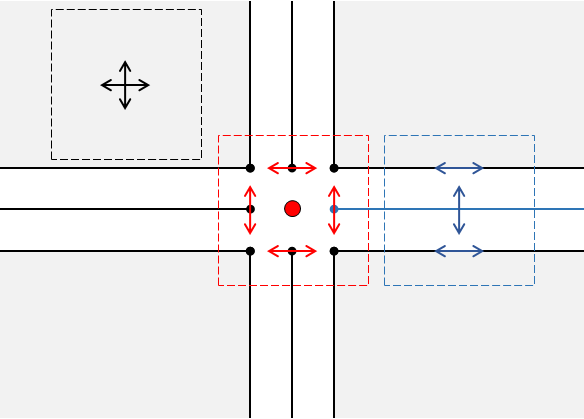}
    \caption{The mixed-dimensional trace operator \(\mathfrak{T'}\) maps strains defined in \( L^{2}(\mathfrak{X}^{1},\mathbb{R}^{n})\) to volumetric changes defined in \( L^{2}(\mathfrak{X}^{n}, \mathbb{R})\) for \( i\in I^{2}\) (black), \( i\in I^{1}\) (blue), \( i\in I^{0}\) (red).}
    \label{fig:mixed-dimentional-trace-operator}
\end{figure}

\begin{example}
Let \( n = 2\) and consider the geometry from Figure \ref{fig:mixed-dimentional-trace-operator}.
\begin{itemize} \small
    \item Let \( i\in I^{2}\). The linearized volumetric strain is given, as in the fixed-dimensional case by the trace of the linear strain, depicted by black arrows in Figure \ref{fig:mixed-dimentional-trace-operator}.

    \item Let \( i\in I^{1}\). Volumetric changes of \( \Omega_{i}\) are measured using two metrics: changes with respect to tangential dilation are captured using the strain on the adjacent skins whereas changes in the aperture are measured using the locally defined strain vector. This is illustrated by the blue arrows in Figure \ref{fig:mixed-dimentional-trace-operator}.

    \item Let \( i\in I^{0}\). In this case, volumetric changes of \( \Omega_{i}\) are measured using the perpendicular components of the strain vectors in the adjacent fractures. This is illustrated by the red arrows in Figure \ref{fig:mixed-dimentional-trace-operator}. 

    \item Let \( n = 3\) and consider an extrusion of the geometry from Figure \ref{fig:mixed-dimentional-trace-operator} in the third dimension. Then, for \( i\in I^{1}\), \( \Omega_{i}\) represents an intersection line between fractures. Here, the tangential stretching is captured using the trace of the three-dimensional strain in the four corner lines whereas the opening is measured using the change in aperture of the adjacent fractures.
\end{itemize}
\end{example}

We also define a mapping of pressures to stresses, generalizing the identity tensor. The following definition ensures that the duality from the fixed-dimensional case is preserved in reference space.
\begin{definition}
\label{def:3.20}
Let the \textit{mixed-dimensional identity} operator \(\mathfrak{T} : L^{2}(\mathfrak{X}^{n}, \mathbb{R})\rightarrow L^{2}(\mathfrak{X}^{1},\mathbb{R}^{n})\) be such that 
\begin{align}
\langle\mathfrak{Ta, b}\rangle_{\mathfrak{X}^{1}} &\coloneqq \langle\mathfrak{a, T'b}\rangle_{\mathfrak{X}^{n}},  
& \forall\mathfrak{a} &\in L^{2}(\mathfrak{X}^{n}, \mathbb{R}) \text{ and } \mathfrak{b}\in L^{2}(\mathfrak{X}^{1},\mathbb{R}^{n}).
\end{align}
\end{definition}
We summarize the linearized operators constructed in Section \ref{sec:mixed_dimensional_linearized} and \ref{sec:linearized_vol} with the diagram shown in Figure \ref{fig:figure7}.  This diagram also includes the co-symmetric-gradient, \( (\mathbb{D}_{s} \cdot) \), which will be defined in Section \ref{sec:governing_eq_linearized}. Its  duality with \(\mathfrak{D}_{s}\) is discussed in Section \ref{sec:the_space}.
\begin{figure}[h]
    \centering
    \includegraphics[width=0.6\textwidth]{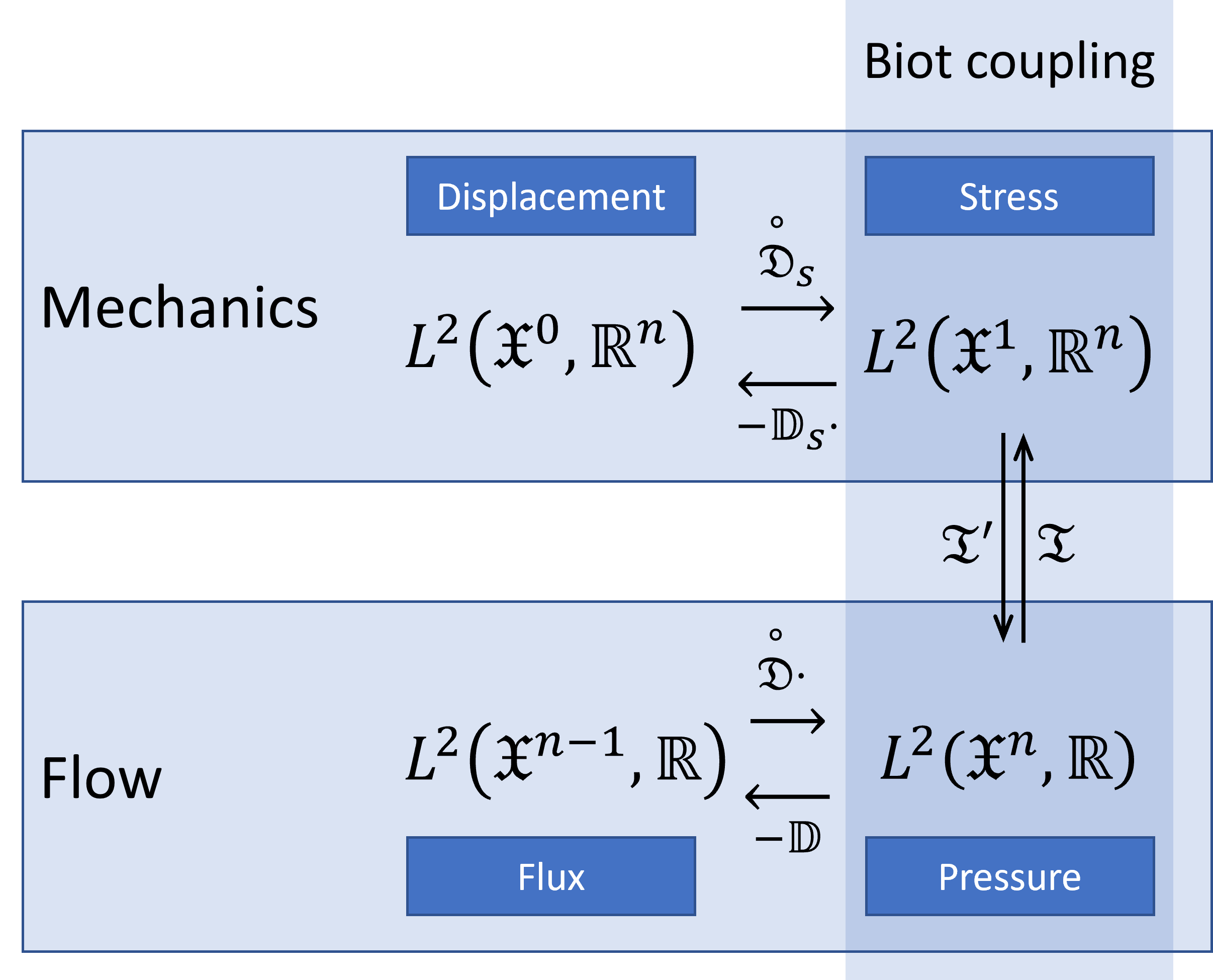}
    \caption{The canonical structure of coupled poromechanics.}
    \label{fig:figure7}
\end{figure}

\section{Mixed-dimensional poromechanics}
\label{sec:mixed-dimentional-poromechanics}

The preceding sections have laid the geometric foundations for considering deformation in mixed-dimensional geometries. In this section, we use these foundations to develop a theory of poromechanics. Our presentation will first establish the spatial structure of the system, as inferred from the geometry and differential operators defined above. We then consider the physical modeling of fluid flow, thence mechanics, and finally summarize the complete model. We close this section by discussing the relationship between mixed-dimensional poromechanics and classical equidimensional models. The exposition follows the standard modeling approach for poromechanics, which in the fixed-dimensional case is carefully reviewed by e.g. Coussy  \cite{coussy2005poromechanics}, and summarized recently in  \cite{reveron2021iterative}.

 We emphasize that all modeling in this section is considered on the reference domains \(\mathfrak{X}\). This implies that all variables and derivatives are with respect to these domains, such that e.g. when considering a mass density \( \rho_{i}\), it is understood that this is mass density per unit volume form on \( X_{i}\). With density as a concrete example, and domains \( i\) with \( d_{i} = 3\), this implies that density is a function of thermodynamics state (pressure and temperature), as well as the volumetric deformation (pullback of \( n\)-forms).  On domains \( i\) with \( d_{i}<3\), we immediately get the interpretation of mass per measure, i.e. that the physical units of mass density are not simply mass per volume, but rather mass per area, length, or point. 

\subsection{Primary variables and structure}
\label{sec:primary_variables}
We consider the quasi-static poroelastic system as the composition of a mass conservative description of fluid flow, together with a mechanical system at equilibrium, as formulated on the reference domains \(\mathfrak{X}^{k}\). As such, we consider displacement \(\mathfrak{u}\in C(\mathfrak{X}^{0},\mathbb{R}^{n})\), total stress \(\mathfrak{s}\in C(\mathfrak{X}^{1},\mathbb{R}^{n})\), fluid mass density \( \rho(\mathfrak{p})\) (represented for simplicity by pressure \(\mathfrak{p}\in C(\mathfrak{X}^{n}, \mathbb{R})\) in absence of thermal effects) and fluid mass flux \(\mathfrak{q}\in C(\mathfrak{X}^{n-1}, \mathbb{R})\) as primary variables. For the sake of exposition, we will frequently also refer to the strain \(\mathfrak{E}\in C(\mathfrak{X}^{1},\mathbb{R}^{n})\), which we consider as a secondary variable derived from the displacement. In this section we will assume enough regularity for all the definitions below to be valid: a weak formulation where regularity is considered in more detail is presented in Section \ref{sec:well-posedness}.

Before we detail the physical constitutive laws, we note that the structure of the function spaces and constitutive laws essentially dictate the underlying canonical structure of the spatial operators. This is already summarized in Figure \ref{fig:figure7} above, which, with reference to the primary variables defined above, can be summarized in terms of the following structures. 

Conservation structure for fluid mass is stated on the space \( C(\mathfrak{X}^{n}, \mathbb{R})\), and can thus (apart from a source term \(\mathfrak{r}_{\mathfrak{m}}\)) only involve the quantities pressure \(\mathfrak{p}\), volumetric strain \(\mathfrak{J}(\Xi\mathfrak{u}+\underline{\Phi })\) and divergence of flux \(\mathfrak{D} \cdot \mathfrak{q}\). 

Similarly, the balance of mechanical forces is stated on the space \( C(\mathfrak{X}^{0},\mathbb{R}^{n})\), and can thus (apart from external forces \(\mathfrak{r}_{\mathfrak{s}}\)) only involve the quantities displacement \(\mathfrak{u}\) and co-gradient of stress \(\mathbb{D} \cdot \mathfrak{s}\). There are no other operators that map to this space. 

We will see that the system has two constitutive laws. For the fluid, this is a binary relationship on \( C(\mathfrak{X}^{n-1}, \mathbb{R})\), involving the flux \(\mathfrak{q}\) and the gradient \(\mathbb{D} \mathfrak{p}\). 

For the mechanical forces, this is in principle a binary relationship on \( C(\mathfrak{X}^{1},\mathbb{R}^{n})\), involving the total stress \(\mathfrak{s}\) and the strain \(\mathfrak{E}\). However, as noted in Section \ref{sec:linearized_vol}, the linearized volumetric strain is actually a trace map from the strain space to the mass density space, and thus its dual, the identity map, enters the force balance. 

In the following, we will present the full temporal modeling for finite strain and non-linear constitutive laws. As a guide to this development, we provide the following structure, which we consider the canonical ``Laplacian" of mixed-dimensional poromechanics, obtained from linearly combining the admissible relationships (with unit weights): 
\begin{subequations}
\begin{align}
    &\text{Force balance:}
    & \partial_t^2\mathfrak{u} - \mathbb{D}_s \cdot \overbrace{\left( \mathfrak{D}_s \mathfrak{u} + \mathfrak{Tp} \right)}^{\mathfrak{s}} &= \mathfrak{r}_{\mathfrak{m}}, \\
    &\text{Mass balance:}
    & \partial_t\mathfrak{p} + \mathfrak{T}^* \mathfrak{D}_s \mathfrak{u} 
    + \mathfrak{D} \cdot \underbrace{\left( - \mathbb{D} \mathfrak{p} \right)}_{\mathfrak{q}} &= \mathfrak{r}_{\mathfrak{s}}.
\end{align}
\end{subequations}

\subsection{Fluid flow}
\label{sec:fluid_flow}
As surveyed in the previous section, the conservation law for fluid mass can be stated as:
\begin{equation}
\partial_{t}\mathfrak{m} + \mathfrak{D} \cdot \mathfrak{q} = \mathfrak{r}_{\mathfrak{m}},
\end{equation}
where the fluid mass content satisfies a constitutive relation dependent the mechanical strain and fluid pressure, 
\begin{equation}
\mathfrak{m = m}(\mathfrak{E, p})\in C(\mathfrak{X}^{n}, \mathbb{R}).
\end{equation}
The fluid mass flux is assumed to satisfy a linear proportionality with a gradient, specifically the co-divergence \(\mathbb{D}\), of fluid pressure, which corresponds to Darcy’s law for intact porous rock (\( i\in I^{n}\)) and laminar flow fractures (\( i\in I^{n-1}\))  \cite{martin2005modeling, boon2018robust}: 
\begin{equation}
\label{eq:Darcy}
\mathfrak{  q} = \kappa(\mathfrak{E,q})(\mathbb{-D}\mathfrak{p}+\mathfrak{r}_{\mathfrak{q}}).
\end{equation}
With \(\mathfrak{r}_{\mathfrak{q}}\) the contribution due to gravity. Note that we have allowed the material coefficient to depend on the strain, in order to accommodate that the fracture conductivity may be altered when there is slip along the fracture. Moreover, the dependency on the fluid flux accommodates non-linear relationships such as the Darcy-Forchheimer law.
\begin{remark}
The permeability of a rock domain \( \iota_{i}\kappa(\mathfrak{E,q})\) with \( d_{i} = 3\), is the permeability on the reference space \( X_{i}\). Given a permeability \( K_{i}\) on the physical domain, these are related by the usual transformation rules  \cite{hughes1983mathematical}, e.g. 
\begin{equation}\label{eq:4.4}
K_{i} = \det(\mathbf{F}_{i})^{-1}\mathbf{F}_{i}(\iota_{i}\kappa(\mathfrak{E,q}))\mathbf{F}_{i}^{T}.
\end{equation}
In particular, this implies that a scalar permeability \( K_{i}\) in physical space may nevertheless be represented by an anisotropic tensor \( \iota_{i}\kappa\) on \( X_{i}\) (and opposite). On the other hand, it is important to note that \eqref{eq:4.4} preserves symmetry properties, thus symmetry of \( \iota_{i}\kappa\) corresponds to symmetry of \( K_{i}\).  Similar comments apply to all material properties introduced in this and the following sections.
\end{remark}

\subsection{Mechanical response}
\label{sec:mechanical_response}

We consider balance of momentum as the basis for modeling the mechanical response. We state the equilibrium assumption recognizing that \(\mathbb{D} \cdot \) (the co-gradient) is a divergence operator on the top-dimensional domains and their boundaries. Thus for a stress variable \(\mathfrak{s}\) (with the interpretation of a Piola-Kirchhoff stress of the second kind), a change in momentum \( \rho_{r}\partial_{t}^{2}\mathfrak{u}\) in view of a mixture mass density for fluid and rock \( \rho_{r}\) leads us to the following balance of momentum: 
\begin{equation}
\rho_{r}\partial_{t}^{2}\mathfrak{u}-\mathbb{D} \cdot (\boldsymbol{\mathfrak{F}}\mathfrak{s}) =\mathfrak{r}_{\mathfrak{s}},
\end{equation}
with \(\mathfrak{r}_{\mathfrak{s}}\) describing body forces. This vector equation has components associated with the basis vectors of \(\mathbb{R}^{n}\). The deformation \(\boldsymbol{\mathfrak{F}}\) enters the momentum balance to transform forces on \( T\hat{X}_{i}\rightarrow\mathbb{R}^{n}\)  \cite{hughes1983mathematical} with \(\hat{X}_{i}\) the extended domain defined in Definitions \ref{def:2.6} and \ref{def:2.7}.

 The mechanical stress state depends on both the strain and the fluid pressure. As we are primarily interested in moderate deformations, we restrict our attention to linear (Saint Venant-Kirchhoff) stress-strain response in the porous rock and its boundaries, \( j\in\mathfrak{S}_{i}\), with \( i\in I^{n}\)  \cite{temam2005mathematical, gurtin1975continuum}:
\begin{equation}\label{eq:4.6}
\iota_{j}\mathfrak{s}(\mathfrak{E,p}) = C_{j} :\iota_{j}\mathfrak{E}-\iota_{j}(\alpha\mathfrak{Tp}).
\end{equation}
With \( \alpha\) a positive, symmetric linear operator generalizing the Biot-Willis constant  \cite{coussy2005poromechanics} and \( C_{j}\) the fourth-order stiffness tensor.

 Physical reality demands a greater generality than linear mechanical response of the fracture, since both friction and contact mechanics will in general be given by non-linear relationships. Moreover, for finite deformation, determination of contact itself may be a non-linear problem. Therefore we allow for a (possibly non-linear and non-local in space) constitutive law \(\mathfrak{A}_{i}\) between stress and strain to be defined on fractures \( i\in I^{n-1}\) and their boundaries  \cite{kikuchi1988contact}. We impose this by introducing \(\mathfrak{A}_{i}\) as a binary relation (see Appendix \ref{sec:appendix}) such that
\begin{equation} \label{eq: binary relation}
( \iota_{i}(\mathfrak{s}+\alpha\mathfrak{Tp}), \iota_{i}(\hat{\gamma }+\check{\gamma }\partial_{t})\mathfrak{E})\in\mathfrak{A}_{i}.
\end{equation}
Here \(\hat{\gamma }\) and \(\check{\gamma }\) are given parameters with \(\hat{\gamma }+\check{\gamma } = 1\) and \(\hat{\gamma }\check{\gamma } = 0\) on each \( X_{i}\) that allow us to model relationships between \textit{either} stress and strain \textit{or} stress and strain rate, respectively. In Section \ref{sec:connection_to_classical}, we will give examples of particular choices of friction and contact laws.

We combine the constitutive laws for the bulk and fracture subdomains in the notation of a mixed-dimensional binary relation: 
\begin{equation}
(\mathfrak{s}+\alpha\mathfrak{Tp, }(\hat{\gamma }+\check{\gamma }\partial_{t})\mathfrak{E}) \in \mathfrak{A}.
\end{equation}
\begin{remark}
The presence of the material law \eqref{eq:4.6} on boundaries, implies that the constitutive modeling is general enough to allow for materials with coated boundaries (say, covered by a thin membrane), or other disturbances of the material parameters associated with the boundaries of the domain. A perfectly homogeneous material with no disturbance in material parameter at its boundaries will then be a degenerate case of the model with \(\mathfrak{A}_{j} = 0\) (for \( j\in I_{i}\) with \( i\in I^{n}\)).
\end{remark}
\subsection{Governing equations for mixed-dimensional finite strain poromechanics}
\label{sec:governing_eq_mixed_dimensional_finite}
We summarize the above developments in the system of equations for mixed-dimensional poroelastic fractured media. 
\begin{table}[!htbp]
    \centering
    \begin{tabular}{|p{11.5cm}|}
    \hline
{
\textbf{Governing equations for mixed-dimensional finite strain poromechanics}
\begin{subequations} \label{Box 4.1}
\begin{align}
&\text{Balance of forces: }&
\rho_{r}\partial_{t}^{2}\mathfrak{u}-\mathbb{D}\cdot\left(\boldsymbol{\mathfrak{F}}\mathfrak{s}\right) =\mathfrak{r}_{\mathfrak{s}}
&\\ 
&\text{Balance of mass: }&  
\partial_{t}\mathfrak{m}+\mathfrak{D}\cdot \mathfrak{q} = \mathfrak{r}_{\mathfrak{m}}  &\\
&\text{Finite strain: }&  
\mathfrak{E} = \mathfrak{E}\left(\mathfrak{u}\right) &\\
&\text{Stress-strain binary relations: }& 
\left(\mathfrak{s}+\alpha\mathfrak{T}\mathfrak{p},\left(\hat{\gamma }+\check{\gamma }\partial_{t}\right)\mathfrak{E}\right)\in \mathfrak{A} &\\
&\text{Darcy’s law: }& 
\mathfrak{q}=\kappa\left(\mathfrak{E,q}\right)\left(\mathbb{-D}\mathfrak{p}+\mathfrak{r}_{\mathfrak{g}}\right) &\\
&\text{Fluid mass content: } & 
\mathfrak{m} = \mathfrak{m}\left(\mathfrak{E, p}\right) &\\
&\text{Finite deformation gradient:} &  
\boldsymbol{\mathfrak{F}} = D\left(\Xi\mathfrak{u}+\underline{\hat{\Phi }}\right) &
\end{align}
\vspace{-0.5cm}
\end{subequations}} 
\\
    \hline
    \end{tabular}
\end{table}

By taking the right-hand sides \(\mathfrak{r}_{\mathfrak{s}}\) and \(\mathfrak{r}_{\mathfrak{m}}\) together with the mixture density \( \rho_{r}\) as given, as well as appropriate initial and boundary conditions, the above set of equations are formally closed. Superficially, we identify that we have 7 equations for the 7 (mixed-dimensional) unknowns \(\left[\mathfrak{u}, \mathfrak{s}, \boldsymbol{\mathfrak{F}}, \mathfrak{m}, \mathfrak{q}, \mathfrak{E}, \mathfrak{p}\right]\). A careful counting of scalar and vector equations on domains of various dimensionality supports this claim. The field equations should be supplemented by appropriate boundary conditions, we will discuss one such choice in Section \ref{sec:well-posedness}.

A well-posedness theory for the general finite deformation mixed-dimensional model is not within reach, since it would require us to simultaneously address open questions in both contact mechanics and poromechanics. On the other hand, in the context of infinitesimal deformation, fairly general results are nevertheless possible to establish. In the next sections, we will therefore restrict our attention to linearized strain and make specific assumptions on the constitutive laws that allow us to rigorously establish an example of a well-posed model for mixed-dimensional poromechanics. 

\subsection{Governing equations for mixed-dimensional linearized strain poromechanics}
\label{sec:governing_eq_linearized}

We continue by considering the model \eqref{Box 4.1} in the context of the mixed-dimensional linearized strain from Def. \ref{def:3.11}. Thus, we assume that the strain is given by
\begin{equation*}
\mathfrak{E}(\mathfrak{u})\approx\mathfrak{e}(\mathfrak{u}) =\mathfrak{D}_{s}\mathfrak{u}
\end{equation*}

Our exposition simplifies (and as we will see, symmetrizes) by including the transformation of the reference configuration in the definition of the divergence operators (in analogy to the fixed-dimensional case, see e.g.  \cite{hughes1983mathematical}).
\begin{definition}\label{def:4.1}
For an initial configuration \(\underline{\hat{\Phi}}\) with derivative \(\underline{\boldsymbol{\mathfrak{F}}}\), let the \textit{mixed-dimensional co-symmetric-gradient }\((\mathbb{D}_{s} \cdot ):C^{1}(\mathfrak{X}^{1},\mathbb{R}^{n})\rightarrow C(\mathfrak{X}^{0},\mathbb{R}^{n})\) be defined such that
\begin{equation*}
\mathbb{D}_{s}\cdot \mathfrak{s} \coloneqq \mathbb{D} \cdot (\underline{\boldsymbol{\mathfrak{F}}}\mathfrak{s})
\end{equation*}
\end{definition}
\begin{remark}
Recall that in the fixed-dimensional case, the symmetric gradient is adjoint to the divergence applied to symmetric tensor fields. This forms a key tool in the well-posedness for fixed-dimensional elasticity models and we note that this property remains valid in the mixed-dimensional setting, as we show in Section \ref{sec:model_equations}.

Secondly, we linearize the fluid mass content relationship \(\mathfrak{m}(\mathfrak{e,p})\). Consistent with the relationships discussed in \cite{coussy2005poromechanics}, we obtain
\begin{equation*}
\mathfrak{ m}(\mathfrak{e,p})\approx\mathfrak{m}_{0}+\mathfrak{T'}(\alpha\mathfrak{e})+\beta\mathfrak{p}
\end{equation*}

Here, \(\mathfrak{m}_{0}\) is the initial mass content, \(\mathfrak{T'}\) is the mixed-dimensional identity operator from Def. \ref{def:3.20}, \( \alpha\) is the generalized Biot-Willis constant from \eqref{eq:4.6}, and \( \beta\) is the specific storativity.

We summarize these simplifications in the following system of equations
\begin{table}[!htbp]
    \centering
    \begin{tabular}{|p{11.5cm}|}
        \hline
{
\textbf{Governing equations for mixed-dimensional linearized strain poromechanics}
\begin{subequations} \label{Box 4.2}
\begin{align}
&\text{Balance of forces: }& 
 \rho_{r}\partial_{t}^{2}\mathfrak{u}-\mathbb{D}_{s}\cdot\mathfrak{s} = \mathfrak{r}_{\mathfrak{s}} 
& \label{eq: 4.11a}\\
&\text{Balance of mass: }& 
 \partial_{t}\mathfrak{m}+\mathfrak{D}\cdot \mathfrak{q} =\mathfrak{r}_{\mathfrak{m}} 
& \label{eq: 4.11b}\\
&\text{Linearized strain: }& 
\mathfrak{e} = \mathfrak{D}_{s}\mathfrak{u} 
& \label{eq: 4.11c}\\
&\text{Stress-strain binary relations: }& 
\left(\mathfrak{s}+\alpha\mathfrak{Tp, }\left(\hat{\gamma }+\check{\gamma }\partial_{t}\right)\mathfrak{e}\right)\in \mathfrak{A} 
& \label{eq: 4.11d}\\
&\text{Darcy’s law: }& 
\mathfrak{q} = \kappa\left(\mathfrak{q,e}\right)\left(\mathbb{-D}\mathfrak{p}+\mathfrak{r}_{\mathfrak{q}}\right) 
& \label{eq: 4.11e}\\
&\text{Fluid mass content: }& 
\mathfrak{m}\left(\mathfrak{e,p}\right) =\mathfrak{m}_{0}+\mathfrak{T'}\left(\alpha\mathfrak{e}\right)+\beta\mathfrak{p} 
& \label{eq: 4.11f}
\end{align}
\vspace{-0.5cm}
\end{subequations}}
\\
    \hline
    \end{tabular}       
\end{table}
\end{remark}

\subsection{Connection to classical continuum mechanical formulations}
\label{sec:connection_to_classical}
In this section, we identify how fixed-dimensional formulations of contact mechanics and fluid flow in fractured porous rocks appear as special cases of the governing equations derived in the preceding sections and as summarized in Section \ref{sec:governing_eq_mixed_dimensional_finite}.  This identification will serve as support for our claim that our development is a consistent generalization of accepted mathematical descriptions, extended to our current setting of mixed-dimensional geometries. 

\subsubsection{Contact mechanics}
\label{sec:contact_mechanics}

We illustrate the connection to contact mechanics by omitting all fluid considerations, and restricting our attention to a single fracture in the context of linearized strain, such as described in Example \ref{eg:3.14}, and illustrated by Figure \ref{fig:geometry}. The governing equations, assuming linearized strain, then simplify to (indexing of \( i\) corresponds to indexing in Example \ref{eg:3.14}).

Hooke’s law for solids:
\begin{subequations}
\begin{align}\label{eq:4.12a}
 \sigma_{i} &= C_{i}:\varepsilon(u_{i})  &
 \text{for } i &= 2,3.
\end{align}
Conservation of momentum for solids:
\begin{align}\label{eq:4.12b}
\rho_{r,i}\partial_{t}^{2}u_{i}-\nabla  \cdot (\underline{\mathbf{\hat{F}}}_{i}\sigma_{i}) &= r_{s,i} 
& \text{for } i &= 2,3.
\end{align}
Hooke’s law for surfaces:
\begin{align}\label{eq:4.12c}
\sigma_{j} &= C_{j}:\varepsilon(u_{j}) &
\text{for } j &= 4,5.
\end{align}
Conservation of momentum on surfaces:
\begin{align}
\rho_{r,j}\partial_{t}^{2}u_{j}-\nabla_{\Vert  } \cdot (\underline{\hat{\mathbf{F}}}_{j}\sigma_{j})+(\sigma_{j-2} \cdot n_j-\sigma_{1}) &= r_{s,j} & 
\text{for }  j &= 4,5.
\end{align}
Constitutive law for fracture: 
\begin{align}
F(\varepsilon_{i},\dot{\varepsilon }_{i},\sigma) &= 0 &
\text{for } i &= 1.
\end{align}
\end{subequations}

These equations embody several classical formulations. Surface stress of elastic materials is an important phenomena at some scales  \cite{gurtin1975continuum}, as we will discuss in Section \ref{sec:coated_deformable}. In this subsection, we will disregard these terms, and consider a degenerate model with \( C_j = 0\) and hence \( \sigma_j = 0\) on domains \( j = 4,5\).  In the quasi-static case, \( \rho_{r,i}\partial_{t}^{2}u_{i} = 0\), the above equations now exactly correspond to the classical equations of elasticity and material contact, as summarized in the book of Kikuchi and Oden (see e.g. the exposition in chapters 2 and 13 of \cite{kikuchi1988contact}). We highlight the importance of two particular choices for the constitutive law for fracture. 
\begin{example}[Signorini problem]
Decomposing as in Section \ref{sec:mixed-dimentional-poromechanics} the fracture strain into its tangential part (sliding) \( \iota_{1}\mathfrak{e}_{\Vert  } =\frac{1}{2}\underline{\mathbf{F}}_{i}^{T}(\iota_{i}\mathbbm{d}_{\underline{\Phi }}\mathfrak{u})_{\parallel}\) and perpendicular part (opening) \( \iota_{1}\mathfrak{e}_{\perp } =\mathcal{l}_{i}^{-1}(\iota_{i}\mathbbm{d}_{\underline{\Phi }}\mathfrak{u})_{\perp }\), we first state the constitutive model for frictionless contact (see e.g. equations (2.31) of \cite{kikuchi1988contact}), where the absence of friction is given by:
\begin{subequations}\label{eqs: signorini}
\begin{equation} \label{signorini friction}
\iota_{1}\mathfrak{s}_{\Vert  } = 0,
\end{equation}
and contact mechanics is given by the Karush-Kuhn-Tucker (KKT) triplet
\begin{equation} \label{signorini KKT}
\begin{split}
(\iota_{1}\mathfrak{s}_{\perp }) (\iota_{1}\mathfrak{e}_{\perp }) & = 0\\
\iota_{1}\mathfrak{s}_{\perp }&\leq 0\\
\iota_{1}\mathfrak{e}_{\perp }&\geq 0.
\end{split}
\end{equation}
\end{subequations}
The two inequalities state that the materials cannot interpenetrate, and the contact cannot be tensile, while the equality states that one of the two conditions must hold as an equality. We recognize that equations \eqref{eqs: signorini} are of the form given by the binary relation \eqref{eq: binary relation}, with \( \iota_{1}\check{\gamma } = 0\).
\end{example}
\begin{example}[Rough surfaces]\label{eg:4.4}
The presence of \(\mathcal{l}_{i}^{-1}\) in the perpendicular part of the strain measure, together with the macroscopic condition that the deformation preserves the mixed-dimensional nature of the problem, essentially provides a multi-scale representation of strain. Indeed with \(\mathcal{l}_{i}\ll 1\), we realize that fracture opening is measured relative to a finer scale than the rest of the deformation. This makes sense relative to the mixed-dimensional continuum assumption, Definition \ref{def:2.2}, since macroscopically, the fracture always has negligible transversal width, while the strain nevertheless measures perturbations of the fracture at the scale of \(\mathcal{l}_{i}\). 

At the length-scale transversal to a fracture opening, it is well-known that the compression and crushing of micro-roughness can significantly impact the stress-strain response. This is illustrated in Figure \ref{fig:illustration-of-multi-scale-contact-mechanics}, which has been adapted from the classical presentation by Oden and Martins \cite{oden1985models}. With access to such representations, it is suggested that an appropriate stress-strain model for a fracture (see Chapters 11 and 13 of  \cite{kikuchi1988contact}) depends on both rate of compression and effective opening (in these expressions the plus sign indicates that only positive values are considered, e.g. \((a)_{+} = \max(a,0)\), and \( C_{1}^{l}\) correspond to material constants):
\begin{equation}\label{eq:4.19}
\sigma_{1,\perp } = -C_{1}^{1}(-\varepsilon_{1,\perp })_{+}^{C_{1}^{2}}-C_{1}^{3}(-\dot{\varepsilon }_{1,\perp })_{+}^{C_{1}^{4}}.
\end{equation}
Complementing the perpendicular stress-strain law for fracture is the Coulomb law of friction, expressed as the KKT triplet :
\begin{subequations}
\begin{align}
\lambda^{2}(c_{1}^{5}\sigma_{1,\perp }-\left\vert \sigma_{1,\Vert  }\right\vert) &= 0 \\
c_{1}^{5}\sigma_{1,\perp }-\left\vert \sigma_{1,\Vert  }\right\vert &\geq 0 \\
\lambda^{2}&\geq 0,
\end{align}
\end{subequations}
where \( \lambda^{2}\) is a Lagrange multiplier associated with sliding: 
\begin{equation}
\dot{\varepsilon }_{1,\Vert  } = \lambda^{2}\sigma_{1,\Vert  }.
\end{equation}
\end{example}
While our framework as stated does not allow for the full generality of \eqref{eq:4.19}, with both \( C_{1}^{1}\) and \( C_{1}^{3}\) non-zero, these contact laws are nevertheless admissible as a binary inclusion when either \( C_{1}^{1}\) or \( C_{1}^{3}\) are zero. 

\subsubsection{Fluid flow in rigid fractured porous media}
\label{sec:fluid_flow_rigid}
In the absence of mechanical deformation, the mixed-dimensional equations \eqref{Box 4.1} have previously been shown to be algebraically equivalent to reduced-dimensional formulations of flow in porous media  \cite{boon2021functional}. These equations have become quite popular over the last decade, as illustrated by a recent literature review  \cite{berre2019flow}. An important consideration is the validity of reduced-dimensional models for fracture flow. This has been extensively studied, and is by now well established in the single-phase regime considered herein  \cite{martin2005modeling, angot2009asymptotic, flemisch2018benchmarks}. 

\subsubsection{Fluid flow in deformable porous media}
\label{sec:fluid_flow_deformable}
Let us now validate the model equations in the case without fractures, wherein the mixed-dimensional geometry trivially reduces to the normal fixed-dimensional geometry, which is to say that \( I = I^{n} = \{ 1\}\), and \( \Omega_{1} = Y\).

A review of the definitions in Section \ref{sec:mixed_dimensional} and \ref{sec:differential_operators} for this case of a single domain now verify that all the mixed-dimensional functions revert to their standard definitions from calculus. Thus e.g. a deformation \(\mathfrak{u}\in C(\mathfrak{X}^{0},\mathbb{R}^{n})\) is identically equal to \( u\in C(Y,\mathbb{R}^{n})\), similarly, the differential operators also reduce to their fixed-dimensional counterparts \(\mathfrak{D} \sim \mathbb{D} \sim \nabla\) while \(\mathfrak{D} \cdot \sim \mathbb{D} \cdot  \sim \nabla \cdot \). 

The equations \eqref{Box 4.1} are thus equivalent to the same equations written in ``Latin letters", which correspond exactly to the standard model for poromechanics subjected to large deformations, as summarized in e.g. Table 3.4 of  \cite{coussy2005poromechanics}.

\subsubsection{Coated deformable solids}
\label{sec:coated_deformable}

In this final example, we will consider the special case of elastic solids with surface coatings. The continuum theory for elastic material surfaces goes back to Gurtin and Murdoch  \cite{gurtin1975continuum}, and our general mixed-dimensional model includes some aspects of their theory, notably the momentum balance and elastic constitutive law. 

 To illustrate this, we consider a single internal domain \( \Omega_{2}\), together with a lower-dimensional domain contained on a part of its exterior boundary \( \Omega_{4}\subset \partial Y\). The numbering is chosen so that the example can be considered as the upper domain of Figure \ref{fig:geometry}, and it is then the interpretation that the bottom part of this domain is coated. In terms of the reference configurations, the governing equations in the finite strain case are now given by equations \eqref{eq:4.12a} and \eqref{eq:4.12b} for the bulk material \( i = 2\). For the surface \( j = 4\), the surface stress is given by \eqref{eq:4.12c}, while the momentum balance simplifies to 
\begin{align}\label{eq:4.24}
\rho_{r,2}\partial_{t}^{2}u_{2}-\nabla_{\Vert} \cdot (\mathbf{F}_{4}\sigma_{4})+\sigma_{2} \cdot n &= r_{s,j} & 
\text{ for } j &= 4.
\end{align}
We recognize our surface momentum balance \eqref{eq:4.24} and our surface stress \eqref{eq:4.12c} as equation (6.1) and the first equation of Section 7 of reference \cite{gurtin1975continuum}, respectively. An important detail is the application to curved surfaces, in which case the surface divergence term becomes (omitting the subscript 4 for clarity): 
\begin{equation}
\nabla_{\Vert  } \cdot (\hat{\mathbf{F}}\sigma) = \nabla_{\Vert  } \cdot \begin{pmatrix}
\mathbf{F} &\mathcal{l}^{-1}\mathbf{n} \\ 
\end{pmatrix}\sigma  = \nabla_{\Vert  }\begin{pmatrix}
\mathbf{F} &\mathcal{l}^{-1}\nabla_{\Vert  }\mathbf{n} \\ 
\end{pmatrix} \cdot \sigma +\hat{\mathbf{F}} \cdot \nabla_{\Vert  }\sigma.
\end{equation}
For curved surfaces, this relationship expresses the fact that the gradient of the normal component of a surface enters the balance of momentum. This geometric relationship can also be expressed in terms of mean curvature, as in equation (2.9) of \cite{gurtin1975continuum}.

\section{Well-posedness of the mixed-dimensional linearized strain model}
\label{sec:well-posedness}

In this section, we will analyze the linearized strain model \eqref{Box 4.2}. While strain is infinitesimal in the bulk, the model retains two important (and non-trivial) non-linearities in the constitutive laws: First, the frictional contact law at the fractures, and secondly non-linear relationship between pressure gradients and fluid flows. This implies that the well-posedness results presented in this section are a significant generalization of the analysis presented in previous work on flow and deformation in fractured media (see e.g. the recent paper  \cite{girault2019mixed} and references therein). 

However, our analysis below excludes one important non-linear dependence: That of fracture permeability on the fracture opening. We will discuss this point in Section \ref{sec:summary}. Throughout this section, we will rely on the theory for maximally monotone binary relations, as well as the existence theory of evolutionary equations, both summarized in Appendix \ref{sec:appendix}. 

\subsection{The space of symmetric tensor fields}
\label{sec:the_space}
By Definition \ref{def:3.11}, the linearized strain \(\mathfrak{e}\) has symmetries in the sense that the representation of \( \iota_{i}\mathfrak{e}\) with \( i\in I^{n}\) is a symmetric tensor field in \(\mathbb{R}^{n\times n}\). Moreover, for all boundaries \( j\in I_{i}\cap\mathfrak{F}^{1}\), we recall that the tangential components \( \iota_{j,\parallel }\mathfrak{e}\) form a symmetric tensor in \(\mathbb{R}^{(n-1)\times (n-1)}\) whereas the normal components \( \iota_{j,\perp }\mathfrak{e}\) are zero, cf. Example \ref{eg:3.12}. These properties are captured in the definition of the following function space
\begin{equation*}
\mathfrak{G}\coloneqq \left\{ \mathfrak{e}\in L^{2}(\mathfrak{X}^{1},\mathbb{R}^{n})\left\vert
\begin{matrix}
\text{asym}(\iota_{j,\parallel }\mathfrak{e}) = 0,   & \forall j\in\mathfrak{F}^{1}, \\ 
\iota_{j,\perp }\mathfrak{e} = 0, & \forall j\in\mathfrak{F}^{1}\cap I_{i}, i\in I^{n} 
\end{matrix} \right.\right\}.
\end{equation*}

By the usual Cauchy arguments, the stress-strain relationship \(\mathfrak{A}\) retains these properties, i.e. we have that the stress \(\mathfrak{s}\) belongs to \(\mathfrak{G}\) as well.

Note that this space generalizes the space of symmetric tensors that is often used in fixed-dimensional linear elasticity. A strong tool in that setting is the adjointness between the divergence on symmetric tensors and the symmetric gradient on vector fields. We now show that this adjointness property is retained in the mixed-dimensional setting.

\begin{lemma}
\label{lemma:5.1}
The mixed-dimensional symmetric gradient and the co-symmetric gradient operators satisfy the following integration by parts formula for all $\mathfrak{s} \in \mathrm{dom}(\mathbb{D}_s \cdot) \subset \mathfrak{G}$ and $\mathfrak{u} \in \mathrm{dom}(\mathring{\mathfrak{D}}_s ) \subset L^2(\mathfrak{X}^0, \mathbb{R}^n)$:
\begin{equation*}
\left\langle\mathbb{D}_{s}\cdot \mathfrak{s}, \mathfrak{u}\right\rangle_{\mathfrak{X}^{0}}+\left\langle\mathfrak{s}, \mathring{\mathfrak{D}}_s \mathfrak{u}\right\rangle_{\mathfrak{X}^{1}} = 0.
\end{equation*}
\begin{proof}
We first observe using Definition \ref{def:4.1} and the duality from Definition \ref{def:2.23} that
\begin{equation*}
\left\langle\mathbb{D}_{s}\cdot \mathfrak{s}, \mathfrak{u}\right\rangle_{\mathfrak{X}^{0}} \coloneqq \left\langle\mathbb{D} \cdot (\boldsymbol{\mathfrak{F}}\mathfrak{s}), \mathfrak{u}\right\rangle_{\mathfrak{X}^{0}} = -\left\langle\mathfrak{s}, \boldsymbol{\mathfrak{F}}^{T}\mathring{\mathfrak{D}}\mathfrak{u}\right\rangle_{\mathfrak{X}^{1}}.
\end{equation*}
Thus it remains to show that 
$\left\langle\mathfrak{s}, \boldsymbol{\mathfrak{F}}^{T}\mathring{\mathfrak{D}}\mathfrak{u}\right\rangle_{\mathfrak{X}^{1}} = \left\langle\mathfrak{s}, \mathring{\mathfrak{D}}_s \mathfrak{u}\right\rangle_{\mathfrak{X}^{1}}$
for $\mathfrak{s} \in \mathrm{dom}(\mathbb{D}_s \cdot)$ and $\mathfrak{u} \in \mathrm{dom}(\mathring{\mathfrak{D}}_s ) $.
We do this by dimension:
\begin{enumerate} \small
    \item For \( i\in I^{n}\), the symmetry of \( \iota_{i}\mathfrak{s}\) gives us
\begin{align*}
\left\langle \iota_{i}\mathfrak{s}, \iota_{i}\boldsymbol{\mathfrak{F}}^{T}\mathring{\mathfrak{D}}(\mathfrak{u})\right\rangle_{X_{i}} &=\left\langle \iota_{i}\mathfrak{s},\underline{\mathbf{F}}_{i}^{T}D\iota_{i}\mathfrak{u}\right\rangle_{X_{i}}\\
&=\left\langle \iota_{i}\mathfrak{s},\frac{1}{2}(\underline{\mathbf{F}}_{i}^{T}D\iota_{i}\mathfrak{u}+(D\iota_{i}\mathfrak{u})^{T}\underline{\mathbf{F}}_{i})\right\rangle_{X_{i}}\\
&=\left\langle \iota_{i}\mathfrak{s}, \iota_{i}\mathring{\mathfrak{D}}_{s}\mathfrak{u}\right\rangle_{X_{i}}
\end{align*}

    \item For \( j\in I_{i}\) with \( i\in I^{n}\), the same argument as above applies for the tangential components. On the other hand, the normal components of \( \iota_{j}\mathfrak{s}\) are zero, immediately giving us
\begin{equation*}
\left\langle \iota_{j}\mathfrak{s}, \iota_{j}\boldsymbol{\mathfrak{F}}^{T}\mathring{\mathfrak{D}}(\mathfrak{u})\right\rangle_{X_{j}} =\left\langle \iota_{j}\mathfrak{s}, \iota_{j}\mathring{\mathfrak{D}}_{s}\mathfrak{u}\right\rangle_{X_{j}}
\end{equation*}

    \item For \( j\in\mathfrak{S}_{i}\) with \( i\in I^{n-1}\), we obtain 
\begin{equation*}
\iota_{j}\boldsymbol{\mathfrak{F}}^{T}\mathring{\mathfrak{D}}(\mathfrak{u}) =\hat{\underline{\mathbf{F}}}_{j}^{T}\iota_{j}\mathbbm{d}\mathfrak{u} = \iota_{j}\mathring{\mathfrak{D}}_{s}\mathfrak{u}
\end{equation*}
\end{enumerate}
Since this covers all \( j\in\mathfrak{F}^{1}\), we have \(\boldsymbol{\mathfrak{F}}^{T}\mathring{\mathfrak{D}}(\mathfrak{u}) =\mathring{\mathfrak{D}}_{s}\) on \(\mathfrak{X}^{1}\) and thus \(\left\langle\mathbb{D}_{s} \cdot \mathfrak{s}, \mathfrak{u}\right\rangle_{\mathfrak{X}^{0}} = -\left\langle\mathfrak{s}, \mathring{\mathfrak{D}}_s \mathfrak{u}\right\rangle_{\mathfrak{X}^{1}}\).
\end{proof}
\end{lemma}
We will require different treatment of the stress-strain relationships depending on whether the relationship concerns strains or strain rates. Recall that we have introduced the parameters \(\hat{\gamma }\), respectively \(\check{\gamma }\), in \eqref{eq: binary relation} to make this distinction. Using these parameters, we define the restricted identity operators as follows.
\begin{definition}
Let the \textit{restricted mixed-dimensional identity} operators \(\check{\mathfrak{T}},\hat{\mathfrak{T}}:L^{2}(\mathfrak{X}^{n}, \mathbb{R})\rightarrow\tilde{\mathfrak{S}}\) be defined as
\begin{align*}
\check{\mathfrak{T}} &\coloneqq \check{\gamma }\mathfrak{T} & 
\text{and}& &
\hat{\mathfrak{T}} &=\hat{\gamma }\mathfrak{T.}
\end{align*}
Similarly, let the \textit{restricted mixed-dimensional trace }operators be defined as \(\check{\mathfrak{T}}' \coloneqq \check{\gamma }\mathfrak{T}'\) and \(\hat{\mathfrak{T}}' \coloneqq \hat{\gamma }\mathfrak{T}'\).
\end{definition}

\subsection{Model equations for well-posedness analysis}
\label{sec:model_equations}

Let us consider the system of equations from \eqref{Box 4.2} with the goal of obtaining a system of four equations and four variables. In particular, we aim to rewrite the system in terms of fluid pressure \(\mathfrak{p}\), fluid flux \(\mathfrak{q}\), and two new variables; namely the bulk velocity \(\mathfrak{v}\), and an augmented stress \(\tilde{\mathfrak{s}}\):
\begin{align} \label{eq: defs v and s}
\mathfrak{v } &\coloneqq \partial_{t}\mathfrak{u}, &
\tilde{\mathfrak{s}} &\coloneqq\mathfrak{s}+\alpha\check{\mathfrak{T}}\mathfrak{p.}
\end{align}

Note that \(\tilde{\mathfrak{s}}\) corresponds to the mechanical stress in the fractures as this is the natural variable for which frictional contact laws are formulated. In the bulk, it equals the original, poroelastic stress \(\mathfrak{s}\) since \(\check{\mathfrak{T}}\) is zero there.

We proceed in four steps. First, substituting the definitions of \(\mathfrak{v}\) and \(\tilde{\mathfrak{s}}\) in \eqref{eq: 4.11a}, the \textbf{balance of forces }becomes 
\begin{equation}
\rho_{r}\partial_{t}\mathfrak{v}-\mathbb{D}_{s} \cdot (\tilde{\mathfrak{s}}-\alpha\check{\mathfrak{T}}\mathfrak{p}) =\mathfrak{r}_{\mathfrak{s}}.
\end{equation}

We make the following assumptions on \( \rho_{r}\):
\begin{assumption}\label{assumption:1}
\( \rho_{r}\) is a coercive, linear operator with coercivity constant \( c_{\rho }>0\).
\end{assumption}
Second, we consider the stress-strain relationships. For that, we first take the derivative of \eqref{eq: 4.11c} with respect to time and use the commutativity of \( \partial_{t}\) and \(\mathfrak{D}_{s}\) (cf. Remark \ref{remark:3.4}):
\begin{equation*}
\partial_{t}\mathfrak{e} = \partial_{t}\mathfrak{D}_{s}\mathfrak{u} = \mathfrak{D}_{s}\mathfrak{v}.
\end{equation*}

Substituting this in \eqref{eq: 4.11d} together with the definition of \(\tilde{\mathfrak{s}}\) from \eqref{eq: defs v and s}, we obtain 
\begin{equation}
(\tilde{\mathfrak{s}}+\alpha\hat{\mathfrak{T}}\mathfrak{p, }(\hat{\gamma }\partial_{t}^{-1}+\check{\gamma })\mathfrak{D}_{s}\mathfrak{v})\in \mathfrak{A}.
\end{equation}

Note that \( \partial_{t}^{-1}\) implies integration in time. In the surrounding bulk (\( i\in I^{n}\)), we assume that \(\hat{\gamma } = 1\) (and \(\check{\gamma } = 0\)) giving us a stress-strain relationship \(\hat{\mathfrak{A}}\). Conversely, we let \(\check{\gamma } = 1\) in the fractures leading to a stress-strain rate relationship describing (frictional) contact \(\check{\mathfrak{A}}\). More precisely, we define the restricted operators
\begin{align*}
\check{\mathfrak{A}} &\coloneqq \check{\gamma }\mathfrak{A}, &
\text{and} & &
\hat{\mathfrak{A}} &\coloneqq \hat{\gamma }\mathfrak{A},
\end{align*}
such that \(\mathfrak{A} = \check{\mathfrak{A}} + \hat{\mathfrak{A}}\). We now continue by making the following assumptions:
\begin{assumption}\label{assumption:2}
 \(\check{\mathfrak{A}}\) is bounded and \(\check{c}\)-maximal monotone relation for some \(\check{c}>0\), c.f. Definition \ref{def:A5}. Moreover, \( (0,0)\in\check{\mathfrak{A}}\). We emphasize that this means that for \((\mathfrak{s}_{1},\dot{\mathfrak{e}}_{1}),(\mathfrak{s}_{2},\dot{\mathfrak{e}}_{2})\in\check{\mathfrak{A}},\) we have
\begin{align*}
\left\langle \iota_{j}(\mathfrak{s}_{1}-\mathfrak{s}_{2}),\iota_{j}(\dot{\mathfrak{e}}_{1}-\dot{\mathfrak{e}}_{2})\right\rangle_{X_{j}} &\geq\check{c}\left\vert \iota_{j}(\mathfrak{s}_{1}-\mathfrak{s}_{2})\right\vert_{X_{j}}^{2},
& \forall i&\in I^{n-1}, j\in\mathfrak{F}^{1}\cap\mathfrak{S}_{i}.
\end{align*}
\end{assumption}
\begin{assumption}\label{assumption:3}
\(\hat{\mathfrak{A}}\) is a coercive linear operator with coercivity constant \(\hat{c}>0\). Thus, for \((\mathfrak{s}_{1},\mathfrak{e}_{1}),(\mathfrak{s}_{2},\mathfrak{e}_{2})\in\hat{\mathfrak{A}}\), it follows that
\begin{align*}
\left\langle \iota_{j}(\mathfrak{s}_{1}-\mathfrak{s}_{2}),\iota_{j}(\mathfrak{e}_{1}-\mathfrak{e}_{2})\right\rangle_{X_{j}}&\geq\hat{c}\left\vert \iota_{j}(\mathfrak{s}_{1}-\mathfrak{s}_{2})\right\vert_{X_{j}}^{2}, 
& \forall i&\in I^{n}, j\in\mathfrak{F}^{1}\cap\mathfrak{S}_{i}.
\end{align*}
\end{assumption}
With the assumed linearity of \(\hat{\mathfrak{A}}\), we have
\begin{equation*}
(\tilde{\mathfrak{s}}+\alpha\hat{\mathfrak{T}}\mathfrak{p,}\hat{\gamma }\partial_{t}^{-1}\mathfrak{D}_{s}\mathfrak{v})\in\hat{\mathfrak{A}}\Leftrightarrow(\tilde{\mathfrak{s}}+\alpha\hat{\mathfrak{T}}\mathfrak{p,}\hat{\gamma }\mathfrak{D}_{s}\mathfrak{v})\in \partial_{t}\hat{\mathfrak{A}}
\end{equation*}

Together with \((\tilde{\mathfrak{s}}+\alpha\hat{\mathfrak{T}}\mathfrak{p,}\check{\gamma }\mathfrak{D}_{s}\mathfrak{v})\in\check{\mathfrak{A}}\), the \textbf{stress-strain (rate) relationships} become
\begin{equation}
(\tilde{\mathfrak{s}}+\alpha\hat{\mathfrak{T}}\mathfrak{p, }\mathfrak{D}_{s}\mathfrak{v})\in\check{\mathfrak{A}}+\partial_{t}\hat{\mathfrak{A}}.
\end{equation}
\begin{remark}
More generally, we may assume \ref{assumption:2} for \(\mathfrak{A}_{i}\) if \( \iota_{i}\check{\gamma } = 1\) and \ref{assumption:3} if \( \iota_{i}\hat{\gamma } = 1\). However, for ease of presentation, we herein consider the case where these parameters are determined by dimension and thus have \ref{assumption:2} in the fractures and \ref{assumption:3} in the bulk and its surfaces.
\end{remark}
Our third equation concerns the mass balance. To capture volumetric change in terms of our four variables, we first use the decomposition induced by \(\check{\gamma }\) and \(\hat{\gamma }\) to rewrite
\begin{equation*}
\mathfrak{e} = \hat{\gamma }\mathfrak{e}+\check{\gamma }\mathfrak{e} = \hat{\mathfrak{A}}(\tilde{\mathfrak{s}}+\alpha\hat{\mathfrak{T}}\mathfrak{p})+\partial_{t}^{-1}\check{\gamma }\mathfrak{D}_{s}\mathfrak{v}
\end{equation*}

Taking the derivative of \eqref{eq: 4.11f} with respect to time and substituting this equality, we have
\begin{equation*}
\partial_{t}\mathfrak{m} = \partial_{t}\hat{\mathfrak{T}}'\alpha\hat{\mathfrak{A}}(\tilde{\mathfrak{s}}+\alpha\hat{\mathfrak{T}}\mathfrak{p})+\check{\mathfrak{T}}'\alpha\mathfrak{D}_{s}\mathfrak{v}+\partial_{t}\beta\mathfrak{p}
\end{equation*}

Inserting this in \eqref{eq: 4.11b}, the \textbf{mass balance }equation becomes 
\begin{equation}
\partial_{t}\hat{\mathfrak{T}}'\alpha\hat{\mathfrak{A}}(\tilde{\mathfrak{s}}+\alpha\hat{\mathfrak{T}}\mathfrak{p})+\check{\mathfrak{T}}'\alpha\mathfrak{D}_{s}\mathfrak{v}+\partial_{t}\beta\mathfrak{p}+\mathfrak{D \cdot q} = \mathfrak{r}_{\mathfrak{m}}.
\end{equation}

Here, we assume that

\begin{assumption}\label{assumption:4}
\( \beta\) is a coercive, linear operator with coercivity constant \( c_{\beta }>0\).
\end{assumption}
Finally, in \eqref{eq: 4.11e}, i.e. \textbf{Darcy’s law}, we neglect the dependency of the permeability on the strain \(\mathfrak{e}\):
\begin{equation*}
\mathfrak{q} = \kappa(\mathfrak{q})(\mathbb{-D}\mathfrak{p}+\mathfrak{r}_{\mathfrak{g}}).
\end{equation*}

This is equivalent to stating that \( \kappa^{-1}(\mathfrak{q})\mathfrak{q = }-\mathbb{D}\mathfrak{p}+\mathfrak{r}_{\mathfrak{g}}\). This relation only contains two variables (due to the negligence of strain dependencies) and we can consider this law, which is possibly non-linear, as a binary relation \( \kappa^{-1}\). Observing that the dependency on \(\mathfrak{q}\) is implied in the notation, we arrive at the binary relation
\begin{equation}
(\mathfrak{q, }-\mathbb{D}\mathfrak{p}+\mathfrak{r}_{\mathfrak{g}})\in \kappa^{-1}
\end{equation}

Again, we make an assumption on the relation \( \kappa^{-1}\), namely that
\begin{assumption}\label{assumption:5}
\( \kappa^{-1}\) is bounded and \( c_{\kappa }\)-maximal monotone for some \( c_{\kappa }>0\). Moreover, \((0,0)\in \kappa^{-1}\).
\end{assumption}
With these simplifications and assumptions in place, we have a system of four equations with four unknowns.
\begin{table}[!htbp]    
\centering
\begin{tabular}{|p{11.6cm}|}
\hline
{
\textbf{Governing equations for simplified mixed-dimensional poromechanics}
\begin{subequations} \label{eqs: simplified md poromechanics}
\begin{align}
&\text{Balance of forces: } & 
 \rho_{r}\partial_{t}\mathfrak{v}-\mathbb{D}_{s}\cdot\left(\tilde{\mathfrak{s}}-\alpha\check{\mathfrak{T}}\mathfrak{p}\right) =\mathfrak{r}_{\mathfrak{s}}  &\\ 
&\text{Balance of mass: } & 
\check{\mathfrak{T}}'\alpha\mathfrak{D}_{s}\mathfrak{v}+\partial_{t}\left(\hat{\mathfrak{T}}'\alpha\hat{\mathfrak{A}}\alpha\hat{\mathfrak{T}}+\beta\right)\mathfrak{p}+\partial_{t}\hat{\mathfrak{T}}'\alpha\hat{\mathfrak{A}}\tilde{\mathfrak{s}}+\mathfrak{D} \cdot \mathfrak{q} = \mathfrak{r}_{\mathfrak{m}} &\\ 
&\text{Stress-strain relations: } &  
\left(\tilde{\mathfrak{s}}+\alpha\hat{\mathfrak{T}}\mathfrak{p, }\mathfrak{D}_{s}\mathfrak{v}\right)\in\check{\mathfrak{A}}+\partial_{t}\hat{\mathfrak{A}}  &\\ 
&\text{Darcy’s law: } & 
\left(\mathfrak{q,}-\mathbb{D}\mathfrak{p}+\mathfrak{r}_{\mathfrak{g}}\right)\in \kappa^{-1} &
\end{align}
\vspace{-0.5cm}
\end{subequations}
}
\\ 
\hline
\end{tabular}
\end{table}

\subsection{Weak formulation}
\label{sec:weak_formualtion}
To accommodate analysis of system \eqref{eqs: simplified md poromechanics}, we next present the weak formulation of the poromechanics problem. The first step is to introduce the relevant function space on which to pose the problem. Following the observations from Section \ref{sec:mixed_dimensional}, we consider the following four function spaces for the variables:
\begin{subequations}
\begin{align}
\mathfrak{v} \in \mathfrak{U} &\coloneqq L^{2}(\mathfrak{X}^{0},\mathbb{R}^{n}), &
\mathfrak{p} \in \mathfrak{P} &\coloneqq L^{2}(\mathfrak{X}^{n}, \mathbb{R}), \\ 
\tilde{\mathfrak{s}}\in \mathfrak{G} &\subset L^{2}(\mathfrak{X}^{1},\mathbb{R}^{n}), &
\mathfrak{q} \in \mathfrak{Q} &\coloneqq L^{2}(\mathfrak{X}^{n-1}, \mathbb{R}).
\end{align}
\end{subequations}

Recall that \(\mathfrak{G}\), as defined in Section \ref{sec:the_space}, contains symmetry properties for stress and strain. Together, these spaces form the composite space \( U\):
\begin{equation}
U \coloneqq \mathfrak{U\times P\times G\times Q}.
\end{equation}

The space \( U\) is naturally endowed with a \( L^{2}\)-type inner product and norm, given by
\begin{equation*}
\begin{split}
\left\langle u_{1},u_{2}\right\rangle_{U} &=\left\langle\begin{bmatrix}
\mathfrak{v}_{1} & 
\mathfrak{p}_{1} & 
\tilde{\mathfrak{s}}_{1} & 
\mathfrak{q}_{1} & 
\end{bmatrix}^T,\begin{bmatrix}
\mathfrak{v}_{1} & 
\mathfrak{p}_{2} & 
\tilde{\mathfrak{s}}_{2} & 
\mathfrak{q}_{2} & 
\end{bmatrix}^T
\right\rangle_{U} \\
&  \coloneqq \left\langle\mathfrak{v}_{1},\mathfrak{v}_{2}\right\rangle_{\mathfrak{X}^{0}}+\left\langle\mathfrak{p}_{1},\mathfrak{p}_{2}\right\rangle_{\mathfrak{X}^{n}}+\left\langle\tilde{\mathfrak{s}}_{1},\tilde{\mathfrak{s}}_{2}\right\rangle_{\mathfrak{X}^{n-1}}+\left\langle\mathfrak{q}_{1},\mathfrak{q}_{2}\right\rangle_{\mathfrak{X}^{n-1}}, \\ 
\left\Vert u\right\Vert_{U} &  \coloneqq \sqrt{\left\langle u,u\right\rangle }. 
\end{split}
\end{equation*}

With the function spaces defined, we continue by considering all operators in the system as binary relations, including the linear operators. E.g., we write 
\begin{equation*}
\rho_{r}\subseteq L^{2}(\mathfrak{X}^{0},\mathbb{R}^{n})\times L^{2}(\mathfrak{X}^{0},\mathbb{R}^{n})
\end{equation*}
also in the case that \( \rho_{r}\) is simply multiplication by a constant (see Example \ref{eg:A2}). To further emphasize this, all mixed-dimensional differential operators (see Section \ref{sec:differential_operators}) are also interpreted as binary relations: 
\begin{equation*}
\begin{split}
\mathring{\mathfrak{D}}_{s} \subset L^{2}(\mathfrak{X}^{0},\mathbb{R}^{n})\times L^{2}(\mathfrak{X}^{1},\mathbb{R}^{n}),  &~~~~~~(\mathbb{D}_{s} \cdot ) \subset L^{2}(\mathfrak{X}^{1},\mathbb{R}^{n})\times L^{2}(\mathfrak{X}^{0},\mathbb{R}^{n}) \\ 
(\mathring{\mathfrak{D}} \cdot ) \subset L^{2}(\mathfrak{X}^{n-1}, \mathbb{R})\times L^{2}(\mathfrak{X}^{n}, \mathbb{R}),& ~~~~~~\mathbb{D \subset }L^{2}(\mathfrak{X}^{n}, \mathbb{R})\times L^{2}(\mathfrak{X}^{n-1}, \mathbb{R}) \\ 
\end{split}
\end{equation*}

Recall that the domains of these differential operators are proper, dense subsets of the \( L^{2}\) spaces, e.g. \( \mathrm{dom}(\mathring{\mathfrak{D}}_{s})\subset L^{2}(\mathfrak{X}^{0},\mathbb{R}^{n})\). In turn, by searching the solution in these domains, we ensure that the solution has sufficient regularity for the corresponding differentials to be well-defined. This same argument enforces the boundary conditions on the variables. We note that our choice of enforcing boundary conditions on the \(\mathring{\mathfrak{D}}\)-type differential operators implies zero (clamped) conditions on the displacement, and similarly zero normal component (no-flow) conditions on the fluid flux. The development below would be equally valid with boundary conditions imposed via \(\mathring{\mathbb{D}}\)-type operators, which would correspond to zero normal stress (floating) conditions for mechanics, and zero fluid pressure (open) conditions for flow.  

Finally, we incorporate the time dependency. Following  \cite{picard2015well}, we introduce the exponentially weighted Bochner space \( L_{\nu }^{2}(\mathbb{R},U)\) as follows.
\begin{definition}
Given \( \nu >0\), let \( L_{\nu }^{2}(\mathbb{R},U) \coloneqq \left\{ f: \mathbb{R} \rightarrow U \mid \int_{\mathbb{R}}^{}\left\Vert e^{-\nu t}f(t)\right\Vert_{U}^{2}\mathrm{d}t<\infty\right\}\). 
\end{definition}
The weight with positive \( \nu\) ensures that causality is preserved. The time derivative is then introduced as an operator acting on this weighted space.
\begin{definition}
Given \( \nu >0\), let \( \partial_{0,\nu }:\mathrm{dom}(\partial_{0,\nu })\subseteq L_{\nu }^{2}(\mathbb{R},H)\rightarrow L_{\nu }^{2}(\mathbb{R},H)\) be given by
\begin{equation*}
\partial_{0,\nu } \coloneqq e^{\nu t}(\partial_{t}+\nu)e^{-\nu t}.
\end{equation*}
\end{definition}
The motivation behind this definition can be found in Definition \ref{def:A14}. With the function space and interpretation of operators in place, we arrive at the weak formulation \eqref{eq: simplified MD problem} of the simplified hydromechanical problem \eqref{eqs: simplified md poromechanics}. 
\begin{table}[!htbp]
    \centering
    \begin{tabular}{|p{11.5cm}|}
    \hline
    \textbf{Weak formulation of the simplified mixed-dimensional poromechanics problem}
    \vspace{3mm}

    Given \( f \coloneqq \left[\mathfrak{r}_{\mathfrak{s}},\mathfrak{r}_{\mathfrak{m}}, 0,\mathfrak{r}_{\mathfrak{g}}\right]^{T}\in L_{\nu }^{2}\left(\mathbb{R},U\right)\), find \( u \coloneqq \left[\mathfrak{v, p, }\tilde{\mathfrak{s}}\mathfrak{, q}\right]^{T}\in L_{\nu }^{2}\left(\mathbb{R},U\right)\) such that
    \begin{equation} \label{eq: simplified MD problem}
    \left(u,f\right) =\left(\begin{bmatrix}
    \mathfrak{v} \\ 
    \mathfrak{p} \\ 
    \tilde{\mathfrak{s}} \\ 
    \mathfrak{q} \\ 
    \end{bmatrix},\begin{bmatrix}
    \mathfrak{r}_{\mathfrak{s}} \\ 
    \mathfrak{r}_{\mathfrak{m}} \\ 
    0 \\ 
    \mathfrak{r}_{\mathfrak{g}} \\ 
    \end{bmatrix}\right) \in \begin{bmatrix}
    \rho_{r}\partial_{0,\nu } &\mathbb{D}_{s} \cdot \alpha\check{\mathfrak{T}} & -\mathbb{D}_{s} \cdot  & 0 \\ 
    \check{\mathfrak{T}}'\alpha\mathring{\mathfrak{D}}_{s} & \partial_{0,\nu }\left(\hat{\mathfrak{T}}'\alpha\hat{\mathfrak{A}}\alpha\hat{\mathfrak{T}}+\beta\right) & \partial_{0,\nu }\hat{\mathfrak{T}}'\alpha\hat{\mathfrak{A}} &\mathring{\mathfrak{D}} \cdot  \\ 
    -\mathring{\mathfrak{D}}_{s} & \partial_{0,\nu }\hat{\mathfrak{A}}\alpha\hat{\mathfrak{T}} &\check{\mathfrak{A}}+\partial_{0,\nu }\hat{\mathfrak{A}} & 0 \\ 
    0 &\mathbb{D} & 0 & \kappa^{-1} \\ 
    \end{bmatrix}\end{equation} \\
    \hline
    \end{tabular}
\end{table}
\begin{remark}
We refer to \eqref{eq: simplified MD problem} as the \textit{weak} formulation since the problem is posed in a Hilbert space setting (the domains of the differential operators are Hilbert spaces  \cite{pedersen1989unbounded, arnold2018finite}). As a direct consequence, we recall that the solution, if it exists, is defined up to the equivalence classes of \( L_{\nu }^{2}(\mathbb{R},U)\). 
\end{remark}
\begin{remark}
 The function space \( L_{\nu }^{2}(\mathbb{R},U)\) does not ensure more regularity than square integrability in space and (weighted) time. However, the presence of the differential operators ensures that the solution, if it exists, has sufficient regularity for these to be well-defined. For example, the term \( \partial_{0,\nu }\mathfrak{v}\) ensures that \(\mathfrak{v}\in \mathrm{dom}(\partial_{0,\nu })\) and thus \( \partial_{0,\nu }\mathfrak{v}\in L_{\nu }^{2}(\mathbb{R},\mathfrak{U})\). Similarly, the solution has \(\mathfrak{v}\in \mathrm{dom}(\mathring{\mathfrak{D}}_{s})\) and thus \(\mathring{\mathfrak{D}}_{s}\mathfrak{v}\in L^{2}(\mathfrak{X}^{1},\mathbb{R}^{n})\). 
\end{remark}
\begin{remark}
 This formulation does not allow for a constant effect of gravity throughout the past since \(\mathfrak{r}_{\mathfrak{g}}\) is assumed to be in \( L_{\nu }^{2}(\mathbb{R},\mathfrak{Q})\). However, this effect can be properly incorporated by instead considering an initial-value problem on the real half-line \(\mathbb{R}_{>0}\), see  \cite{trostorff2012alternative} and Remark 3.3 in \cite{picard2015well}.
\end{remark}

\subsection{Well-posedness of the weak formulation}
\label{sec:wellposedness_weak}

In order to analyse problem \eqref{eq: simplified MD problem} in the appropriate setting, we recognize that the binary relation has a favourable underlying structure. In particular, we recognize that problem \eqref{eq: simplified MD problem} is an \textit{evolutionary equation}  \cite{picard2015well} of the form
\begin{equation*}
( u, f)\in \partial_{0,\nu }M_{0}+M_{1}+A_{\nu }.
\end{equation*}

Here, the first two components \( M_{0}\) and \( M_{1}\) are linear operators, given by
\begin{align}
M_{0} &\coloneqq  \Sigma  '\begin{bmatrix}
\rho_{r} &  &  &  \\ 
 & \beta  &  &  \\ 
 &  &\hat{\mathfrak{A}} &  \\ 
 &  &  & 0 \\ 
\end{bmatrix} \Sigma  , &
\Sigma  &\coloneqq \begin{bmatrix}
1 &  &  &  \\ 
 & 1 &  &  \\ 
 & \alpha\hat{\mathfrak{T}} & 1  &  \\ 
 &  &  & 0 \\ 
\end{bmatrix}, & 
M_{1} &\coloneqq \begin{bmatrix}
0 &  &  &  \\ 
 & 0 &  &  \\ 
 &  &\check{c} &  \\ 
 &  &  & c_{\kappa } \\ 
\end{bmatrix}.
\end{align}

Furthermore, \( A_{\nu }\subseteq L_{\nu }^{2}(\mathbb{R},U)\times L_{\nu }^{2}(\mathbb{R},U)\) is a temporal extension (see Def. \ref{def:A11}) given by
\begin{equation*}
A_{\nu } \coloneqq \left\{ u,v\in L_{\nu }^{2}(\mathbb{R},U)\mid (u(t),v(t))\in A_{0}+A_{1},  \text{ for a.e. } t\in \mathbb{R}\right\}.
\end{equation*}
Where the spatial relations \( A_{0}\) and \( A_{1}\) are given by
\begin{align}
A_{0} &\coloneqq \begin{bmatrix}
0 &  &  &  \\ 
 & 0 &  &  \\ 
 &  &\check{\mathfrak{A}}-\check{c} &  \\ 
 &  &  & \kappa^{-1}-c_{k}
\end{bmatrix}, &
A_{1} &\coloneqq \begin{bmatrix}
0 &\mathbb{D}_{s} \cdot \alpha\check{\mathfrak{T}} & -\mathbb{D}_{s} \cdot  &  \\ 
\check{\mathfrak{T}}'\alpha\mathring{\mathfrak{D}}_{s} & 0 &  &\mathring{\mathfrak{D}} \cdot  \\ 
-\mathring{\mathfrak{D}}_{s} &  & 0 &  \\ 
 &\mathbb{D} &  & 0 
\end{bmatrix}.
\end{align}

The decomposition of the evolutionary equation in terms of the binary relations \( M_{0}\), \( M_{1}\) and \( A_{0}\), \(A_{1}\) highlights the structure of the problem: \( M_{0}\) and \( M_{1}\) contain the weights of the time-derivative and diagonal terms, respectively, while \( A_{0}\) and \( A_{1}\) contain the non-linearities and differential operators, respectively. 

This identification of the problem allows us to use the solution theory of evolutionary equations, in particular we recall the following key theorem.
\begin{theorem}[Well-posedness of autonomous evolutionary inclusions]\label{theorem:5.1}
Let \( \nu >0\) and \( r>\frac{1}{2\nu }\). Let \( A_{\nu }\subseteq L_{\nu }^{2}(\mathbb{R},U)\times L_{\nu }^{2}(\mathbb{R},U)\) be a binary relation and \( M_{0},M_{1}\subseteq U\times U\) linear, bounded mappings. Assume the following hypotheses:
\begin{enumerate}[label=H\arabic*., ref=H\arabic*] \small
    \item \label{hypothesis:1}
    $A_{\nu }$ is maximal monotone, time translation-invariant (autonomous), and satisfies 
    \begin{align*}
    \int_{-\infty }^{0}\text{Re}(\left\langle u_{1}(t)-u_{2}(t),v_{1}(t)-v_{2}(t)\right\rangle)e^{-2\nu t}\mathrm{d}t &\geq 0, &
    \forall(u_{1},v_{1}),(u_{2},v_{2}) &\in A_{\nu }.
    \end{align*}
    \item \label{hypothesis:2}
    \( \exists c>0\) such that \( z^{-1}M_{0}+M_{1}-c\) is monotone for all \( z\in\mathcal{B}_{\mathbb{C}}(r,r)\). Here, \(\mathcal{B}_{\mathbb{C}}(r,r)\) denotes the open complex ball with radius \( r\), centered at \( r\).
\end{enumerate}
Then for each \( f\in L_{\nu }^{2}(\mathbb{R},U)\), there exists a unique \( u\in L_{\nu }^{2}(\mathbb{R},U)\) such that
\begin{equation*}
(u,f)\in\overline{\partial_{0,\nu }M_{0}+M_{1}+A_{\nu }}.\\ 
\end{equation*}
Moreover, the solution operator \((\overline{\partial_{0,\nu }M_{0}+M_{1}+A_{\nu }})^{-1}\) is causal and Lipschitz-continuous with a Lipschitz constant bounded by \(\frac{1}{c}\). 

\end{theorem}
\begin{proof}
See Theorem 3.2 of \cite{picard2015well}.
\end{proof}

This theorem leads to the statement of our main well-posedness result. 
\begin{theorem}[The main result]\label{theorem:5.2}
If assumptions \ref{assumption:1}-\ref{assumption:5} are fulfilled, then the mixed-dimensional poromechanics problem \eqref{eq: simplified MD problem} is well-posed. In particular, for any right-hand side \( f\in L_{\nu }^{2}(\mathbb{R},U)\), a solution \( u\in L_{\nu }^{2}(\mathbb{R},U)\) exists uniquely such that \((u,f)\in\overline{\partial_{0,\nu }M_{0}+M_{1}+A_{\nu }}\). Moreover, a \( c>0\) exists such that the solution operator is bounded:
\begin{equation*}
\left\vert u\right\vert_{L_{\nu }^{2}(\mathbb{R},U)}\leq\frac{1}{c}\left\vert f\right\vert_{L_{\nu }^{2}(\mathbb{R},U)}.
\end{equation*}
\begin{proof}
Lemmas \ref{lemma:5.3} and \ref{lemma:5.4}, presented below, suffice to invoke Theorem \ref{theorem:5.1}. The bound is a direct consequence of the Lipschitz-continuity of the solution operator.
\end{proof}
\end{theorem}

The proof of Theorem \ref{theorem:5.2} requires validating maximal monotonicity of several operators, for which we often need to take a difference between two elements \( u_{1},u_{2}\in U\). As a short-hand notation, we denote
\begin{equation*}
\delta u \coloneqq u_{1}-u_{2}
\end{equation*}
and let \(\left[\delta\mathfrak{v,}\delta\mathfrak{p,}\delta\hat{\mathfrak{s}},\delta\mathfrak{q}\right]^{T} \coloneqq \delta u\). Moreover, we omit the subscripts on inner products and norms for notational brevity.
\begin{lemma}[H1]\label{lemma:5.3}
If assumptions \ref{assumption:2},\ref{assumption:5} are satisfied, then [\ref{hypothesis:1}] is fulfilled.
\begin{proof}
First, \( A_{0}\) is maximal monotone by \ref{assumption:2} and \ref{assumption:5}. Second, \( A_{1}\) is linear and skew-selfadjoint and therefore also maximal monotone. Furthermore, \( 0\) is in the domain of both operators and \( A_{0}\) is bounded due to \ref{assumption:2},\ref{assumption:5}. We then invoke Lemma \ref{lemma:A1} from the appendix to conclude that the sum \( A_{0}+A_{1}\) is maximal monotone.

Next, we note that \( A_{\nu }\) is the temporal extension of \( A_{0}+A_{1}\) and we have that \((0,0)\in A_{0}+A_{1}\). Proposition 2.5 of \cite{trostorff2012alternative} then ensures that \( A_{\nu }\) is maximal monotone.

Time translation-invariance follows directly from the definition of \( A_{\nu }\). The positivity of the time integral is clear by the monotonicity of \( A_{0}+A_{1}\).
\end{proof}
\end{lemma}
\begin{lemma}[H2]\label{lemma:5.4}
If assumptions \ref{assumption:1}-\ref{assumption:5} are satisfied, then [\ref{hypothesis:2}] is fulfilled.
\begin{proof}
Let \((u_{1},v_{1}),(u_{2},v_{2})\in z^{-1}M_{0}+M_{1}\). The assumptions and Lemmas \ref{lemma:5.4.1} and \ref{lemma:5.4.2}, presented below, imply that
\begin{equation*}
\begin{split}
\left\langle \delta u,\delta M_{0}u\right\rangle & \geq c_{0}(\left\Vert \delta\mathfrak{v}\right\Vert^{2}+\left\Vert \delta\mathfrak{p}\right\Vert^{2}+\left\Vert\hat{\gamma }\delta\tilde{\mathfrak{s}}\right\Vert^{2}) \\ 
\left\langle \delta u,\delta M_{1}u\right\rangle & \geq c_{1}(\left\Vert \delta\mathfrak{q}\right\Vert^{2}+\left\Vert\check{\gamma }\delta\tilde{\mathfrak{s}}\right\Vert^{2}), \\ 
\end{split}
\end{equation*}
for some \( c_{0},c_{1}>0\). Using these bounds, we derive
\begin{equation*}
\begin{split}
\text{Re}(\left\langle \delta u,\delta v\right\rangle)&  = \text{Re}(z^{-1}\left\langle \delta u, \delta M_{0}u\right\rangle +\left\langle \delta u,\delta M_{1}u\right\rangle) \\ 
& \geq \text{Re}(z^{-1})c_{0}(\left\Vert \delta\mathfrak{v}\right\Vert^{2}+\left\Vert \delta\mathfrak{p}\right\Vert^{2}+\left\Vert\hat{\gamma }\delta\tilde{\mathfrak{s}}\right\Vert^{2})+c_{1}(\left\Vert \delta\mathfrak{q}\right\Vert^{2}+\left\Vert\check{\gamma }\delta\tilde{\mathfrak{s}}\right\Vert^{2}).
\end{split}
\end{equation*}

Since \( z\in\mathcal{B}_{\mathbb{C}}(r,r)\), its real part satisfies \( \text{Re}(z^{-1})\geq\frac{1}{2r}\) and thus we obtain:
\begin{equation*}
\text{Re}(\left\langle \delta u,\delta v\right\rangle)\geq\frac{c_{0}}{2r}(\left\Vert \delta\mathfrak{v}\right\Vert^{2}+\left\Vert \delta\mathfrak{p}\right\Vert^{2}+\left\Vert\hat{\gamma }\delta\tilde{\mathfrak{s}}\right\Vert^{2})+c_{1}(\left\Vert \delta\mathfrak{q}\right\Vert^{2}+\left\Vert\check{\gamma }\delta\tilde{\mathfrak{s}}\right\Vert^{2}).
\end{equation*}
We then set \( c = \min\left\{\frac{c_{0}}{2r},c_{1}\right\}\) to conclude that
\begin{equation*}
\text{Re}(\left\langle \delta u,\delta v\right\rangle)\geq c\left\Vert \delta u\right\Vert^{2}.
\end{equation*}
\end{proof}
\end{lemma}
\begin{lemma}\label{lemma:5.4.1}
Given assumptions \ref{assumption:1},\ref{assumption:3}, and \ref{assumption:4}, then the relationship \( M_{0}\) satisfies
\begin{equation*}
\left\langle \delta u,\delta M_{0}u\right\rangle \geq c_{0}(\left\Vert \delta\mathfrak{v}\right\Vert^{2}+\left\Vert \delta\mathfrak{p}\right\Vert^{2}+\left\Vert\hat{\gamma }\delta\tilde{\mathfrak{s}}\right\Vert^{2}),
\end{equation*}
for some \(c_0 \gt 0 \).
\begin{proof}
Assumptions \ref{assumption:1},\ref{assumption:3}, and \ref{assumption:4} give us that
\begin{equation*}
\left\langle \delta u,\delta M_{0}u\right\rangle \geq c_{\rho }\left\Vert \delta\mathfrak{v}\right\Vert^{2}+c_{\beta }\left\Vert \delta\mathfrak{p}\right\Vert^{2}+\hat{c}\left\Vert\hat{\gamma }(\delta\tilde{\mathfrak{s}}+\alpha\mathfrak{T}\delta\mathfrak{p})\right\Vert^{2}.
\end{equation*}

Next, we use the continuity of \( \alpha\) to obtain
\begin{equation*}
\begin{split}
c_{\beta }\left\Vert \delta\mathfrak{p}\right\Vert^{2}+\hat{c}\left\Vert\hat{\gamma }\delta(\tilde{\mathfrak{s}}+\alpha\mathfrak{T}\delta\mathfrak{p})\right\Vert^{2}& \gtrsim\frac{1}{2}\left\Vert \delta\mathfrak{p}\right\Vert^{2}+\frac{1}{2}\left\Vert -\delta \alpha\mathfrak{p}\right\Vert^{2}+\left\Vert\hat{\gamma }\delta(\tilde{\mathfrak{s}}+\alpha\mathfrak{T}\delta\mathfrak{p})\right\Vert^{2} \\ 
& \gtrsim\left\Vert \delta\mathfrak{p}\right\Vert^{2}+\left\Vert\hat{\gamma }\delta\tilde{\mathfrak{s}}\right\Vert^{2}, \\ 
\end{split}
\end{equation*}

in which \( a\gtrsim b\) implies that a \( c>0\) exists such that \( a\geq cb\). The combination of these bounds now proves the result.
\end{proof}
\end{lemma}
\begin{lemma}\label{lemma:5.4.2}
Given assumptions \ref{assumption:2},\ref{assumption:5}, then \( M_{1}\) satisfies
\begin{equation*}
\left\langle \delta u,\delta M_{1}u\right\rangle \geq c_{1}(\left\vert \delta\mathfrak{q}\right\Vert^{2}+\left\Vert\check{\gamma }\delta\tilde{\mathfrak{s}}\right\Vert^{2}),
\end{equation*}
for some \( c_{1}>0\).
\begin{proof}
Substituting the definitions and using assumptions \ref{assumption:3}, \ref{assumption:5}, we have
\begin{equation*}
\left\langle \delta u,\delta M_{1}u\right\rangle  = c_{\kappa }\left\Vert \delta\mathfrak{q}\right\Vert^{2}+\check{c}\left\Vert\check{\gamma }\delta\mathfrak{s}\right\Vert^{2}.
\end{equation*}
The result therefore follows with \( c_{1} = \min\left\{ c_{\kappa },\check{c}\right\}\).
\end{proof}
\end{lemma}
\subsection{Degeneracies: Maximal monotone contact relations with bounded inverse}
\label{sec:degeneracies}

For concrete applications, it may be desirable to relax some of assumptions \ref{assumption:1}-\ref{assumption:5}. Such relaxations will often correspond to sending a physical parameter to zero, and can therefore be considered as degenerate limits of the base model \eqref{eq: simplified MD problem} as analysed in Section \ref{sec:wellposedness_weak}. As a general expectation, these limit models need to be analysed on a case-by-case basis, as they will in principle imply that the Lipschitz constant in Theorem \ref{theorem:5.2} is not bounded. To illustrate the implications of such degenerate limits, and to show how they can be treated within the theory as presented above, we include an analysis of maximal monotone contact relations with bounded inverse.

We focus on the assumption \ref{assumption:2} regarding the frictional contact law \(\check{\mathfrak{A}}\). This may be too restrictive for some conventional friction relations (we return to this in Section \ref{sec:exemplary_models}), in particular due to the bound on the constant \(\check{c}>0\) and the assumed boundedness of the relation. In this section we consider a more general class of models, namely those concerning maximal monotone contact relations with bounded inverse. To be concrete, we adopt the following relaxation of assumption \ref{assumption:2}:
\begin{assumption}\label{assumption:6}
    \(\check{\mathfrak{A}}\mathfrak{\subseteq G\times G}\) is a maximal monotone relation with \((0,0)\in\check{\mathfrak{A}}\). Moreover, \(\check{\mathfrak{A}}^{-1}\) is bounded.
\end{assumption}
A key generalization from \ref{assumption:2} is that we now allow for the case \(\check{c} = 0\), and as such the treatment below is a consideration of a degenerate limit of the main model equations from Section \ref{sec:model_equations}. The penalty we pay for allowing this degeneracy is that we lose control over the \( L^{2}\)-norm of the components \( \iota_{j}\mathfrak{s}\) wherever \( \iota_{j}\check{\gamma } = 1\). As a result, the problem now needs to be posed in a smaller space, which we define by
\begin{equation} \label{eq:restriction}
\hat{\mathfrak{G}} \coloneqq \hat{\gamma }\mathfrak{G}.
\end{equation}
We remark that for \(\hat{\mathfrak{s}}\in\hat{\mathfrak{G}}\), we have \( \iota_{j}\hat{\mathfrak{s}} = 0\) for \( \iota_{j}\hat{\gamma } = 0\). This space is therefore isomorphic to a restriction of \(\mathfrak{G}\), but our definition \eqref{eq:restriction} avoids the need for explicitly introducing restriction and inclusion operators. The endowed norm is
\begin{align} \label{eq: norm G hat}
\left\Vert\hat{\mathfrak{s}}\right\Vert_{\hat{\mathfrak{G}}}^{2} \coloneqq \left\Vert\hat{\mathfrak{s}}\right\Vert_{\mathfrak{X}^{1}}^{2} &= \sum_{
\substack{
j\in\mathfrak{F}^{1}, \\ 
\iota_{j}\hat{\gamma } = 1}} \left\Vert \iota_{j}\hat{\mathfrak{s}}\right\Vert_{X_{j}}^{2}, &
\forall\hat{\mathfrak{s}} &\in\hat{\mathfrak{G}}.
\end{align}

Importantly, \eqref{eq: norm G hat} forms a natural choice for this setting since it does not contain any terms from the fractures, i.e. on \( X_{i}\) for \( i\in I^{n-1}\) (and its descendants). The stress variable can be decomposed in a similar manner as the stress-strain relationship, and we therefore define 
\begin{equation*}
\mathfrak{s} = \hat{\mathfrak{s}}+\check{\mathfrak{s}} \coloneqq \hat{\gamma }\mathfrak{s}+\check{\gamma }\mathfrak{s}
\end{equation*}

We proceed by eliminating \(\check{\mathfrak{s}}\) from the system. For that, we note that since \(\check{\gamma }\mathfrak{D}_{s}\mathfrak{v} = \mathbbm{d}\mathfrak{v}\), the stress-strain relationship \((\check{\mathfrak{s}},\check{\gamma }\mathfrak{D}_{s}\mathfrak{v})\in\check{\mathfrak{A}}\) can be restated as
\begin{equation*}
(\check{\mathfrak{s}},\mathbbm{d}\mathfrak{v})\in\check{\mathfrak{A}}\Leftrightarrow(\mathbbm{d}\mathfrak{v,}\check{\mathfrak{s}})\in\check{\mathfrak{A}}^{-1}\Leftrightarrow(\mathfrak{v,}\check{\mathfrak{s}})\in\check{\mathfrak{A}}^{-1}\mathbbm{d}
\end{equation*}

The momentum balance equation, i.e. the first row of \eqref{eq: simplified MD problem}, is rewritten using this substitution as
\begin{equation*}
\left(
\begin{bmatrix}
\mathfrak{v} \\ 
\mathfrak{p} \\ 
\hat{\mathfrak{s}} \\ 
\mathfrak{q} \\ 
\end{bmatrix},\mathfrak{r}_{\mathfrak{s}}\right)\in\begin{bmatrix}
\rho_{r}\partial_{0,\nu }+\mathbbm{d}'\check{\mathfrak{A}}^{-1}\mathbbm{d} 
&\mathbb{D}_{s} \cdot \alpha\check{\mathfrak{T}} 
& -\mathbb{D}_{s} \cdot  
& 0 
\end{bmatrix}
\end{equation*}
with \(\mathbbm{d'}\) the adjoint of the jump operator \(\mathbbm{d}\). On the other hand, the stress-strain relationship in the bulk is now given by
\begin{equation*}
(\hat{\mathfrak{s}}+\alpha\hat{\mathfrak{T}}\mathfrak{p, }\hat{\gamma }\mathfrak{D}_{s}\mathfrak{v})\in \partial_{t}\hat{\mathfrak{A}}.
\end{equation*}

Since the stress component \(\check{\mathfrak{s}}\) does not influence the remaining equations, we are ready to pose the new problem \eqref{eq: degenerate problem} in the composite space
\begin{equation*}
\hat{U}\mathfrak{ \coloneqq U\times P\times }\hat{\mathfrak{G}}\mathfrak{\times Q,}
\end{equation*}
endowed with the norm \(\left\Vert  \cdot \right\Vert_{\hat{U}} =\left\Vert  \cdot \right\Vert_{U}\).
\begin{table}[!htbp]
    \centering
    \begin{tabular}{|p{11.5cm}|}
    \hline
    \textbf{Weak formulation of the mixed-dimensional poromechanics problem using maximal
monotone contact relations with bounded inverse}
\vspace{3mm}

Given \(\left[\mathfrak{r}_{\mathfrak{s}},\mathfrak{r}_{\mathfrak{m}}, 0,\mathfrak{r}_{\mathfrak{g}}\right]^{T}\in L_{\nu }^{2}\left(\mathbb{R},\hat{U}\right)\), find \(\left[\mathfrak{v, p, }\hat{\mathfrak{s}}\mathfrak{, q}\right]^{T}\in L_{\nu }^{2}\left(\mathbb{R},\hat{U}\right)\) such that

\begin{equation} \label{eq: degenerate problem}
\begin{bmatrix}
\mathfrak{v} \\ 
\mathfrak{p} \\ 
\hat{\mathfrak{s}} \\ 
\mathfrak{q} \\ 
\end{bmatrix},\begin{bmatrix}
\mathfrak{r}_{\mathfrak{s}} \\ 
\mathfrak{r}_{\mathfrak{m}} \\ 
0 \\ 
\mathfrak{r}_{\mathfrak{g}} \\ 
\end{bmatrix}\in\begin{bmatrix}
\rho_{r}\partial_{0,\nu }+\mathbb{d}'\check{\mathfrak{A}}^{-1}\mathbb{d} &\mathbb{D}_{s} \cdot \alpha\check{\mathfrak{T}} & -\mathbb{D}_{s} \cdot  & 0 \\ 
\check{\mathfrak{T}}'\alpha\mathring{\mathfrak{D}}_{s} & \partial_{0,\nu }\left(\hat{\mathfrak{T}}'\alpha\hat{\mathfrak{A}}\alpha\hat{\mathfrak{T}}+\beta\right) & \partial_{0,\nu }\hat{\mathfrak{T}}'\alpha\hat{\mathfrak{A}} &\mathring{\mathfrak{D}} \cdot  \\ 
-\hat{\gamma }\mathring{\mathfrak{D}}_{s} & \partial_{0,\nu }\hat{\mathfrak{A}}\alpha\hat{\mathfrak{T}} & \partial_{0,\nu }\hat{\mathfrak{A}} & 0 \\ 
0 &\mathbb{D} & 0 & \kappa^{-1} \\ 
\end{bmatrix}
\end{equation}\\
    \hline
    \end{tabular}
\end{table}

Again, we recognize the structure of problem \eqref{eq: degenerate problem} as an evolutionary equation and we note that the decomposition \( \partial_{0,\nu }M_{0}+M_{1}+A_{\nu }\) now holds with \( M_{0}\) unchanged and \( M_{1}\) given by 
\begin{equation} \label{eq: new M1}
M_{1} \coloneqq \begin{bmatrix}
0 &  &  &  \\ 
 & 0 &  &  \\ 
 &  & 0 &  \\ 
 &  &  & c_{\kappa } \\ 
\end{bmatrix}.
\end{equation}

On the other hand, \( A_{\nu }\) is here the temporal extension of \( A_{0}+A_{1}\) with 
\begin{equation*}
A_{0} \coloneqq \begin{bmatrix}
\mathbbm{d}'\check{\mathfrak{A}}^{-1}\mathbbm{d} &  &  &  \\ 
 & 0 &  &  \\ 
 &  & 0 &  \\ 
 &  &  & \kappa^{-1}-c_{k} \\ 
\end{bmatrix},  A_{1} \coloneqq \begin{bmatrix}
0 &\mathbb{D}_{s} \cdot \alpha\check{\mathfrak{T}} & -\mathbb{D}_{s} \cdot  &  \\ 
\check{\mathfrak{T}}'\alpha\mathfrak{D}_{s} & 0 &  &\mathfrak{D} \cdot  \\ 
-\hat{\gamma }\mathfrak{D}_{s} &  & 0 &  \\ 
 &\mathbb{D} &  & 0 \\ 
\end{bmatrix}.
\end{equation*}
\begin{theorem}\label{theorem:5.5}
If assumptions \ref{assumption:1} and \ref{assumption:3}-\ref{assumption:6} are fulfilled, then the mixed-dimensional poromechanics problem with maximal monotone contact relations \eqref{eq: degenerate problem} is well-posed.
\begin{proof}
Lemmas \ref{lemmma:5.6} and \ref{lemma:5.7}, presented below, show that hypotheses [\ref{hypothesis:1}]-\ref{hypothesis:2}] are fulfilled for \eqref{eq: degenerate problem}. Theorem \ref{theorem:5.1} then provides the result.
\end{proof}
\end{theorem}
\begin{lemma}[H1]\label{lemmma:5.6}
If assumptions \ref{assumption:5},\ref{assumption:6} are satisfied, then [\ref{hypothesis:1}] is fulfilled.
\begin{proof}
The fact that \( A_{0}\) is maximal monotone and bounded follows from \ref{assumption:5},\ref{assumption:6}. The skew-self adjointness of \( A_{1}\) is verified by
\begin{equation*}
\left\langle -\hat{\gamma }\mathfrak{D}_{s}\mathfrak{v,}\hat{\mathfrak{s}}\right\rangle +\left\langle\mathfrak{v,}-\mathbb{D}_{s} \cdot \hat{\mathfrak{s}}\right\rangle  =\left\langle -\mathfrak{D}_{s}\mathfrak{v,}\hat{\mathfrak{s}}\right\rangle +\left\langle\mathfrak{D}_{s}\mathfrak{v,}\hat{\mathfrak{s}}\right\rangle  = 0.
\end{equation*}

Now, the arguments from Lemma \ref{lemma:5.3} provide the result.
\end{proof}
\end{lemma}
\begin{lemma}[H2]\label{lemma:5.7} If assumptions \ref{assumption:1} and \ref{assumption:3}-\ref{assumption:5} are satisfied, then [\ref{hypothesis:2}] is fulfilled.
\begin{proof}
Since \( M_{0}\) has remained unchanged, Lemma \ref{lemma:5.4.1} gives us the bound
\begin{equation*}
\left\langle \delta u,\delta M_{0}u\right\rangle \geq c_{0}(\left\Vert \delta\mathfrak{v}\right\Vert^{2}+\left\Vert \delta\mathfrak{p}\right\Vert^{2}+\left\Vert \delta\hat{\mathfrak{s}}\right\Vert^{2})
\end{equation*}
for some \( c_{0}>0.\) Continuing with \( M_{1}\), its definition \eqref{eq: new M1} directly gives us
\begin{equation*}
\left\langle \delta u,\delta M_{1}u\right\rangle  = c_{\kappa }\left\Vert \delta\mathfrak{q}\right\Vert^{2}.
\end{equation*}
Using these bounds, the arguments from Lemma \ref{lemma:5.4} are followed to conclude the proof.
\end{proof}
\end{lemma}
\subsection{Exemplary models}
\label{sec:exemplary_models}

We finalize this section by describing two example models. The two models differ on whether assumption \ref{assumption:2} or \ref{assumption:6} is satisfied. Since the only difference lies in the frictional contact law \(\check{\mathfrak{C}}\), we first present the other components of the model.
\begin{example}\label{eg:5.1}
\phantom{}
\begin{itemize} \small
    \item Let the bulk density be given by a scalar \( c_{\rho }>0\). Then the corresponding binary relation, defined by multiplication with this scalar, is given by
\begin{equation*}
\rho_{r} \coloneqq \left\{(\mathfrak{v}_{1},\mathfrak{v}_{2})\in L_{\nu }(\mathbb{R},\mathfrak{U})\times L_{\nu }(\mathbb{R},\mathfrak{U})\mid \mathfrak{v}_{2} = c_{\rho }\mathfrak{v}_{1}\right\} \\ 
\end{equation*}
and satisfies \ref{assumption:1}.
    \item Let the bulk medium and its boundaries be isotropic materials. Then the stress-strain relationship \(\hat{\mathfrak{A}}\) can be described, using the Lamé parameters \( \mu ,\lambda\), as
\begin{equation*}
\hat{\mathfrak{A}} \coloneqq \left\{(\mathfrak{s,e})\in L_{\nu }(\mathbb{R},\mathfrak{G})\times L_{\nu }(\mathbb{R},\mathfrak{G})\mid \iota_{j}\mathfrak{s} = 2\mu \iota_{j}\mathfrak{e}+\lambda(I:\iota_{j}\mathfrak{e}_{\parallel})I,   \forall j\in\mathfrak{S}_{i}, i\in I^{n}\right\} \\ 
\end{equation*}
Here, \( I\) is the identity tensor in \(\mathbb{R}^{d_j}\). It is easy to see that \(\hat{\mathfrak{A}}\) satisfies \ref{assumption:3} with \(\hat{c} =(2\mu +n\lambda)^{-1}\).

    \item Similar to \ref{assumption:1}, we set \( \beta\) as multiplication with \( c_{\beta }>0\): \\ \begin{equation*}
\beta  \coloneqq \left\{(\mathfrak{p}_{1},\mathfrak{p}_{2})\in L_{\nu }(\mathbb{R},\mathfrak{P})\times L_{\nu }(\mathbb{R},\mathfrak{P})\mid \mathfrak{p}_{2} = c_{\beta }\mathfrak{p}_{1}\right\} \\ 
\end{equation*}
which satisfies \ref{assumption:4}.

    \item Finally, the constitutive law relating flux and pressure is assumed to be given by Darcy(-Forchheimer) flow.
\begin{equation*}
\kappa^{-1} \coloneqq \left\{(\mathfrak{q,}-\mathbb{D}\mathfrak{p})\in L_{\nu }(\mathbb{R},\mathfrak{Q})\times L_{\nu }(\mathbb{R},\mathfrak{Q})\mid \mathbb{-D}\mathfrak{p} = (\kappa_{1}+\kappa_{2}\left\vert\mathfrak{q}\right\vert)\mathfrak{q}\right\}.
\end{equation*}
With the material parameters \( \kappa_{1}\),\( \kappa_{2}>0\), we have \( c_{\kappa }>0\) in \ref{assumption:5}.
\end{itemize}
\end{example}

Next, we assume that the frictional contact law \(\check{\mathfrak{A}}\) is defined as the direct sum of disjoint relations:
\begin{equation}
\check{\mathfrak{A}} \coloneqq \bigoplus_{
\substack{
i\in I^{n-1} \\ 
j\in\mathfrak{S}_{i}\cap\mathfrak{F}^{1}}} 
A_{j,\parallel }\oplus A_{j,\perp }
\end{equation}
Here, the subscript \( \parallel\) indicates the tangential friction law and \( \perp\) the perpendicular contact. It is clear that if each \( A_{j.\parallel }\) and \( A_{j.\perp }\) satisfies \ref{assumption:2}, respectively \ref{assumption:6}, then so does \(\check{\mathfrak{A}}\). We dedicate the following two subsections to the description of two exemplary models that fulfill these respective assumptions. 

\subsubsection{Maximal monotone contact relations with bounded inverse}
\label{sec:maximal_monotone}

We start by considering relations that satisfy \ref{assumption:6}, since this assumption is easier to fulfil in practice. For the stress-strain relation in the parallel direction, let us consider Tresca friction. For a given threshold \( \tau >0\), this law is given by
\begin{equation} \label{eq: Tresca}
A_{j,\parallel} \coloneqq \left\{(\sigma ,\dot{\varepsilon })\in \iota_{j,\parallel }\mathfrak{G\times }\iota_{j,\parallel }\mathfrak{G}\left\vert  
\begin{array}{c}
\tau\dot{\varepsilon } =\left\vert\dot{\varepsilon }\right\vert \sigma \text{  or  }\dot{\varepsilon } = 0 \\ 
\left\vert \sigma\right\vert \leq \tau 
\end{array} \text{ a.e. on } X_{j} \right.\right\}
\end{equation}
\begin{lemma}\label{lemma:5.8}
Tresca friction \( A_{j,\parallel }\) of \eqref{eq: Tresca} satisfies \ref{assumption:6}, i.e. is maximal monotone, has \((0,0)\in A_{j,\parallel },\) and has bounded inverse.
\begin{proof}
The fact that \((0,0)\in A_{j,\parallel }\) is immediate and the boundedness of the inverse (cf. Definition \ref{def:A8}) follows from the inequality \(\left\vert \sigma\right\vert \leq \tau\). It remains to show maximal monotonicity. Let \((\sigma_{1},\dot{\varepsilon }_{1}),(\sigma_{2},\dot{\varepsilon }_{2})\in A_{j,\parallel }\). We consider the inner product of the differences \(\left\langle \delta \sigma ,\delta\dot{\varepsilon }\right\rangle_{X_{j}}\) and distinguish three cases:
\begin{enumerate}  \small
    \item If \( \tau\dot{\varepsilon }_{k} =\left\vert\dot{\varepsilon }_{k}\right\vert \sigma_{k}\) for both \( k\), then

\begin{align*}
\left\langle \delta \sigma ,\delta\dot{\varepsilon }\right\rangle_{X_{j}}&  = \tau\left\langle\frac{\dot{\varepsilon }_{1}}{\left\vert\dot{\varepsilon }_{1}\right\vert }-\frac{\dot{\varepsilon }_{2}}{\left\vert\dot{\varepsilon }_{2}\right\vert },\dot{\varepsilon }_{1}-\dot{\varepsilon }_{2}\right\rangle_{X_{j}} \\ 
& \geq \tau \int_{X_{j}}^{}(\left\vert\dot{\varepsilon }_{1}\right\vert +\left\vert\dot{\varepsilon }_{2}\right\vert -\left\vert\dot{\varepsilon }_{1}\right\vert -\left\vert\dot{\varepsilon }_{2}\right\vert) = 0.
\end{align*}
    \item If \(\dot{\varepsilon }_{k} = 0\) for both \( k\), then \( \delta\dot{\varepsilon } = 0\).
    \item If \( \tau\dot{\varepsilon }_{1} =\left\vert\dot{\varepsilon }_{1}\right\vert \sigma_{1}\) and \(\dot{\varepsilon }_{2} = 0\), then     
\begin{equation*}
\left\langle \delta \sigma ,\delta\dot{\varepsilon }\right\rangle_{X_{j}} =\left\langle \tau\frac{\dot{\varepsilon }_{1}}{\left\vert\dot{\varepsilon }_{1}\right\vert }-\sigma_{2},\dot{\varepsilon }_{1}\right\rangle_{X_{j}}\geq \int_{X_{j}}^{}(\tau -\left\vert \sigma_{2}\right\vert)\left\vert\dot{\varepsilon }_{1}\right\vert \geq 0.
\end{equation*}
\end{enumerate}

Hence, \( A_{j,\parallel }\) is monotone. Finally, we note that \( I+A_{j,\parallel }\) is surjective by the following arguments:
\begin{enumerate}  \small
    \item If \(\left\vert\dot{\varepsilon }\right\vert <\tau\), then \((\sigma ,\dot{\varepsilon })\in I+A_{j,\parallel }\) for \( \sigma  =\dot{\varepsilon }\).

    \item If \(\left\vert\dot{\varepsilon }\right\vert \geq \tau\), then \((\sigma ,\dot{\varepsilon })\in I+A_{j,\parallel }\) for \( \sigma  = \tau\frac{\dot{\varepsilon }}{\left\vert\dot{\varepsilon }\right\vert }\).

\end{enumerate}
In turn, Theorem 1.6 from \cite{trostorff2011well} ensures that \( A_{j,\parallel }\) is maximal monotone.
\end{proof}
\end{lemma}
For the (perpendicular) contact law, we use the relation for rough surfaces described in Example \ref{eg:4.4}. Setting \( C_{j}^{1} = 0\), \( C_{j}^{3}\geq 0\), and \( C_{j}^{4}\geq 1\), this law is given by
\begin{equation} \label{eq: def Aperp}
A_{j,\perp } \coloneqq \left\{(\sigma ,\dot{\varepsilon })\in \iota_{j,\perp }\mathfrak{G\times }\iota_{j,\perp }\mathfrak{G}\mid \sigma  = -C_{j}^{3}(-\dot{\varepsilon })_{+}^{C_{j}^{4}}\right\}
\end{equation}
\begin{lemma}\label{lemma:5.9}
 The contact law \( A_{j,\perp }\) of \eqref{eq: def Aperp} satisfies \ref{assumption:6}.
\begin{proof}
It is easy to verify that \((0,0)\in A_{j,\perp }\). Boundedness of the inverse relation follows from the fact that \(\left\vert \sigma\right\vert \leq\left\vert\dot{\varepsilon }\right\vert^{C_{j}^{4}}\). Secondly, it is easy to see that the function \( f(x) = -(-x)_{+}^{C_{j}^{4}}\) is monotone. Since \( A_{j,\perp }\) is the inverse of the graph of \( f\), it is a monotone relation as well. Finally, it is clear that no monotone extension of \( A_{j,\perp }\) exists and hence it is maximal.
\end{proof}
\end{lemma}
We finalize this subsection by stating the well-posedness, which is a direct result of Theorem \ref{theorem:5.5}.
\begin{theorem}\label{theorem:5.10}
 Problem \eqref{eq: degenerate problem}, in which the relations are given by Example \ref{eg:5.1}, \eqref{eq: Tresca}, and \eqref{eq: def Aperp} is well-posed in the space \( L_{\nu }^{2}(\mathbb{R},\hat{U})\).
\end{theorem}

\subsubsection{Bounded, c-maximal monotone contact relations}
\label{sec:bounded_c}

On the other hand, our base model \eqref{eq: simplified MD problem} contains assumption \ref{assumption:2} which requires that the relation itself (instead of its inverse) is bounded. This is not the case for Tresca friction introduced in \eqref{eq: Tresca}. Thus, in order to obtain a model that satisfies \ref{assumption:2}, we perform two regularizations. First, we introduce a maximal strain rate \( c_{\infty }>0\) such that the regularized friction relation becomes
\begin{equation}
A_{j,\parallel }^{reg} \coloneqq \left\{(\sigma ,\dot{\varepsilon })\in \iota_{j,\parallel }\mathfrak{G\times }\iota_{j,\parallel }\mathfrak{G}\left\vert  
\begin{array}{c}
\tau\dot{\varepsilon } =\left\vert\dot{\varepsilon }\right\vert \sigma \text{  or  }\left\vert \sigma\right\vert\dot{\varepsilon } = c_{\infty }1_{\left\{\left\vert \sigma\right\vert >\tau\right\} }\sigma  \\ 
\left\vert\dot{\varepsilon }\right\vert \leq c_{\infty } \\ 
\end{array}\right. \text{a.e. on } X_{j}\right\}
\end{equation}
with \( 1_{\left\{\left\vert \sigma\right\vert >\tau\right\} }\) the indicator function of the set \( \left\{\left\vert \sigma\right\vert >\tau\right\}  \). We emphasize that \( c_{\infty }\) can be chosen sufficiently large to ensure that this bound is not reached in physical applications. Secondly, we add a constant \(\check{c}\) to the regularized law in order to ensure \( c\)-maximal monotonicity.
\begin{lemma}\label{lemma:5.11}
The friction law \( (A_{j,\parallel }^{reg}+\check{c})\) satisfies \ref{assumption:2} for any \(\check{c}>0\). I.e. it is bounded, \(\check{c}\)-maximal monotone, and \((0,0)\in A_{j,\parallel }^{reg}+\check{c}\).
\begin{proof}
The fact that \((0,0)\in A_{j,\parallel }^{reg}+\check{c}\) is clear. Moreover, the boundedness of the post-set is guaranteed by the inequality \(\left\vert\dot{\varepsilon }\right\vert \leq c_{\infty }.\) It remains to show that \( A_{j,\parallel }^{reg}\) is maximal monotone. Let \((\sigma_{1},\dot{\varepsilon }_{1}),(\sigma_{2},\dot{\varepsilon }_{2})\in A_{j,\parallel }^{reg}\). We distinguish three cases:
\begin{enumerate}  \small
    \item If \( \tau\dot{\varepsilon }_{k} =\left\vert\dot{\varepsilon }_{k}\right\vert \sigma_{k}\) for both \( k\), then we use the same arguments as the first case in Lemma \ref{lemma:5.8}.

    \item If \(\left\vert \sigma_{k}\right\vert\dot{\varepsilon }_{k} = c_{\infty }1_{\left\{\left\vert \sigma_{k}\right\vert >\tau\right\} }\sigma_{k}\) for both \( k\), then
    
\begin{equation*}
\begin{split}
\left\langle \delta \sigma ,\delta\dot{\varepsilon }\right\rangle_{X_{j}}&  = c_{\infty }\left\langle 1_{\left\{\left\vert \sigma_{1}\right\vert >\tau\right\} }\frac{\sigma_{1}}{\left\vert \sigma_{1}\right\vert }-1_{\left\{\left\vert \sigma_{2}\right\vert >\tau\right\} }\frac{\sigma_{2}}{\left\vert \sigma_{2}\right\vert },\sigma_{1}-\sigma_{2}\right\rangle_{X_{j}} \\ 
& \geq c_{\infty }\int_{X_{j}}^{}(1_{\left\{\left\vert \sigma_{1}\right\vert >\tau\right\} }(\left\vert \sigma_{1}\right\vert -\left\vert \sigma_{2}\right\vert)+1_{\left\{\left\vert \sigma_{2}\right\vert >\tau\right\} }(\left\vert \sigma_{2}\right\vert -\left\vert \sigma_{1}\right\vert))\geq 0 \\ 
\end{split}
\end{equation*}

    \item If \( \tau\dot{\varepsilon }_{1} =\left\vert\dot{\varepsilon }_{1}\right\vert \sigma_{1}\) and \(\left\vert \sigma_{2}\right\vert\dot{\varepsilon }_{2} = c_{\infty }1_{\left\{\left\vert \sigma_{2}\right\vert >\tau\right\} }\sigma_{2}\), then \(\left\vert \sigma_{1}\right\vert  = \tau\) and so we obtain 
\begin{equation*}
\begin{split}
\left\langle \delta \sigma ,\delta\dot{\varepsilon }\right\rangle_{X_{j}}&  =\left\langle \sigma_{1}-\sigma_{2},\left\vert\dot{\varepsilon }_{1}\right\vert\frac{\sigma_{1}}{\left\vert \sigma_{1}\right\vert } -c_{\infty }1_{\left\{\left\vert \sigma_{2}\right\vert >\tau\right\} }\frac{\sigma_{2}}{\left\vert \sigma_{2}\right\vert }\right\rangle_{X_{j}} \\ 
& \geq \int_{X_{j}}^{}(\left\vert\dot{\varepsilon }_{1}\right\vert(\left\vert \sigma_{1}\right\vert -\left\vert \sigma_{2}\right\vert)+c_{\infty }1_{\left\{\left\vert \sigma_{2}\right\vert >\tau\right\} }(\left\vert \sigma_{2}\right\vert -\left\vert \sigma_{1}\right\vert)) \\ 
&  = \int_{X_{j}}^{}(\left\vert\dot{\varepsilon }_{1}\right\vert(\tau -\left\vert \sigma_{2}\right\vert)+c_{\infty }1_{\left\{\left\vert \sigma_{2}\right\vert >\tau\right\} }(\left\vert \sigma_{2}\right\vert -\tau)) \\ 
\end{split}
\end{equation*}
The case \(\left\vert \sigma_{2}\right\vert \leq \tau\) now follows directly. For the other case, we have 
\begin{equation*}
\left\vert\dot{\varepsilon }_{1}\right\vert(\tau -\left\vert \sigma_{2}\right\vert)+c_{\infty }(\left\vert \sigma_{2}\right\vert -\tau) =(c_{\infty }-\left\vert\dot{\varepsilon }_{1}\right\vert)(\left\vert \sigma_{2}\right\vert -\tau)\geq 0.
\end{equation*}
\end{enumerate}
Hence, \( A_{j,\parallel }^{reg}\)is monotone. Finally, maximality can now be shown by noting that \( 1+A_{j,\parallel }^{reg}\) is surjective and invoking Theorem 1.6 from \cite{trostorff2011well}.
\end{proof}
\end{lemma}

For the contact law, we encounter the same issue: The relation \eqref{eq: def Aperp} is not a bounded relation because the post-set of \( \sigma  = 0\) contains all \(\dot{\varepsilon }>0\). We therefore use the same regularization as in the friction law and define a \( c_{\infty }>0\) that is set outside of the relevant limits of the physical model. This leads us to the following relation
\begin{equation}
A_{j,\perp }^{reg} \coloneqq \left\{(\sigma ,\dot{\varepsilon })\in \iota_{j,\perp }\mathfrak{G\times }\iota_{j,\perp }\mathfrak{G}\left\vert 
\begin{array}{c}
(\sigma +C_{j}^{3}(-\dot{\varepsilon })_{+}^{C_{j}^{4}})(\varepsilon -c_{\infty })  = 0 \\ 
\dot{\varepsilon } \leq c_{\infty } \\ 
\sigma \geq -C_{j}^{3}(-\dot{\varepsilon })_{+}^{C_{j}^{4}} \\ 
\end{array},\right. \text{ a.e. on } X_{j}\right\}
\end{equation}
\begin{lemma}\label{lemma:5.12}
The contact law \((A_{j,\perp }^{reg}+\hat{c})\) satisfies \ref{assumption:2} for any \(\hat{c}>0\).
\end{lemma} 
\begin{proof}
Boundedness and the fact that \((0,0)\in A_{j,\perp }^{reg}+\check{c}\) are easy to verify. We show monotonicity of \( A_{j,\perp }^{reg}\) by considering three cases:
\begin{enumerate} \small
    \item If \(\dot{\varepsilon }_{k} = c_{\infty }\) for both \( k = 1,2\) then we have \( \delta\dot{\varepsilon } = 0\).

    \item If \(\sigma_{k} = -C_{j}^{3}(-\dot{\varepsilon }_{k})_{+}^{C_{j}^{4}}\) for both \( k = 1,2\), then we use the fact that \( f(x) = -(-x)_{+}^{c}\) is a monotone function and thus  
\begin{equation*}
\left\langle \delta \sigma ,\delta\dot{\varepsilon }\right\rangle_{X_{j}} = C_{j}^{3}\left\langle -(-\dot{\varepsilon }_{1})_{+}^{C_{j}^{4}}+(-\dot{\varepsilon }_{2})_{+}^{C_{j}^{4}},\dot{\varepsilon }_{1}-\dot{\varepsilon }_{2}\right\rangle_{X_{j}}\geq 0.
\end{equation*}

    \item If \( \sigma_{1} = -C_{j}^{3}(-\dot{\varepsilon }_{1})_{+}^{C_{j}^{4}}\) and \(\dot{\varepsilon }_{2} = c_{\infty }\) then \( \sigma_{1}\leq 0\) and \( \sigma_{2}\geq 0\) and so
    
\begin{equation*}
\left\langle \delta \sigma ,\delta\dot{\varepsilon }\right\rangle_{X_{j}} =\left\langle \sigma_{1}-\sigma_{2},\dot{\varepsilon }_{1}-c_{\infty }\right\rangle_{X_{j}}\geq 0.
\end{equation*}

\end{enumerate}
This shows that \( A_{j,\perp }^{reg}\) is monotone. No monotone extension of this law exists and hence it is maximal. In turn, \((A_{j,\perp }^{reg}+\hat{c})\) is $\hat{c}$-maximal monotone. 
\end{proof}

\begin{remark}\label{remark:5.9}
Applying the bound \( c_{\infty }\) is equivalent to applying a cut-off operator on the strain rate similar to  \cite{brun2020monolithic, sun2005discontinuous}. We note that, in practice, one can always check a posteriori whether the solution has attained the bound at any moment in time. If not, then the solution is independent of this bound, as desired. If the solution attains the bound, however, then the value of \( c_{\infty }\) can be increased. If no bounded \( c_{\infty }\) exists, then the physical model assumptions will at some point be violated, since this would indicate an arbitrarily large bulk velocity.
\end{remark}

The well-posedness of the resulting problem is now a direct consequence of Theorem \ref{theorem:5.2}, which we present formally in the following theorem.

\begin{theorem}\label{theorem:5.11}
Problem \eqref{eq: simplified MD problem}, in which the relations are given by Example \ref{eg:5.1} and those analyzed in Lemmas \ref{lemma:5.11} and \ref{lemma:5.12}, is well-posed in the space \( L_{\nu }^{2}(\mathbb{R},U)\).
\end{theorem}

\section{Summary and final remarks}
\label{sec:summary}

In this manuscript, we have provided a general mixed-dimensional finite strain model for poromechanics in the presence of fracture (Section \ref{sec:governing_eq_mixed_dimensional_finite}), its simplification in the case of linearized strain (Section \ref{sec:governing_eq_linearized}), and a well-posedness theory allowing in the setting of linearized stain, still allowing for a generality in terms of the constitutive laws (Section \ref{sec:wellposedness_weak}). These developments are to the best of our knowledge all new. Our finite strain theory is rotationally invariant, and our mixed-dimensional model have several well-established models as its simplifications (Section \ref{sec:connection_to_classical}).  

As presented, the model allows for a range of physical phenomena, some of which may be desirable to neglect in various concrete applications. One such example arises in friction, where Tresca friction does not conform to a parameterization of the full model, but instead is in a different class of maximally monotone relations where there is no positive lower bound. This barely violates assumption \ref{assumption:2}, and is thus a degenerate limit of our model framework (using our general theory, as stated in Theorem \ref{theorem:5.2}, leads to a continuity constant that is unbounded). On the other hand, Tresca friction (and other models satisfying assumption \ref{assumption:6}) can nevertheless be allowed by a small perturbation of the spaces considered, as illustrated in Section \ref{sec:degeneracies}.

An example of a different kind of degeneracy, which is not analyzed herein, is related to the presence of mechanical strain and stress terms at solid surfaces. This allows for modeling of surface effects, as may appear due to mineral processes in the subsurface, or through surface coating in industrial applications. On the other hand, such surface effects may be desirable to neglect for ``clean" fractures. We consider this also as a degenerate limit of our model, but in this case it is assumption \ref{assumption:3} that is violated (and possibly also \ref{assumption:1}). 

Another important point which is not covered by our well-posedness analysis is the dependency of fracture permeability on aperture. In terms of the model structure, this enters in the sense of the permeability depending on the mixed-dimensional strain, as stated in \eqref{eq:Darcy}. Such dependencies have recently been analyzed in the fixed-dimensional case \cite{bociu2016analysis}, and we are optimistic that their approach can be extended to the mixed-dimensional setting. 

We have previously shown how numerical methods can be derived for the simpler problem of flow in porous media (see references in Section \ref{sec:fluid_flow_rigid}). Development of numerical methods for the current problem is ongoing, and we look forward to reporting on this in future work.  

\section*{Acknowledgements}
The authors wish to thank Jakub Both, Omar Duran, Eirik Keilegavlen and Ivar Stefansson for many helpful discussions and comments on the manuscript. WMB was supported by the Dahlquist Research Fellowship, funded by Comsol AB. The work of JMN took place in the context of NFR project 250223 and the "Akademia" grant at the UoB titled "FracFlow" (funded by Equinor ASA).

\bibliography{sn-bibliography}
\begin{appendices}
\section{Evolutionary equations and monotonicity}
\label{sec:appendix}

This work employs the theoretical framework from \cite{picard2015well} in order to present the model and its analysis. This setting is more general than the conventional approach in which mappings in Sobolev spaces are used, and therewith provides us three key advantages. 
\begin{itemize}
    \item First, the domains of differential operators are derived from the operator, rather than vice versa. This relieves the need for characterizing the solution space and the theory is formulated in \( L^{2}\)-type spaces instead. 

    \item Second, it allows us to include constitutive laws in our model that are not, strictly speaking, mappings between Sobolev spaces. In fact, we can consider the larger class of maximal monotone relations. 

    \item Third, the theory of \textit{evolutionary equations} naturally incorporates the time derivatives in a continuous setting. This provides an existence result for the entire time domain, rather than requiring arguments based on discrete time stepping. 
\end{itemize}
In order to provide accessible reference, we recall the key concepts from this framework in this appendix. Let \( H\) be a Hilbert space endowed with inner product \(\left\langle  \cdot , \cdot \right\rangle\).
\begin{definition}\label{def:A1}
A \textit{binary relation} between \( H\) and \( H\) is a subset \( A\subseteq H\times H\).
\end{definition}

This set-theoretical perspective on binary relations allows us to speak of \textit{closed} relations (if the set \( A\) is closed) and of the \textit{closure} of a relation, which we denote by \(\overline{A}\). 

Binary relations have an algebraic structure in the sense that for \( A,B\subseteq H\times H\) and \( \lambda\in \mathbb{R}\), we have
\begin{equation}
A+\lambda B =\left\{(x,y)\in H\times H\left\vert 
\begin{array}{c}
\exists y_{A}\in H \text{ with } (x,y_{A})\in A, \\ 
\exists y_{B}\in H \text{ with } (x,y_{B})\in B, \\ 
\end{array} \right.  \text{ such that } y = y_{A}+\lambda y_{B}\right\} 
\end{equation}
\begin{example}\label{eg:A1}
 Given a mapping \( f:H\rightarrow H\), then its graph \( \mathrm{Gr}(f)\) given by all pairs \((x,f(x))\) with \( x\in H\) is a binary relation. The algebraic structure of binary relations naturally generalizes that of mappings in the sense that
\begin{equation*}
\mathrm{Gr}(f)+\lambda \mathrm{Gr}(g) = \mathrm{Gr}(f+\lambda g).
\end{equation*}

We do not distinguish between a mapping and its graph and instead reuse the notation \( f\), i.e. write \((x,y)\in f\subseteq H\times H\) when referring to the corresponding binary relation.
\end{example}
\begin{example}\label{eg:A2}
For a constant \( c\in \mathbb{R}\), we have \((x,y)\in c\) if \( y = cx\). The binary relation \( c\subseteq H\times H\) therefore corresponds to the multiplication by \( c\). Importantly, this multiplicative structure implies that \( (x,\partial_t y) \in \partial_t c\).
\end{example}
\begin{example}\label{eg:A3}

We adopt a matrix/vector notation for binary relations between tuples of variables. Thus, let \( H = H_{1}\times H_{2}\) and let $ u \coloneqq \begin{bmatrix}
u_{1} \\ u_{2}
\end{bmatrix} \in H $ and 
$v \coloneqq \begin{bmatrix}
 v_{1} \\ v_{2}
\end{bmatrix} \in H$. Moreover, let \( A_{ij}\subseteq H_{j}\times H_{i}\) for \( i,j\in\left\{ 1,2\right\}\). We denote
\begin{equation}
(u,v) =\left(\left[
\begin{array}{c}
u_{1} \\ 
u_{2} \\ 
\end{array}\right],\left[
\begin{array}{c}
v_{1} \\ 
v_{2} \\ 
\end{array}\right]\right)\in\begin{bmatrix}
A_{11} & A_{12} \\ 
A_{21} & A_{22} \\ 
\end{bmatrix}
\end{equation}
if and only if \( v_{i} = \sum_{j}^{}v_{ij}\) with \((u_{j},v_{ij})\in A_{ij}\) for each \( i\) and \( j\).

\end{example}
\begin{definition}\label{def:A2}
 The \textit{domain} and \textit{range} of a binary relation \( A\subseteq H\times H\) are denoted by \(\mathrm{dom}(A)\) and \( \text{ran}(A)\), respectively, and are given by
\begin{equation}
\begin{split}
\mathrm{dom}(A)&  \coloneqq \left\{ x\in H\mid \exists y\in H \text{ s.t. }(x,y)\in A\right\}  \\ 
\text{ran}(A)&  \coloneqq \left\{ y\in H\mid \exists x\in H \text{ s.t. }(x,y)\in A\right\}  \\ 
\end{split}
\end{equation}
\end{definition}
As is common in a functional analysis setting, \( \text{ran}(A)\) may be a proper subset of \( H\), i.e. the range of an operator is allowed to be smaller than its codomain. In this setting, we typically also have that \(\mathrm{dom}(A)\) is a proper subset of \( H\). Thus, it is important to remember that both the domain and range of \( A\) can be proper subsets of \( H\), despite \( A\) being defined as \( A\subseteq H\times H\).
\begin{definition}\label{def:A3}
A binary relation\textit{ }\( A\subseteq H\times H\) is \textit{monotone} if for all \((x_{1},y_{1}),(x_{2},y_{2})\in A\) it holds that
\begin{equation*}
\text{Re}(\left\langle x_{1}-x_{2},y_{1}-y_{2}\right\rangle)\geq 0.
\end{equation*}
\end{definition}
\begin{definition}\label{def:A4}
A binary relation\textit{ }\( A\subseteq H\times H\) is \textit{maximal monotone} if it is monotone and if for all \( B\subseteq H\times H\) with \( B\) monotone and \( A\subseteq B\), it follows that \( A = B\).
\end{definition}
\begin{definition}\label{def:A5}
 A binary relation\textit{ }\( A\subseteq H\times H\) is \( c\)-\textit{maximal monotone} for some \( c>0\) if \( A-c\) is maximal monotone. 
\end{definition}
The following example illustrates the generality that these definitions allow for.
\begin{example}\label{eg:A4}
Let \( H\mathbb{ \coloneqq R}\) and let the binary relation \( A\mathbb{\subset R\times R}\) be given by 
\begin{equation*}
A \coloneqq \left\{(x,y)\in \mathbb{R\times R}\left\vert 
\begin{array}{c}
y = 0 \text{ if } \left\vert x\right\vert <1   \\ 
xy\geq 0 \text{ if } \left\vert x\right\vert  = 1 \\ 
\end{array}\right.\right\} 
\end{equation*}

Graphically, \( A\) is given by the horizontal line segment \((x,0)\) for \( -1\leq x\leq 1\) and the vertical half-lines \((\text{sgn}(y), y)\). In turn, \(\mathrm{dom}(A) =\left[-1,1\right]\), which we emphasize is a proper subset of \(\mathbb{R}\), and \( \text{ran}(A) = \mathbb{R}\). The fact that \( A\) is monotone can be verified by a straightforward computation:
\begin{itemize} \small
    \item \(\left\vert x_{1}\right\vert ,\left\vert x_{2}\right\vert <1.\) Then \( y_{1} = y_{2} = 0\) and so \(\left\langle x_{1}-x_{2},y_{1}-y_{2}\right\rangle  = 0.\)

    \item \(\left\vert x_{1}\right\vert <1 =\left\vert x_{2}\right\vert .\)Then \( y_{1} = 0\) and \(\left\langle x_{1}-x_{2},y_{1}-y_{2}\right\rangle  =(x_{2}-x_{1})y_{2}\geq 0\) since the two terms have the same sign.

    \item \(\left\vert x_{1}\right\vert  =\left\vert x_{2}\right\vert  = 1\). Then \((x_{1}-x_{2})\) is either zero or has the same sign as \((y_{1}-y_{2})\), giving the result.

\end{itemize}
Moreover, \( A\) is maximal monotone since there exists no monotone extension \( B\mathbb{\subseteq R\times R}\) with \( A\subseteq B\) except for \( B = A\).
\end{example}
\begin{definition}\label{def:A6}
The \textit{adjoint} of a binary relation \( A\subseteq H_{1}\times H_{2}\) is given by
\begin{equation*}
A' \coloneqq \left\{(u,v)\in H_{2}\times H_{1}\mid \left\langle u,y\right\rangle_{H_{2}} =\left\langle v,x\right\rangle_{H_{1}},  \forall(x,y)\in A\right\} 
\end{equation*}
\end{definition}
\begin{definition}\label{def:A7}
The \textit{inverse} of a binary relation \( A\subseteq H\times H\) is given by
\begin{equation*}
A^{-1} \coloneqq \left\{(x,y)\in H\times H\mid (y,x)\in A\right\} 
\end{equation*}
\end{definition}
Importantly, this definition implies that \textit{any} binary relation has an inverse relation. Moreover, it is easy to see that the inverse of a maximal monotone relation is itself maximal monotone as well (by the symmetry of the inner product).
\begin{example}\label{eg:A5}
Continuing with example \ref{eg:A4}, we have 
\begin{equation*}
A^{-1} =\left\{(x,y)\in \mathbb{R\times R}\left\vert 
\begin{array}{c}
y = \text{sgn}(x) \text{ if } x\neq 0   \\ 
y\in\left[-1,1\right] \text{ if } x = 0 \\ 
\end{array}\right.\right\} .
\end{equation*}

As is apparent here, \( A^{-1}\) can be described as a \textit{set-valued} function. However, for the reasons mentioned at the beginning of the section, we herein prefer the use of binary relations.
\end{example}

Even though the inverse of a relation is always well-defined, we often require more properties of the inverse, in particular that it corresponds to a Lipschitz continuous mapping. This can be obtained by Minty’s theorem, which lies at the heart of the analysis of evolutionary equations.
\begin{theorem}[Minty’s theorem]\label{th:A1} Let \( A\subseteq H\times H\) be a \( c\)-maximal monotone relation for some \( c>0\). Then the inverse relation \( A^{-1}\) defines a Lipschitz continuous mapping with domain \(\mathrm{dom}(A^{-1}) = H\) and a Lipschitz constant bounded by \(\frac{1}{c}\).
\begin{proof}
See Theorem 1.1 of \cite{picard2015well}.
\end{proof}
\end{theorem}
\begin{example}\label{eg:A6}
Continuing with example \ref{eg:A5}, let \( B = A+\epsilon\) for some \( \epsilon >0.\) We then compute 
\begin{equation*}
B^{-1} =\left\{(x,y)\in \mathbb{R\times R}\left\vert 
\begin{array}{c}
y = \text{sgn}(x) \text{ if }\left\vert x\right\vert \geq \epsilon    \\ 
y =\frac{1}{\epsilon }x \text{ if }\left\vert x\right\vert <\epsilon  \\ 
\end{array}\right.\right\} .
\end{equation*}

Note that, by construction, \( B-c = A\) is maximal monotone with \( c = \epsilon >0\). In turn, Minty’s theorem implies that \( B^{-1}\) is Lipschitz continuous with possible Lipschitz constant \(\frac{1}{\epsilon }\) and this can easily be verified. On the other hand, we see that \( A^{-1}\) is not Lipschitz continuous at the origin. In turn, the converse of Minty’s theorem implies that there is no \( c>0\) such that \( A\) is \( c\)-maximal monotone.
\end{example}
Minty’s theorem forms a powerful tool and it is therefore important to determine whether a binary relation is maximal monotone. It becomes a natural question to ask whether the sum of two maximal monotone relations is maximal monotone as well. To answer this, we first define bounded relations.
\begin{definition}\label{def:A8}
A relation \( A\subseteq H\times H\) is called \textit{bounded} if for all bounded sets \( M\subseteq H\), the post-set given by \(\left\{ y\in H\mid \exists x\in M \text{ s.t. } (x,y)\in A\right\}\) is bounded as well.
\end{definition}

We remark that if \( A\) corresponds to a mapping, the post-set is typically referred to as the \textit{image} of \( M\) under \( A\). Using the boundedness property, we have the following result concerning the sum of maximal monotone relations.
\begin{lemma}\label{lemma:A1}
Let \( A,B\subset H\times H\) be maximal monotone, let \( A\) be bounded and let \(\mathrm{dom}(A)\cap\mathrm{dom}B)\neq \varnothing\). Then the sum \( A+B\) is maximal monotone.
\begin{proof}
See Proposition 1.22 of \cite{trostorff2011well}.
\end{proof}
\end{lemma}

On the other hand, we will frequently encounter unbounded operators as well. An important class of these is given by densely defined, unbounded linear operators.
\begin{definition}\label{def:A9}
An operator \( A\subseteq H\times H\) is \textit{densely defined} if \(\mathrm{dom}(A)\) is dense in \( H\).
\end{definition}

Differential operators are a primary example of unbounded operators in our work, defined as mappings between \( L^{2}\)-type spaces. We emphasize this using the
\begin{example}\label{eg:A7}
Let \( \Omega \subset\mathbb{R}^{n}\) be a Lipschitz domain with \( n\in\left\{ 2,3\right\}\). Let \(\tilde{\nabla }_{0}\) be the gradient operator defined on \( \mathring{C}^{\infty }(\Omega)\), i.e. on the space of infinitely differentiable functions with vanishing trace on \( \partial \Omega\): 
\begin{equation*}
\tilde{\nabla }_{0}:\mathring{C}^{\infty }(\Omega)\subset L^{2}(\Omega ,\mathbb{R})\rightarrow L^{2}(\Omega ,\mathbb{R}^{n})
\end{equation*}

Now, let the divergence \((\nabla  \cdot )\) and its dual be given by
\begin{equation*}
(\nabla  \cdot ) \coloneqq -(\tilde{\nabla }_{0})'  \text{ and }  \nabla_{0} \coloneqq -(\nabla  \cdot )'.
\end{equation*}

Let \( H\) be the pair of spaces
\begin{equation*}
H \coloneqq L^{2}(\Omega, \mathbb{R})\times L^{2}(\Omega ,\mathbb{R}^{n}).
\end{equation*}

Finally, let the relation \( A\subseteq H\times H\) be given by
\begin{equation*}
A \coloneqq \begin{bmatrix}
0 & \nabla  \cdot  \\ 
\nabla_{0} & 0 \\ 
\end{bmatrix}.
\end{equation*}

We emphasize that the domain of the operator \( A\) follows from its definition. In fact, \(\mathrm{dom}(A)\) is given by all functions \((f, g)\in H\) with \( f\in\mathrm{dom}(\nabla_{0})\subseteq L^{2}(\Omega, \mathbb{R})\) and \( g\in\mathrm{dom}(\nabla  \cdot )\subseteq L^{2}(\Omega ,\mathbb{R}^{n})\). For these functions, it follows by definition that \( \nabla_{0}f\in L^{2}(\mathbb{R}^{n})\) and \( \nabla  \cdot g\in L^{2}(\Omega, \mathbb{R})\). In turn, we have the following characterization in more conventional notation:
\begin{equation*}
\mathrm{dom}(A) = \mathrm{dom}(\nabla_{0})\times \mathrm{dom}(\nabla  \cdot ) = \mathring{H}^{1}(\Omega)\times H(\nabla  \cdot ,\Omega).
\end{equation*}

Finally, we note that \( \mathrm{dom}(A)\) is dense in \( H\) and thus \( A\) is a densely defined, unbounded linear operator.
\end{example}
\begin{remark}\label{remark:A1}
Analogously, we may start with \((\tilde{\nabla }_{0} \cdot )\) as the divergence acting on vector-valued, infinitely differentiable functions with vanishing normal trace. By taking the appropriate adjoints, this leads us to the operators \( \nabla\) and \( (\nabla_{0} \cdot )\) with
\begin{align}
\mathrm{dom}(\nabla) &= H^{1}(\Omega),  &
\mathrm{dom}(\nabla_{0} \cdot ) &= \mathring{H}(\nabla\cdot,\Omega).
\end{align}
Here, \( \mathring{H}(\nabla  \cdot ,\Omega)\) denotes the subspace of \( H(\nabla  \cdot ,\Omega)\) consisting of functions with zero normal trace on \( \partial \Omega\). 
\end{remark}
\begin{figure}
    \centering
    \includegraphics[width=\textwidth]{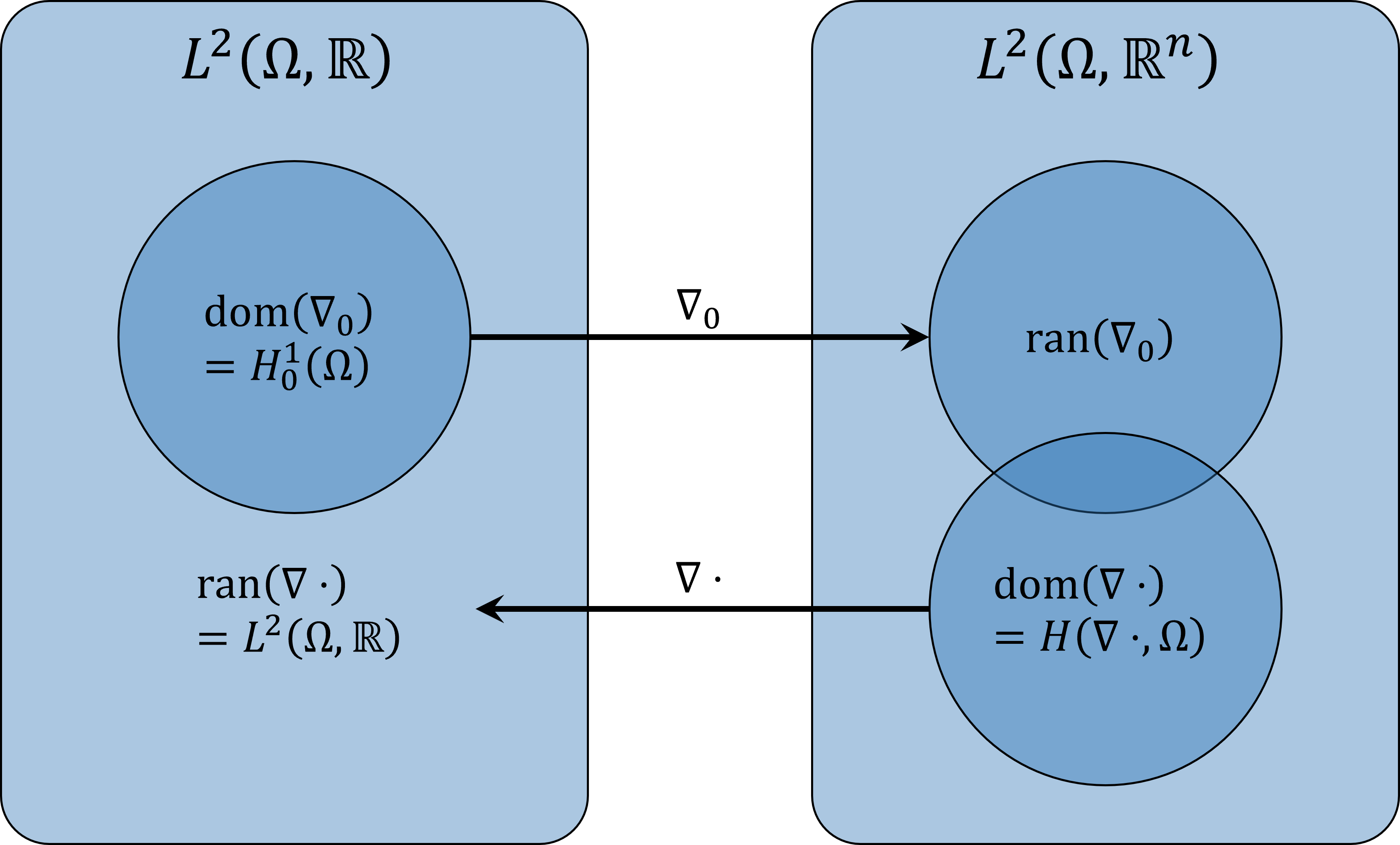}
    \caption{The gradient and divergence mappings defined as densely defined, unbounded linear operators. Note that both the domain and range of \( \nabla_{0}\) is a proper subset of \( L^{2}\). The fact that the range of \( \nabla  \cdot \) equals \( L^{2}\) is classical.}
    \label{fig:the-gradient-and-divergence-mappings}
\end{figure}
\begin{example}\label{eg:A8}
Continuing with example \ref{eg:A7}, we note that \( A\) is linear and skew-selfadjoint, and therefore maximal monotone.  In particular, for \([f_{1}, g_{1}]^T, [f_{2}, g_{2}]^T \in \mathrm{dom}(A)\), we have
\begin{equation}
\begin{split}
\left\langle\left[
\begin{array}{c}
f_{1}-f_{2} \\ 
g_{1}-g_{2} \\ 
\end{array}\right],A\left[
\begin{array}{c}
f_{1} \\ 
g_{1} \\ 
\end{array}\right]-A\left[
\begin{array}{c}
f_{1} \\ 
g_{1} \\ 
\end{array}\right]\right\rangle &  =\left\langle f_{1}-f_{2},\nabla  \cdot (g_{1}-g_{2})\right\rangle +\left\langle \nabla_{0}(f_{1}-f_{2}),g_{1}-g_{2}\right\rangle  \\ 
&  =\left\langle f_{1}-f_{2},\nabla  \cdot (g_{1}-g_{2})\right\rangle -\left\langle f_{1}-f_{2},\nabla  \cdot (g_{1}-g_{2})\right\rangle  \\ 
&  = 0 \\ 
\end{split}
\end{equation}
Here, we have used angled brackets for the inner products of both the product space \( H\) and its components.
\end{example} 

In order to define an evolutionary equation, we next consider the spatiotemporal setting. Following \cite{picard2015well}, we introduce an exponentially weighted Bochner-type function space. 
\begin{definition}\label{def:A10}
For $\nu > 0$, let
\begin{equation*}
L_{\nu }^{2}(\mathbb{R},H) \coloneqq \left\{ f\mathbb{:R\rightarrow }H\mid \int_{\mathbb{R}}^{}\left\vert e^{-\nu t}f(t)\right\vert_{H}^{2}\mathrm{d}t<\infty\right\} 
\end{equation*}
\end{definition}
\begin{definition}\label{def:A11}
Given a binary relation \( A\subseteq H\times H\), its \textit{temporal extension} \( A_{\nu }\subseteq L_{\nu }^{2}(\mathbb{R},H)\times L_{\nu }^{2}(\mathbb{R},H)\) is given by 
\begin{equation*}
A_{\nu } \coloneqq \left\{ u,v\in L_{\nu }^{2}(\mathbb{R},H)\mid (u(t),v(t))\in A,  \text{for} a.e. t\in \mathbb{R}\right\} 
\end{equation*}
\end{definition}
\begin{definition}\label{def:A12}
 An operator \( A:\mathrm{dom}(A)\subseteq L_{\nu }^{2}(\mathbb{R},H)\rightarrow L_{\nu }^{2}(\mathbb{R},H)\) is \textit{time translation-invariant} if for each \((u,v)\in A\) and \( h\in\mathbb{R}\), we have \((u( \cdot +h),v( \cdot +h))\in A.\)
\end{definition}
\begin{definition}\label{def:A13}
A closed mapping\textit{ }\( A:\mathrm{dom}(A)\subseteq L_{\nu }^{2}(\mathbb{R},H)\rightarrow L_{\nu }^{2}(\mathbb{R},H)\) is called \textit{causal} if for all \( h\in\mathbb{R}\) and \( f,g\in \mathrm{dom}(A)\) with \( f\left. \right\vert_{t\leq h} = g\left. \right\vert_{t\leq h}\), it follows that \((Af)\left. \right\vert_{t\leq h} =(Ag)\left. \right\vert_{t\leq h}\).
\end{definition}
Finally, we present the time derivative on the weighted space \( L_{\nu }^{2}(\mathbb{R},H)\). Its definition is motivated by the following short calculation for differentiable \( f\mathbb{:R\rightarrow R}\).
\begin{equation*}
e^{\nu t}\partial_{t}(e^{-\nu t}f) = e^{\nu t}(-\nu e^{-\nu t}f+e^{-\nu t}\partial_{t}f) =(-v+\partial_{t})f
\end{equation*}

In turn, we have \( e^{\nu t}(\partial_{t}+\nu)e^{-\nu t}f = \partial_{t}f\), which suggests the following definition.
\begin{definition}\label{def:A14}
Given \( \nu >0\), let \( \partial_{0,\nu }:\mathrm{dom}(\partial_{0,\nu })\subseteq L_{\nu }^{2}(\mathbb{R},H)\rightarrow L_{\nu }^{2}(\mathbb{R},H)\) be given by
\begin{equation*}
\partial_{0,\nu } \coloneqq e^{\nu t}(\partial_{t}+\nu)e^{-\nu t}.
\end{equation*}
\end{definition}
\end{appendices}
\end{document}